\documentclass[a4paper,leqno,12pt]{amsart}
\usepackage{epsfig}
\usepackage{amsmath}
 \usepackage[pdftex]{hyperref}
\usepackage{amsfonts}
\usepackage{amsmath,amscd}
\usepackage{amssymb}
\usepackage{amsthm}
\usepackage{mathrsfs}
\usepackage[mathscr]{euscript}
\usepackage{graphicx}
\usepackage{xypic}

\newtheorem{definition}{Definition}[section]
\newtheorem{theorem}[definition]{Theorem}
\newtheorem{remark}[definition]{Remark}
\newtheorem{final Remarks}[definition]{Final Remarks}
\newtheorem{lemma}[definition]{Lemma}
\newtheorem{question}[definition]{Question}
\newtheorem{problem}[definition]{Problem}
\newtheorem{conjecture}[definition]{Conjecture}
\newtheorem{addendum}[definition]{Addendum}
\newtheorem{proposition}[definition]{Proposition}
\newtheorem{example}[definition]{Example}
\newtheorem{corollary}[definition]{Corollary}

\numberwithin{equation}{section}
\begin{document}
\title{Smooth and PL-Rigidity Problems on Locally Symmetric Spaces}
\vspace{2cm}
\author{Ramesh Kasilingam}

\email{rameshkasilingam.iitb@gmail.com  ; mathsramesh1984@gmail.com  }
\address{Statistics and Mathematics Unit,
Indian Statistical Institute,
Bangalore Centre, Bangalore - 560059, Karnataka, India.}
\date{}
\maketitle
\begin{abstract}
This is a survey on known results and open problems about Smooth and PL-Rigidity Problem for negatively curved locally symmetric spaces. We also review some developments about studying the basic topological properties of the space of negatively curved Riemannian metrics and the Teichmuller space of negatively curved metrics on a manifold. 
\end{abstract}

\section{\large Introduction}
A fundamental problem in geometry and topology is the following :
\begin{problem}\label{genprorig}
Let $f : N\to M$ denote a homotopy equivalence between two closed Riemannian manifolds.
Is $f$ homotopic to a diffeomorphism, a PL-homeomorphism or a homeomorphism?
\end{problem}
In an earlier article \cite{Ram15}, we discussed Problem \ref{genprorig} for topological rigidity. In this paper, we review the status of the above problem \ref{genprorig} for closed locally symmetric spaces of noncompact type, its recent developments and many related interesting open question. We also discuss the constructions of counterexamples to smooth and PL-rigidity Problem \ref{genprorig} for negatively curved locally symmetric spaces  given by \cite{FJ89a, FJ93a, FJ94, Ont94, FJO98, FJO98a, Oku02, AF03, AF04, Ram14}.\\
\indent\\
Here is an outline of the material:\\
In section 2, we give the notation and state the basic definitions and results that will be used throughout the paper.\\
In section 3, we review smooth and PL-rigidity problem \ref{genprorig} for closed real hyperbolic manifolds. We also discuss the constructions of examples given by \cite{FJ89a, Ont94}, which provide counterexamples to the smooth (topological) Lawson-Yau Conjecture and the smooth analogue of Borel's Conjecture. Lawson and Yau conjecture states that if $M$ and $N$ are closed negatively curved such that $\pi_1(M)\cong \pi_1(N)$, then $M$ is CAT(Diff, PL or Top)-isomorphic to $N$. Roughly speaking the manifold $N$ in the above counterexample will be $M\#\Sigma$, where $\Sigma$ is an exotic sphere and $M$ is a stably parallelizable, real hyperbolic manifold (with $m\geq 5$) and having sufficiently large injectivity radius \cite{FJ89a}. Such manifold $M$ exists due to Sullivan, building on joint work with Pierre Deligne \cite{Sul79}.\\
In section 4, we review smooth rigidity problem \ref{genprorig} for finite volume but non-compact real hyperbolic manifolds and closed complex hyperbolic manifolds. We also discuss examples of negatively curved Riemannian manifolds, which are homeomorphic but not diffeomorphic to finite volume real \cite{FJ89a}, or closed complex \cite{FJ94, Ram14} hyperbolic manifolds. A possible way to change smooth structure on a smooth manifold $M^n$, without changing its homeomorphism type, is to take its connected sum $M^n\#\Sigma^n$ with a homotopy sphere $\Sigma^n$. A surprising observation reported in \cite{FJ93a} is that connected sums can never change the smooth structure on a connected, noncompact smooth manifold. Farrell and Jones used a different method in \cite{FJ93a}, which can sometimes change the smooth structure on a noncompact manifold $M^n$. The method is to remove an embedded tube $\mathbb{S}^1\times \mathbb{D}^{n-1}$ from $M^n$ and then reinsert it with a "twist". The problem \ref{genprorig} for closed complex hyperbolic 
manifolds is considered in \cite{FJ94} and was solved by using connected sums method. Here we give the outline of these constructions.\\
In section 5, we discuss series of examples of exotic smooth structures on compact locally symmetric spaces of noncompact type given by \cite{Oku02}. The examples are obtained by taking the connected sum with an exotic sphere. To detect the change of the smooth structure Boris Okun \cite{Oku02} used a tangential map from the locally symmetric space to its dual compact type symmetric space. This method was subsequently used by C.S. Aravinda and F.T. Farrell \cite{AF03, AF04} in their construction of exotic smooth structures on certain quaternionic and Cayley hyperbolic manifolds supporting metrics of strict negative curvature. Here we discuss how we can look at the problem of detecting when $M^n\#\Sigma^n$ and $M^n$ are not diffeomorphic, where $M$ is a closed locally symmetric space of noncompact type is essentially reduced to look at the problem of detecting exotic structure on the dual symmetric space $M_u$ of $M$. We also discuss how to recover exotic smoothings of Farrell-Jones-Aravinda \cite{FJ89a, FJ94, AF03, AF04} from 
Okun results \cite{Oku02}.\\
Let $\mathcal{MET}^{sec<0}(M)$ is the space of all Riemannian metrics on $M^n$ that have all sectional curvatures less that 0 and $\mathcal{T}(M)$ and $ \mathcal{T}^{\epsilon}(M)$ are the Teichmuller space of metric and $\epsilon$-pinched negatively curved metrics on $M$ respectively. $\mathcal{M}^{sec<0}(M)$ is the moduli space of negatively curved metrics on $M$.\\
In section 6, we want to study the space of negatively curved metrics and geometries. We also discuss the following problems \cite{FO09, FO10}. 
\begin{itemize}
\item Is the space $\mathcal{MET}^{sec<0}(M)$ path connected ?
\item Is (each path component of ) the space $\mathcal{MET}^{sec<0}(M)$ path connected ?
\item Is $\pi_k(\mathcal{MET}^{sec<0})$ trivial?
\item Is the inclusion $\mathcal{T}^{\epsilon}(M)\mapsto \mathcal{T}(M)$ null homotopic ?
\item Is the map $\pi_k(\mathcal{MET}^{sec<0}(M))\longrightarrow \pi_k(\mathcal{M}^{sec<0}(M))$ non-zero?
\item Is the map $\widetilde{H}_k(\mathcal{MET}^{sec<0}(M))\longrightarrow \widetilde{H}_k(\mathcal{M}^{sec<0}(M))$ non-zero?
\end{itemize}
In section 7, we review many interesting open problems along this direction.
\section{\large Basic Definitions and results}
In this section, we review some basic definitions, results and notation to be used throughout the article.\\
We write $\rm{Diff}$ for the category of smooth manifolds, $\rm{PL}$ for the category of piecewise-linear manifolds, and $\rm{Top}$ for the category of topological manifolds. We generically write $\rm{CAT}$ for any one of these geometric categories.\\
Let $\mathbb{K}=\mathbb{R}$, $\mathbb{C}$, $\mathbb{H}$ or $\mathbb{O}$ denote the real, complex, quaternion and Cayley numbers, viz., four division algebras $\mathbb{K}$ over the real numbers whose dimensions over $\mathbb{R}$ are $d = 1$, $2$, $4$ and $8$. For every prime $p$, the ring of integers modulo $p$ is a finite field of order $p$ and is denoted $\mathbb{Z}_p$. Let $I=[0, 1]$ be the closed interval in $\mathbb{R}$. $\mathbb{R}^n$ is $n$-dimensional Euclidean space, $\mathbb{D}^n$ is the unit disk, $\mathbb{S}^n$ is the unit sphere, $\Sigma_g$ is the closed orientable surface of genus $g$ and $T^n=\mathbb{S}^1\times \mathbb{S}^1\times....\times \mathbb{S}^1$ ($n$-factors) is $n$-dimensional torus with its natural smooth structures and orientation. Let $\textbf{H}^n= \{(x_1,x_2,...,x_n)\in \mathbb{R}^n : x_1\geq 0\}$, $\mathbb{C}^n$ be the $n$-dimensional complex space, and $\mathbb{H}^n$ be the  $n$-dimensional quaternionic space. Let $\rm{Aut}(G)$ denotes the group of automorphisms of group $G$ and 
$\rm{Out}(G)$ be
the group of outer automorphisms of group $G$. Let $\rm{Top}(X)$  be the group of all self-homeomorphisms of a topological space $X$, $\rm{Isom}(M)$ be the group of isometries of a Riemannian Manifold $M$. The general linear group $GL(n,\mathbb{K})$ is the group consisting of all invertible $n\times n$ matrices over the field $\mathbb{K}$, and the group of orthogonal $n\times n$ real matrices be denoted by $O(n)$.\\
Topological spaces are typically denoted by $X$, $Y$, $Z$. Manifolds tend to be denoted by $M^n$, $N^n$, where $n$ indicates the dimension.
\begin{definition}\rm{
A Riemannian manifold $M^{dn}$ is called a real, complex, quarternionic or Cayley hyperbolic manifold provided its universal cover is isometric to $\mathbb{R}\textbf{H}^n$, $\mathbb{C}\textbf{H}^n$, $\mathbb{H}\textbf{H}^n$ and $\mathbb{O}\textbf{H}^2$, respectively. (Note that $n\geq 2$ and when $\mathbb{K}= \mathbb{O}$, $n$ can only be $2$.)}
\end{definition}
The following definition is taken from \cite{Oku01}:
\begin{definition}\rm{
Given a symmetric space $G/K$ of non-compact type, we define the dual symmetric space in the following manner. Let $G_c$ denote the complexification of the semi-simple Lie group $G$, and let $G_u$ denote the maximal compact subgroup in $G_c$. Since $K$ is compact, under the natural inclusions $K\subset G\subset G_c$, we can assume that $K\subset G_u$ (up to conjugation). The symmetric space dual to $G/K$ is defined to be the symmetric space $G_u/K$. By abuse of language, if $X=\Gamma\setminus G/K$ is a locally symmetric space modelled on the symmetric space $G/K$, we will say that $X$ and $G_u/K$ are dual spaces.}
\end{definition}
\begin{remark}\rm{
The dual symmetric spaces of real, complex, quarternionic or Cayley hyperbolic manifolds are the sphere, complex projective space, quaternionic projective space or Cayley projective plane respectively.}
\end{remark}
\begin{example}\label{exm.loc4}{\rm \cite{Oku01}}\rm{
Suppose $G=G_c$ is a complex semi simple Lie group, and let $G_u$ denote its maximal compact subgroup. The complexification of $G_c$ is then isomorphic to $G_c\times G_c:$
$$(G_c)_c= G_c\times G_c$$
The maximal compact subgroup of the complexification is $G_c\times G_c$. So we have dual symmetric spaces:
\begin{align*}
X&=\Gamma\setminus G_c/G_u\\
X_u&=G_u=G_c\times G_c/G_u
\end{align*}
In particular, if $G=SL(n,\mathbb{C})$, we have a pair:
\begin{align*}
X&=\Gamma\setminus SL(n,\mathbb{C})/SU(n)\\
X_u&=SU(n)
\end{align*}}
\end{example}
\begin{definition}\rm{
The dimension of the maximal torus of a compact Lie group $G$ is called the rank of $G$.}
\end{definition}
\begin{theorem}(Whitney Embedding Theorem)(\cite{Whi36, Whi44})
If $f : N^n\to M^m$ is a map of manifolds such that 
\begin{center}
either $2n + 1\leq m$\\
or $m = 2n\geq 6$ and $\pi_1(M)=0$
\end{center}
then $f$ is homotopic to an embedding $N^n\hookrightarrow M^m$.
\end{theorem}

\begin{definition}\rm{
A smooth manifold $M^n$ is stably parallelizable if its tangent bundle $TM$ is stably trivial.}
\end{definition}
\begin{definition}\rm{
A group $G$ is residually finite if for every nontrivial element $g$ in $G$ there is a homomorphism $h$ from $G$ to a finite group, such that $h(g)\neq 1$. }
\end{definition}
\begin{theorem}(Sullivan, \cite{Sul79})\label{sullivan}
Every closed real hyperbolic manifold has a stably parallelizable finite sheeted cover.
\end{theorem} 
 \begin{theorem}(Borel, \cite{Bor63})\label{recomhy}
There exists closed real hyperbolic manifolds in every dimension $n\geq2$, as well as closed complex hyperbolic manifolds in every even (real) dimension.
\end{theorem}
\begin{theorem}(Malcev, \cite{MR81})\label{malcev}
Let $M$ be a closed Riemannian manifold which is either real or complex hyperbolic, then $\pi_{1}M$ is residually finite. 
\end{theorem}
\begin{definition}\label{smoo.stru}(Smooth structure sets)
Let $M$ be a closed topological manifold. We define the smooth structure set $\mathcal{S}^{s}(M)$ to be the set of equivalence classes of pairs $(N,f)$ where $N$ is a closed smooth manifold and $f: N\to M$ is a homotopy equivalence.
And the equivalence relation is defined as follows :
\begin{center}
  $(N_1,f_1)\sim  (N_2,f_2)$ if there is a diffeomorphism $h:N_1\to N_2$ \\
  such that $f_2\circ h$ is homotopic to $f_1$.
\end{center}
 \end{definition}
\begin{definition}\rm{
Let $M$ be a topological manifold. Let $(N,f)$ be a pair consisting of a smooth manifold $N$ together with a homeomorphism $f:N\to M$. Two such pairs $(N_{1},f_{1})$ and $(N_{2},f_{2})$ are concordant provided there exists a diffeomorphism $g:N_{1}\to N_{2}$ such that the composition $f_{2}\circ g$ is topologically concordant to $f_{1}$ i.e., there exists a homeomorphism $F: N_{1}\times [0,1]\to M\times [0,1]$ such that $F_{|N_{1}\times 0}=f_{1}$ and $F_{|N_{1}\times 1}=f_{2}\circ g$.\\
Equivalently, let $N_1$, $N_2$ be two smooth structures on $M$. $N_1$ is said to be (topologically) concordant to $N_2$ if there is a smooth structure $\bar{M}$ on $M\times [0, 1]$ such that $\partial_{-}\bar{M}=N_1$ and $\partial_{-}\bar{M}=N_2.$}. The set of all such concordance classes is denoted by $\mathcal{C}(M)$.
\end{definition}
\indent The key to analyzing $\mathcal{C}(M)$ is the following result due to Kirby and Siebenmann :
\begin{theorem}\label{kirby}{\rm \cite[p.194]{KS77}}
There exists a connected $H$-space $Top/O$ such that for any smooth manifold $M$ with $\dim M\geq5$, there is a bijection between $\mathcal{C}(M)$ and $[M; Top/O]$. Furthermore, the equivalence class of $(M, id_{M})$ corresponds to the homotopy class of the constant map under this bijection.
\end{theorem}
\indent We recall some terminology from \cite{KM63} :
\begin{definition}\rm{
\begin{itemize}
\item[(a)] A homotopy $n$-sphere $\Sigma^n$ is an oriented smooth closed manifold homotopy equivalent to $\mathbb{S}^n$.
\item[(b)]A homotopy $n$-sphere $\Sigma^n$ is said to be exotic if it is not diffeomorphic to $\mathbb{S}^n$.
\item[(c)] Two homotopy $n$-spheres $\Sigma^{n}_{1}$ and $\Sigma^{n}_{2}$ are said to be equivalent if there exists an orientation preserving diffeomorphism $f:\Sigma^{n}_{1}\to \Sigma^{n}_{2}$.
\end{itemize}
The set of equivalence classes of homotopy $n$-spheres is denoted by $\Theta_n$. The equivalence class of $\Sigma^n$ is denoted by [$\Sigma^n$]. When $n\geq 5$, $\Theta_n$ forms an abelian group with group operation given by connected sum $\#$ and the zero element represented by the equivalence class  of the round sphere $\mathbb{S}^n$. M. Kervaire and J. Milnor \cite{KM63} showed that each  $\Theta_n$ is a finite group; in particular,  $\Theta_{8}$, $\Theta_{14}$ and  $\Theta_{16}$ are cyclic groups of order $2$, $\Theta_{10}$ and $\Theta_{20}$ are cyclic groups of order 6 and 24 respectively and $\Theta_{18}$ is a group of order 16. }
\end{definition}
\indent Start by noting that there is a homeomorphism $h: M^n\#\Sigma^n \to M^n$ $(n\geq5)$ which is the inclusion map outside of homotopy sphere $\Sigma^n$ and well defined up to topological concordance. We will denote the class in $\mathcal{C}(M)$ of $(M^n\#\Sigma^n, h)$ by $[M^n\#\Sigma^n]$. (Note that $[M^n\#\mathbb{S}^n]$ is the class of $(M^n, id_{M^n})$.) Let $f_{M}:M^n\to \mathbb{S}^n $  be a degree one map. Note that $f_{M}$ is well-defined up to homotopy. Composition with $f_{M}$ defines a homomorphism
$$f_{M}^*:[\mathbb{S}^n, Top/O]\to [M^n ,Top/O],$$ and in terms of the identifications
\begin{center}
$\Theta_n=[\mathbb{S}^n, Top/O]$ and $\mathcal{C}(M^n)=[M^n ,Top/O]$
\end{center}
given by Theorem \ref{kirby}, $f_{M}^*$ becomes $[\Sigma^n]\mapsto [M^n\#\Sigma^n]$.
\begin{theorem}(Vanishing Theorem, \cite{FJ91})\label{whigro}
 Let $M$ be a closed (connected) non-positively curved Riemannian manifold. Then $Wh(\pi_1(M))=0$, where $Wh(G)$ denotes the whitehead group at any group $G$.
\end{theorem}
\begin{theorem}\label{nilrig}\cite{FH83}
Let $M$ be a closed aspherical manifold whose fundamental group is virtually nilpotent. Let $N$ be a topological manifold (possibly with boundary) and let $H: N\to M\times \mathbb{D}^m$ be a homotopy equivalence which is a homeomorphism on the boundary. Assume $m+n>4$. Then $H$ is homotopic to a homeomorphism $\overline{H}:N \to M\times \mathbb{D}^m$ relative to the boundary.
\end{theorem}
\begin{theorem}\label{gentoprigd}(Topological Rigidity Theorem, \cite{FJ93})
Let $M$ be a closed connected $n$-dimensional Riemannian manifold with non positive sectional
curvature. Let $N$ be a topological manifold (possibly with boundary) and let $H: N\to M\times \mathbb{D}^m$ be a homotopy equivalence which is a homeomorphism on the boundary. Assume $m+n\neq 3, 4$. Then $H$ is homotopic to a homeomorphism $\overline{H}:N \to M\times \mathbb{D}^m$ relative to the boundary.
\end{theorem}
\begin{theorem}\label{borel}(Borel Conjecture, 1955)
Let $f:M\to N$ be a homotopy equivalence where both $M$ and $N$ are closed aspherical manifolds. Then $f$ is homotopic to a homemorphism.
\end{theorem}
\begin{theorem}\label{bieberbachrig}(Bieberbach's Rigidity Theorem, 1912)
Let $f:N\to M$ be a homotopy equivalence between closed flat Riemannian manifolds. Then $f$ is homotopic to an affine diffeomorphism.
\end{theorem}
\begin{theorem}\label{mostow}(Mostow Rigidity Theorem, \cite{Mos73})
Let $M$ and $N$ be compact, locally symmetric Riemannian manifolds with everywhere nonpositive curvature having no closed 1 or 2-dimensional geodesic subspaces which are locally direct factors. If  $f:M\to N$ is a homotopy equivalence, then $f$ is homotopic to an isometry. 
\end{theorem}
\section{\large Detecting Exotic Structures on Hyperbolic Manifolds}
In this section we discuss smooth and PL-rigidity problem \ref{genprorig}. In particular, we review Problem \ref{genprorig} for negatively curved locally symmetric spaces. First we focus on the following problem :
\begin{problem}\label{prorig}
Let $f:N\to M$ denote a homotopy equivalence between closed smooth manifolds. Is $f$ homotopic to a diffeomorphism ?
\end{problem}
\begin{remark}\rm{
\indent
 \begin{itemize}
\item[\bf{1.}] Every homotopy equivalence of 2-dimensional closed manifolds is homotopic to a diffeomorphism, by the classification of surfaces. A homotopy equivalence of 3-dimensional closed manifolds is not in general homotopic to a diffeomorphism. The first examples of such homotopy equivalences appeared in the classification of the 3-dimensional lens spaces in the 1930's : the Reidemeister torsion of a lens space is a diffeomorphism invariant which is not homotopy invariant([RA02]). Algebraic $K$-theory invariants such as Reidemeister and Whitehead torsion are significant in the classification of manifolds with finite fundamental group, and in deciding if $h$-cobordant manifolds are diffeomorphic (via s-Cobordism Theorem), but they are too special to decide if an arbitrary homotopy equivalence of closed manifolds is homotopic to a diffeomorphism.
\item[\bf{2.}] In 1956, Milnor \cite{Mil56} constructed an exotic sphere $\Sigma^7$ with homotopy equivalence (in fact a homeomorphism) $\Sigma^7\mapsto \mathbb{S}^7$ which is not homotopic to a diffeomorphism.
\item[\bf{3.}] If both closed manifolds in Problem \ref{prorig} are real hyperbolic and of dimension greater than $2$, Mostow’s Rigidity Theorem \ref{mostow} says that they are isometric, in particular diffeomorphic. If $M$ is flat manifold in Problem \ref{prorig}, then it follows from the works of Gromoll and Wolf \cite{GW71}, and Yau \cite{Yau71}, that $N$ is flat manifold; thus $f$ in Problem \ref{prorig} must be homotopic to an affine diffeomorphism by Bieberbach's Rigidity Theorem \ref{bieberbachrig}. If $M$ is an irreducible locally symmetric space of rank $\geq 2$, Gromov \cite{BGS85} has shown that after rescaling the metric on $M$, $f$ will be homotopic to an isometry. Eberlein \cite{Ebe83, Ebe83a} independently proved the same result under the hypothesis that the universal cover of $M$ is reducible. 
\end{itemize}}
\end{remark}
Recall that by an infranilmanifold we mean the quotient of a simply-connected nilpotent Lie group $G$ by the action of a torsion free discrete subgroup $\Gamma$ of the semidirect product of $G$ with a compact subgroup of $Aut(G)$. A weaker property than flatness is admitting an infranil structure. Lee and Raymond \cite{LR85} proved the smooth rigidity for infranilmanifolds:
\begin{theorem}\label{infranil}
 Let $f:M\to N$ be a homotopy equivalence between two closed infranilmanifolds $M$ and $N$. Then $f$ is homotopic to a diffeomorphism.
\end{theorem}
Let $\rm{Aff(S)}$ denote the group of all affine motions of $S$; i.e., $\rm{Aff(S)}$ is the semi-direct product $S\rtimes \rm{Aut(S)}$.
\begin{definition}\rm{
A compact infrasolvmanifold $M$ is a compact orbit space of the form $M=\Gamma\setminus S$ where $S$ is solvable and simply connected Lie group and closed torsion free subgroup $\Gamma \subset$ Aff$(S)$ satisfying 
\begin{itemize}
\item[(i)] $\Gamma^0$ is contained in the nil radical of $S$ and 
\item[(ii)] the closure of the image of $\Gamma$ in $\rm{Aut(S)}$ is compact.
\end{itemize}}
\end{definition}
\begin{remark}\rm{
\indent
\begin{itemize}
\item[\bf{1}.] The notion of compact infrasolvmanifold generalizes the notion of compact Riemannian flat manifold and compact infranilmanifold which are its special cases when $S$ is respectively abelian and nilpotent. It also generalizes the notion of compact solvmanifold which is the special case when $\Gamma \subset S$. 
\item[\bf{2.}] The fundamental group of an infrasolvmanifold is a virtually polycyclic group. A result of Farrell and Jones \cite{FJ88} on aspherical manifolds with virtually polycyclic fundamental group shows that infrasolvmanifolds are topologically rigid. A weaker property than that of admitting an infranil structure is admitting an infrasolv structure. F. T. Farrell and L. E. Jones \cite{FJ97} proved the smooth rigidity for infrasolvmanifolds: Here is the result:
\end{itemize}}
\end{remark}
\begin{theorem}\label{FJ-infrasol}
Any two compact infrasolvmanifolds, of dimension different from $4$, whose fundamental groups are isomorphic are smoothly diffeomorphic. In fact, any given isomorphism of fundamental groups is induced by a smooth diffeomorphism.
\end{theorem}
\begin{remark}\rm{
\indent
\begin{itemize}
\item[\bf{1}.] In the cases of compact Riemannian flat manifolds and compact infranilmanifolds, Theorem \ref{FJ-infrasol} was previously proven by Bieberbach \cite{Bie12} and Lee-Raymond \cite{LR85}, respectively. The isomorphism in these cases is in fact induced by an affine diffeomorphism (see \cite{Bie12} and \cite{Bro65}) (A map is affine provided it sends geodesics to geodesics). And Mostow \cite{Mos54} previously showed that Theorem \ref{FJ-infrasol} is also true for solvmanifolds; however an affine diffeomorphism is not always possible in this case. One difficulty preventing an affine diffeomorphism is that a group $\Gamma$ can be lattice in two different simply connected solvable Lie groups; but in atmost one simply connected nilpotent Lie group because of Malcev's Rigidity Theorem \cite{Rag72}.  
\item[\bf{2}.] Wilking \cite{Wil00} improved on Theorem \ref{FJ-infrasol}  by showing the condition “ $\dim=4$” can be dropped. Results of Wilking on rigidity properties of isometric actions on solvable Lie-groups \cite{Wil00} imply the smooth rigidity of infrasolvmanifolds in all dimensions.
\item[\bf{3}.] In well known cases, smooth rigidity properties of geometric manifolds are closely connected to rigidity properties of lattices in Lie groups. Oliver Baues \cite{Bau04} proved the smooth rigidity of infrasolvmanifolds from natural rigidity properties of virtually polycyclic groups in linear algebraic groups. More generally, he proved  rigidity results for manifolds which are constructed using affine, not necessarily isometric, actions of virtually polycyclic groups on solvable Lie groups. This approach leads us to a new proof of the rigidity of infrasolvmanifolds, and also to a geometric characterization of infrasolvmanifolds in terms of polynomial actions on the affine space $\mathbb{R}^n$.
\item[\bf{4}.] Indeed, the theory of harmonic maps had been very successful in showing rigidity results, see for instance Siu \cite{Siu80}, Sampson \cite{Sam86}, Hernandez \cite{Her91}, Corlette \cite{Cor92}, Gromov and Schoen \cite{GS92}, Jost and Yau \cite{JY93}, and Mok, Sui and Yeung \cite{MSY93}. Because of this evidence it seems reasonable that Lawson and Yau conjectured that the answer to Problem \ref{prorig} was affirmative.
\end{itemize}}
\end{remark}
\begin{conjecture} (Lawson-Yau Conjecture)
Let $M_{1}$ and $M_{2}$ be  negatively curved manifolds. If $\pi_{1}(M_{1})=\pi_{1}(M_{2})$, then $M_{1}$ is diffeomorphic to $M_{2}$.
\end{conjecture}
\rm{ Recall that the obvious smooth analogue of Borel's Conjecture is false. Namely, Browder had shown in \cite{Bro65} that it is false even in the basic case where $M=T^n$ is the $n$-torus (see also \cite[Remark 4.7]{Ram15}). In fact, it was shown in \cite{Wal71} that $T^n$ and $T^n\#\Sigma^n$ ($n\geq5$) are homeomorphic but not diffeomorphic if $n\geq 5$ and $\Sigma^n$ is any exotic $n$-sphere. However, when it is assumed that both $M$ and $N$ in Borel Conjecture \ref{borel} are non-positively curved Riemannian manifolds, then smooth rigidity frequently happens. The most fundamental instance of this is an immediate consequence of Mostow's Strong Rigidity Theorem \ref{mostow}. Mostow's Rigidity Theorem for hyperbolic space forms gives a positive answer to Lawson-Yau Conjecture when both $M_1$ and $M_2$ have constant sectional curvature and more generally when both $M_1$ and $M_2$ are locally symmetric spaces. See Siu \cite{Siu80} and Hamenstadt \cite{Ham91} for generalizations of Mostow's Rigidity Theorem relevant to Lawson-Yau 
Conjecture. Also see Mostow and Siu \cite{Mos80} and Gromov and Thurston \cite{GT87} for other examples of negatively curved manifolds which are not diffeomorphic to a locally symmetric space. For Lawson-Yau Conjecture, Farrell and Jones gave counterexamples \cite{FJ89a}, which was loosely  motivated by Farrell and Jones construction \cite{FJ78} used to prove the following results:}
\begin{definition}\rm{
Let $M$ be a closed smooth manifold. A self-map $f:M\to M$ is said to be an expanding endomorphism provided $M$ supports a Riemannian metric such that $|df(v)|>|v|$ for every non-zero vector $v$ tangent to $M$}.
\end{definition}
Here is the following important result relative to this definition:
\begin{theorem}(Cartan) Let $M$ be non-positively  curved and $x_0\in M$ be a base point. Then $\exp :T_{x_0}M\to M$ is an expanding map. Furthermore it is a covering projection and hence a diffeomorphism when $\pi_1(M)=0$.
\end{theorem}
\begin{question}\label{expand.exi}
 What closed smooth manifolds support expanding endomorphisms?
\end{question}
The question is answered up to topological classification as follows by results due to Shub \cite{Shu69}, Franks \cite{Fra70} and Gromov \cite{Gro81}.
\begin{theorem}\label{expand}
 If a closed smooth manifold $M$ supports an expanding endomorphism, then $M$ is homeomorphic to an infranilmanifold.
\end{theorem}
\begin{remark}\rm{
Shub showed that the universal cover $\widetilde{M}$ of $M$ is diffeomorphic to $\mathbb{R}^n$ where $n=\dim M$. Then Franks showed that $\pi_{1}M$ has polynomial growth and that $M$ is homeomorphic to an infranilmanifold provided $\pi_{1}M$ is virtually solvable. Gromov completed the proof of Theorem \ref{expand} by showing that a group of polynomial growth must be virtually nilpotent. Gromov's result was motivated by Hirsch's paper \cite{Hir70} where it is shown that the solution to Hilbert's fifth problem is related to  Theorem \ref{expand}. Hirsch also implicitly poised Question \ref{expand.exi} in his Remark 1; i.e., whether the word ``homeomorphism'' can be replaced by ``diffeomorphism'' in Theorem \ref{expand}. But Farrell  and Jones showed in \cite{FJ78} that this is not the case; namely, they proved the following result:}
\end{remark} 
\begin{theorem}\label{expand.torus}
Let $T^n$ be the $n$-torus $(n>4)$ and $\Sigma^n$ an arbitrary homotopy sphere, then the connected sum $T^n\#\Sigma^n$ admits an expanding endomorphism.
\end{theorem}
\begin{remark}\rm{
 By Theorem \ref{expand.torus} and Theorem \ref{expand}, $T^n\#\Sigma^n$ is homeomorphic to an infranilmanifold.}
\end{remark}
\begin{theorem}\label{diff-con}\cite{Far96}
Let $M^n$ be a closed Riemannian manifold (with $n\geq 5$) which is a locally symmetric space whose sectional curvatures are either identically zero or all negative. Let $N^n$ be a smooth structure on $M^n$. If $N^n$ is diffeomorphic to $M^n$, then $N^n$ and $M^n$ represent the same element in $\mathcal{C}(M)$; i.e., they are topologically concordant.
\end{theorem}
We have the following result due to Browder \cite{Bro65} and Brumfiel \cite{Bru71}:
\begin{theorem}\label{brum}\cite{Far96}
Assume that $M^n$ is an oriented closed (connected) smooth manifold which is stably parallelizable and that $n\geq 5$. Then $f_{M}^*:\Theta_n \to [M^n ,Top/O]$ is monic.
\end{theorem}
\begin{proof}
Since $X\to [X,Top/O]$ is a homotopy functor on the category of topological spaces, Theorem \ref{brum} would follow immediately if $f_M : M^n\to \mathbb{S}^n$ is homotopically split. That is, if there exists a map $g:\mathbb{S}^n\to M^n$ such that $f_M\circ g$ is homotopic to $id_{\mathbb{S}^n}$. Unfortunately, $f_M$ is only homotopically split when $M$ is a homotopy sphere. But we can use the fact that $M$ is stably parallelizable to always stably split $f_M$ up to homotopy; to show that the $(n+1)$-fold suspension
\begin{center}
 $\Sigma^{n+1}(f_M):\Sigma^{n+1}M^n\to \mathbb{S}^{2n+1}$
\end{center}
of $f_M$ is homotopically split. This is done as follows. Note first that $M^n\times \mathbb{D}^{n+1}$ can be identified with a codimension-0 smooth submanifold of $\mathbb{S}^{2n+1}$ by using the Whitney embedding theorem together with the fact that $M$ is stably parallelizable. Let $\star$ be a base point in $M$. Then dual to the inclusion
\begin{center}
 $M^n\times \mathbb{D}^{n+1}\subset \mathbb{S}^{2n+1}$
\end{center}
is a quotient map $\phi:\mathbb{S}^{2n+1}\to \Sigma^{n+1}M^n$ realizing the $(n+1)$-fold reduced suspension $\Sigma^{n+1}M^n$ of $M^n$ as a quotient space of  $\mathbb{S}^{2n+1}$. Namely, $\phi$ collapses everything outside of $M^n\times Int(\mathbb{D}^{n+1})$ together with $\star\times \mathbb{D}^{n+1}$ to the base point of $\Sigma^{n+1}M^n$, and is a bijection between the remaining points. And it is easy to see that the composition $\Sigma^{n+1}(f_M)\circ \phi$ is homotopic to $id_{\mathbb{S}^n}$; i.e., $\Sigma^{n+1}(f_M)$ is homotopically split. But this is enough to show that $f^{*}_M$ is monic since $Top/O$ is an $\infty$-loop space \cite{BV73}; in particular, there exists a topological space $Y$ such that $\Omega^{n+1}(Y)=Top/O$. This fact is used to identify the functor
\begin{center}
 $X\to [X,Top/O]=[X,\Omega^{n+1}(Y)]$
\end{center}
with the functor $X\to [\Sigma^{n+1}M^n, Y]$. Consequently,
\begin{center}
 $f_{M}^*:[\mathbb{S}^n, Top/O]\to [M^n,Top/O]$
\end{center}
is identified with 
\begin{center}
 $(\Sigma^{n+1}(f_M))^*:[\mathbb{S}^{2n+1}, Y]\to [\Sigma^{n+1}M^n, Y]$.
\end{center}
But, this last homomorphism is monic since $\Sigma^{n+1}(f_M)$ is homotopically split.
\end{proof}
The following result is an immediate consequence of Theorem \ref{diff-con} and Theorem \ref{brum}; it will be used to construct exotic smoothings of some symmetric spaces.
\begin{theorem}\label{corr.brum}\cite{Far96}
Let $M^n$ be a closed,oriented (connected) stably parallelizable Riemannian locally symmetric space (with $n\geq 5$) whose sectional curvatures are either identically zero or all negative. Let $\Sigma^n$ be an exotic sphere, then $M\#\Sigma^n$ is not diffeomorphic to  $M^n$. 
\end{theorem}
\begin{remark}\rm{
\indent
 \begin{itemize}
 \item [\bf{1.}] Since $T^n$ is a stably parallelizable flat Riemannian manifold, by Theorem \ref{corr.brum}, $T^n\#\Sigma^n$ is not diffeomorphic to $T^n$ where $\Sigma^n$ is an exotic sphere.
 \item[\bf{2.}] The Malcev's Rigidity Theorem \cite{Mal51}, shows that any closed infranilmanifold with abelian fundamental group  must be Riemannian flat. And  Bieberbach's Rigidity Theorem \ref{bieberbachrig} shows that any such manifold is diffeomorphic to a torus. This implies that $T^n\#\Sigma^n$ is not diffeomorphic to any infranilmanifold.
 \end{itemize}}
\end{remark}
By Theorem \ref{recomhy}, Theorem \ref{sullivan} and Theorem \ref{corr.brum}, we have the following exotic smoothings of Farrell-Jones \cite{FJ89a}:
\begin{theorem}
Let $\Sigma^n$ be an exotic sphere $(n\geq 5)$. Then there exists a closed real hyperbolic manifold $M^n$ such that $M^n\#\Sigma^n$ and $M^n$ are not diffeomorphic.
\end{theorem}
In fact, Farrell and Jones proved the following results \cite{FJ89a}:\\

Let $\Sigma_{1}$, $\Sigma_{2}$ , ..., $\Sigma_{k}$ be a complete list of inequivalent exotic spheres of dimension $n$, where two exotic spheres are equivalent provided they are diffeomorphic, but not necessarily via an orientation preserving diffeomorphism. (The standard sphere $\mathbb{S}^n$ is not included in this list.) The set of equivalence classes of homotopy $n$-spheres under this relation is denoted by $\Theta_{n}^{+}$.
\begin{theorem}\label{hype.smooth}
Let  $M^n$ be a closed real hyperbolic manifold with dimension $n\geq 5$.
Given any real number  $\delta >0$, there is a finite sheeted covering space $\widehat{M}$ of $M$ which satisfies the following properties:
\begin{itemize}
\item[(i)] No two of the manifolds $\widehat{M}$, $\widehat{M}\#\Sigma_{1}$, $\widehat{M}\#\Sigma_{2}$, ... , $\widehat{M}\#\Sigma_{k}$ are diffeomorphic, but they are all homeomorphic to one another.
\item[(ii)] Each of the manifolds $\widehat{M}\#\Sigma_{1}$, $\widehat{M}\#\Sigma_{2}$ , ... , $\widehat{M}\#\Sigma_{k}$  supports a Riemannian metric, all of whose sectional curvature values lie in the interval $(-1 -\delta , -1 +\delta)$.
\end{itemize}
\end{theorem}
\begin{remark}\rm{
\indent
\begin{itemize}
\item[\bf{1.}] This result is startling for a number of reasons. First, note that the manifolds are obviously homeomorphic. Second, by Mostow Rigidity Theorem \ref{mostow}, $\widehat{M}^n\#\Sigma^n$ cannot admit a metric of constant negative curvature, or else it would be isometric, hence diffeomorphic to $\widehat{M}^n$. Thus these manifolds can be added to the short list of closed manifolds which have (pinched) negative curvature and are not diffeomorphic to a locally symmetric space. Third, the examples of Farrell and Jones given by Theorem \ref{hype.smooth} provide counterexamples to the Lawson-Yau Conjecture.
\item[\bf{2.}] Theorem \ref{hype.smooth} gives interesting counterexamples for smooth analogue of Borel's Conjecture \ref{borel}.
\end{itemize}}
\end{remark}
The proof of Theorem \ref{hype.smooth} depends on the following two theorem:
\begin{theorem}\label{hype.exto}
Let  $M^n$ be a closed real hyperbolic manifold with dimension $n\geq 5$.
Suppose that the finite sheeted covering space $\widehat{M}$ of $M$ is stably parallelizable. Then no two of the manifolds $\widehat{M}$, $\widehat{M}\#\Sigma_{1}$, $\widehat{M}\#\Sigma_{2}$ , ..., $\widehat{M}\#\Sigma_{k}$ are diffeomorphic.
\end{theorem}
\begin{theorem}\label{hype.metric}
Given a real number $\delta >0$, there is a real number $\alpha > 0$ which depends only on $n= \dim M$ and $\delta$ such that the following is true. Suppose that the finite sheeted covering space $\widehat{M}$ of $M$ has radius of injectivity greater than $3 \alpha$ at some point $p\in\widehat{M}$. Then each of the manifolds $\widehat{M}\#\Sigma_{1}$, $\widehat{M}\#\Sigma_{2}$,...., $\widehat{M}\#\Sigma_{k}$ supports a Riemannian metric, all of whose sectional curvature values lie in the interval $(-1 -\delta , -1 +\delta)$.
\end{theorem}
\paragraph{Proof of Theorem \ref{hype.smooth} assuming Theorem \ref{hype.exto} and Theorem \ref{hype.metric} :}
First, we use Theorem \ref{sullivan} that there is a finite covering space $\tilde{M}$ of $M$ such that $\tilde{M}$ is stably parallelizable. Let $g_1$, $g_2$, ... ,$g_x$ be a list of all the closed geodesics in $\tilde{M}$ which have length less than or equal to $6\alpha$, where $\alpha$ comes from Theorem  \ref{hype.metric}. Choose a point $q\in \tilde{M}$ and elements $\beta_1$, $\beta_2$, ... ,$\beta_x$ $\in \pi_{1}(\tilde{M} ,q)$ such that each $\beta_i$ is represented by a map $f_{i}:\mathbb{S}^{1}\to \tilde{M}$ which is freely homotopic to $g_i$. By Theorem \ref{malcev}, $\pi_{1}(\tilde{M} ,q)$ is residually finite and so there is a homomorphism $h:\pi_{1}(\tilde{M} ,q)\to G$ onto a finite group $G$ such that $h(\beta_i)\neq 1$ for all indices $i$.
Let $\widehat{M}$ denote the finite covering space of $\tilde{M}$ corresponding to the kernel of
$h$, which is a subgroup of $\pi_{}(\tilde{M} ,q)$. Note that  $\widehat{M}$ is a finite sheeted covering
space of $M$ which is stably parallelizable and which has radius of injectivity
greater than $3\alpha$ at each of its points. Thus, we may apply Theorem \ref{hype.exto} and Theorem \ref{hype.metric} to $\widehat{M}$ to conclude the proof of this theorem. Note that in (a) of this theorem, we use the general fact that the topological type of a manifold is not changed by forming a connected sum of it with an exotic sphere.\\

We shall now prove Theorems \ref{hype.exto}. For that we need the following lemma :
\begin{theorem}\label{Lemma 2.1}
Let  $M^n$ be a closed real hyperbolic manifold with dimension $n\geq 5$ and let $\widehat{M}$ be a finite sheeted  orientable covering space  of $M$. Set $N =\widehat{M}$, $N_{1}=\widehat{M}\#\Sigma_{i}$, and $N_{2}=\widehat{M}\#\Sigma_{j}$. Then  $N_{1}$ is concordant to  $N_{2}$ if and only if $N_{1}$ is diffeomorphic to  $N_{2}$ via an orientation-preserving diffeomorphism.
\end{theorem}
\begin{proof}
First suppose that $N_{1}$ is concordant to $N_{2}$ via a smooth structure $\bar{N}$ for $N\times[0,1]$. Since $\bar{N}$ is topologically a product, it follows from   Theorem \ref{whigro} and from the smooth s-cobordism theorem that $\bar{N}$ is a product in the smooth category. Thus, $N_{1}$ is orientation-preserving diffeomorphic to $N_{2}$, since $N_{1}=\partial_{-}\bar{N}$ and $N_{2}=\partial_{-}\bar{N}$.

We can assume that the connected sums of  $\widehat{M}$ with $\Sigma_{i}$ and $\Sigma_{j}$ are taking place on the boundary of a small metric ball $B$ in $\widehat{M}$, so we think of $N_{1}$ and $N_{2}$ as being topologically identified with $\widehat{M}$, and the changes in the smooth structure happen inside $B$. Let $f:\widehat{M}\#\Sigma_{i}\rightarrow \widehat{M}\#\Sigma_{j}$ be a diffeomorphism. First we consider a special case, where $f:\widehat{M}\rightarrow \widehat{M}$ is homotopic to the identity. Thus we have a homotopy $f:\widehat{M}\times [0,1]\rightarrow \widehat{M}$ with $h|_{\widehat{M}\times 1} = f$ and $h|_{\widehat{M}\times 0} = id$. Define $H:\widehat{M}\times [0,1]\rightarrow \widehat{M}\times [0,1]$ by $H(x, t) = (h(x, t), t)$. Note that $H$ is a homotopy equivalence which restricts to a homeomorphism on the boundary. Therefore, by Farrell and Jones Topological Rigidity Theorem \ref{gentoprigd}, $H$ is homotopic rel boundary to a homeomorphism 
$\widehat{H}:\widehat{M}\times [0,1]\rightarrow \widehat{M}\times [0,1]$. We put the smooth structure $N_{2}\times [0,1]$ on the range of $\widehat{H}$ and, by pulling it back via $\widehat{H}$, obtain the smooth structure $N$ on the domain of $\widehat{H}$. Since, by construction, $\widehat{H}:N\rightarrow N_{2}\times [0,1]$ is a diffeomorphism, $\widehat{H}|_{\widehat{M}\times 1} = f$ and $\widehat{H}|_{\widehat{M}\times 0} = id$, $N$ is a concordance between $N_{1}$ and $N_{2}$.

The general case reduces to the above special case as follows: If $f:\widehat{M}\rightarrow \widehat{M}$ is an orientation preserving homeomorphism, then, by the Strong Mostow Rigidity Theorem \ref{mostow}, $f^{-1}$ is homotopic to an orientation preserving isometry $g$. Since one can move around small metric balls in  $\widehat{M}$ by smooth isotopies, $\widehat{M}$ is homotopic to a diffeomorphism $\widehat{g}:\widehat{M}\rightarrow \widehat{M}$ such that  $\widehat{g}|_{B} = id$. Since $N_{2}$ is obtained by taking the connected sum along the boundary of $B$ and $\widehat{g}|_{B} = id$ and $\widehat{g}|_{M\setminus B}$ is a diffeomorphism, it follows that $\widehat{g}:N_{2}\rightarrow N_{2}$ is also a diffeomorphism. Therefore the composition $\widehat{g}\circ f:N_{1}\rightarrow N_{2}$ is a diffeomorphism homotopic to the identity and it follows from the previous special case that $N_{1}$ and $N_{2}$ are concordant.
If $f$ is an orientation reversing homeomorphism, then similar argument produces a diffeomorphism $\widehat{g}:N_{2}\rightarrow N_{2}$ and it follows that $N_{1}$ is concordant to $N_{2}$.
\end{proof}
\begin{theorem}\label{Addendum 2.3}
Let  $M^n$ be a closed real hyperbolic manifold with dimension $n\geq 5$ and let $\widehat{M}$ be a finite sheeted covering space of $M$.
Suppose $\widehat{M}\#\Sigma_{i}$ is diffeomorphic to $\widehat{M}\#\Sigma_{j}$, then either $\widehat{M}\#\Sigma_{i}$ is concordant to $\widehat{M}\#\Sigma_{j}$ or to  $\widehat{M}\#(-\Sigma_{j})$. Also, if $\widehat{M}\#\Sigma_{i}$ is diffeomorphic to $\widehat{M}$, then $\widehat{M}\#\Sigma_{i}$ is concordant to $\widehat{M}$.
\end{theorem}
\paragraph{Proof of Theorem \ref{hype.exto} :}
Because of Theorem \ref{Addendum 2.3}, it suffices to prove that, for any pair $(\Sigma,\widehat{\Sigma})$ of distinct elements in $\Theta_{n}^{+}$, $\widehat{M}\#\Sigma$ is not concordant to $\widehat{M}\#\widehat{\Sigma}$. By Theorem \ref{kirby}, there is a one-to-one correspondence between concordance classes of smooth structures on $M$ and the homotopy classes of maps from $\widehat{M}$ to $Top/0$ (denoted by $[M, Top/0])$ with the hyperbolic structure on $\widehat{M}$ corresponding to the class of the constant map $[KS77]$. For the same reasons, there is a one-to-one correspondence between $\Theta_{n}^{+}$ and $[\mathbb{S}^n , Top/0]$ with $\mathbb{S}^n$ corresponding to the constant class. Let $\beta_{1}$ and $\beta_{2}$ in $[\mathbb{S}^n, Top/0]$ correspond to $\Sigma$ and $\widehat{\Sigma}$, respectively. Since $Top/0$ is an infinite loop space \cite{BV73}; in particular, there exists a topological space $X$ such that $Top/0=\Omega^2(X)$. If $\beta$ is the class of a map from $\mathbb{S}^n$ to $Top/0$, 
then  $\beta^*$ denotes the class of its composite with a degree-one map $f$ from $\widehat{M}$ to $\mathbb{S}^n$. The naturality of this construction, one can compute that the smooth structures $\widehat{M}\#\Sigma$ and $\widehat{M}\#\widehat{\Sigma}$  correspond to the elements $\beta_{1}^*$ and $\beta_{2}^*$, respectively. Thus, to complete the proof, it suffices to show that the map $f^*:[\mathbb{S}^n , Top/0]\rightarrow[\widehat{M}, Top/0]$ (given by sending $\beta$ to $\beta^*$) is monic, i.e., $\beta_{1}^*=\beta_{2}^*$. But, the homomorphism $f^*$ is monic by Theorem \ref{brum}. This completes the proof of Theorem \ref{hype.exto}.\\

Before beginning the proof of Theorem \ref{hype.metric}, we will recall a fact about manifolds of constant negative curvature and state a lemma \cite{FJ89a}:\\

Let $A( , )$ denote the Riemannian metric on $\mathbb{S}^{n-1}\times (0, 3)$ which is the product of the standard Riemannian metric on the unit $(n-1)$-sphere $\mathbb{S}^{n-1}$ with the standard Riemannian metric on the interval $(0, 3)$. Let $\xi$, $\gamma$ denote the distributions
on $\mathbb{S}^{n-1}\times (0, 3)$ which are tangent to the first and second factors, respectively. Let $P_{1}: T(\mathbb{S}^{n-1}\times (0, 3))\rightarrow \xi$, $P_{2}: T(\mathbb{S}^{n-1}\times (0,3))\rightarrow \gamma$ be the $A$-orthogonal projections. Define 2-tensors $A_{i}( , )$, $i = 1, 2 $ on $\mathbb{S}^{n-1}\times (0, 3)$ by $A_{i}(v, w)=A(P_{i}(v), P_{i}(w))$. Define a Riemannian metric $\overline{A}( , )$ on $\mathbb{S}^{n-1}\times (0, 3)$ as follows:\\
$\mathbf{Property~~ 1 :}$
$\overline{A}(v, w) =\sinh^2(\alpha t)A_{1}(v, w)+\alpha^2 A_{2}(v, w)$ for any pair of vectors  $v$, $w$ tangent to $\mathbb{S}^{n-1}\times (0, 3)$ at a point $(q, t)\in \mathbb{S}^{n-1}\times (0, 3)$.
It is well known that if $\widehat{M}$ has radius of injectivity greater than 3$\alpha$ at
$p\in \widehat{M}$, there is a smooth map $h:\mathbb{S}^{n-1}\times(0, 3)\rightarrow \widehat{M}$ which satisfies the following properties:\\
$\mathbf{Property~~ 2 :}$
\begin{itemize}
\item[(a)]$h$ is an embedding.
\item[(b)]For each $q\in\mathbb{S}^{n-1}$, the path $g(t)=h(q, t)$ is a geodesic of speed with \[ \lim_{t \to 0}g(t)= p \]
\item[(c)] The pull back along $h :\mathbb{S}^{n-1}\times (0, 3)\rightarrow \widehat{M}$ of the Riemannian metric $\langle ,\rangle_{\widehat{M}}$ is equal to $\overline{A}(  , )$.
\end{itemize}
Note that since $\langle ,\rangle_{\widehat{M}}$ has constant sectional curvature equal to $-1$, it
follows by Property 2(c) that $\overline{A}$ also has constant sectional curvature equal to $-1$. We let $B( , )$ be any Riemannian metric on $\mathbb{S}^{n-1}\times [1,2]$ which satisfies the following: \\
$\mathbf{Property~~ 3 :}$
\begin{itemize}
\item[(a)] For any $v\in \xi|_{\mathbb{S}^{n-1}\times[1,2]}$, $w\in \gamma|_{\mathbb{S}^{n-1}\times[1,2]}$, we have that $B(v, w)= 0$.
\item[(b)]If $t$ denotes the second coordinate variable in the product $\mathbb{S}^{n-1}\times[1,2]$,
then we have that $B(\frac{\displaystyle{\partial}}{\displaystyle{\partial t}}, \frac{\displaystyle{\partial}}{\displaystyle{\partial t}})= 1$.
\end{itemize}
We define a new Riemannian metric $\overline{B}( , )$ on  $\mathbb{S}^{n-1}\times[1,2]$ as follows:\\
$\mathbf{Property~~ 4 :}$
$\overline{B}(v, w) =\sinh^2(\alpha t)B_{1}(v, w)+\alpha^2 B_{2}(v, w)$ for any pair of vectors $v$, $w$ tangent to $\mathbb{S}^{n-1}\times [1,2]$ at a point $(q, t)\in \mathbb{S}^{n-1}\times [1,2]$, where $B_{i}(v, w) = B(P_{i}(v),P_{i}(w))$.
\begin{lemma}\label{pinched}
We let $P$ denote a $2$-plane tangent to $\mathbb{S}^{n-1}\times [1,2]$, and we let $K_{\overline{B}}(P)$ denote the sectional curvature of $P$ with respect to $\overline{B}(  , )$. Then \[\lim_{\alpha\to\infty} K_{\overline{B}}(P)=-1\] uniformly in $P$.
\end{lemma}
\begin{proof}
The proof of the theorem relies upon the following claim:\\
$\mathbf{Claim~~ A :}$ For any $(p_{0},t_{0})\in \mathbb{S}^{n-1}\times [1,2]$, there are two coordinated system $(\bar{x}_{1},\bar{x}_{2},....,\bar{x}_{n-1})$ and $(\bar{y}_{1},\bar{y}_{2},....,\bar{y}_{n-1})$ for $\mathbb{S}^{n-1}$ near $p_{0}$, and a linear coordinate $\bar{t}$ for $[1,2]$ near $t_{0}$ such that the following hold true:
\begin{itemize}
\item[(a)] $\overline{A}( , )=\bar{g}_{ij}^a d\bar{x}_{i} d\bar{x}_{j}+ d\bar{t}^2$, $\overline{B}( , )=\bar{g}_{ij}^b d\bar{y}_{i} d\bar{y}_{j}+d\bar{t}^2$, where $\bar{g}_{ij}^a=\overline{A}(\frac{\displaystyle{\partial}}{\displaystyle{\partial \bar{x}_{i}}},\frac{\displaystyle{\partial}}{\displaystyle{\partial \bar{x}_{j}}})$, and $\bar{g}_{ij}^b=\overline{B}(\frac{\displaystyle{\partial}}{\displaystyle{\partial \bar{y}_{i}}},\frac{\displaystyle{\partial}}{\displaystyle{\partial \bar{y}_{j}}})$.
\item[(b)]$\bar{g}_{ij}^a(p_{0},t_{0})=\bar{g}_{ij}^b(p_{0},t_{0})$.
\item[(c)]Let $k$ and $s$ denote any non-negative integers satisfying $k + s < 2$, and let $\overline{X}_{i,j,i_{1}....i_{k},s}$ and $\overline{Y}_{i,j,i_{1}....i_{k},s}$ denote the partial derivatives (through second order) $\frac{\displaystyle{\partial^{k+s} \bar{g}_{ij}^a}}{\displaystyle{\partial \bar{x}_{i_{1}}}...\displaystyle{\partial \bar{x}_{i_{k}}} \displaystyle{\partial \bar{t}^{s}}}(p_{0},t_{0})$ and $\frac{\displaystyle{\partial^{k+s} \bar{g}_{ij}^b}}{\displaystyle{\partial \bar{y}_{i_{1}}}...\displaystyle{\partial \bar{y}_{i_{k}}} \displaystyle{\partial \bar{t}^{s}}}(p_{0},t_{0})$.
\end{itemize}
Then we must have \[\lim_{\alpha\to \infty} \overline{X}_{i,j,i_{1}....i_{k},s} =0^k 2^s \delta^i_j\]  and  \[\lim_{\alpha\to \infty} \overline{Y}_{i,j,i_{1}....i_{k},s} =0^k 2^s \delta^i_j\] uniformly in $(p_0, t_0)$ (where $0^0 = 1$ and $0^k =0$ if $k\geq 1$). We shall first complete the proof of this theorem based on Claim A. Then we shall verify Claim A. Choose an orthonormal basis ${v_1, v_2}$ for the $2$-plane $P$, and write $v_{i}=\sum_{j} a_{ij} \frac{\displaystyle{\partial}}{\displaystyle{\partial \bar{x}_{j}}}+a_{in} \frac{\displaystyle{\partial}}{\displaystyle{\partial \bar{t}}}$. Set $\widehat{v}_{i}=\sum_{j}a_{ij} \frac{\displaystyle{\partial}}{\displaystyle{\partial \bar{y}_{j}}}+a_{in} \frac{\displaystyle{\partial}}{\displaystyle{\partial \bar{t}}}$ and let $\widehat{P}$ denote the $2$-plane spanned by ${\widehat{v}_{1},\widehat{v}_{2}}$. Note that it follows from Claim A(a) and A(b) and from the classical relation between the coefficients of the curvature tensor and of the first fundamental 
form \cite{Hic65} that $K_{\overline{A}}(\widehat{P})$ is a polynomial in the ${\overline{X}_{i,j,i_{1}....i_{k},s},a_{ij}}$ and that  $K_{\overline{B}}(P)$ is the same polynomial in the  ${\overline{Y}_{i,j,i_{1}....i_{k},s}, a_{ij}}$. Thus, by Claim 7(c), we have that \[\lim_{\alpha\to \infty}(K_{\overline{B}}(P)-(K_{\overline{A}}(\widehat{P})))= 0\] uniformly in $P$. Since $K_{\overline{A}}(\widehat{P})= -1$, \[\lim_{\alpha\to\infty} K_{\overline{B}}(P) =-1\] uniformly in $P$.
It remains to construct the coordinates $(\bar{x}_{1},\bar{x}_{2},....,\bar{x}_{n-1})$ and $(\bar{y}_{1},\bar{y}_{2},....,\bar{y}_{n-1})$ and $\bar{t}$ and to verify Claim A. Towards this end, we choose normal  coordinates $(x_{1},x_{2},....,x_{n-1})$ and $(y_{1},y_{2},....,y_{n-1})$ for $\mathbb{S}^{n-1}$ near $p_{0}$ with respect to the metric $A|_{\mathbb{S}^{n-1}\times t_{0}}$ and $B|_{\mathbb{S}^{n-1}\times t_{0}}$ respectively so that the following hold :\\
$\mathbf{Property~~ 5 :}$
\begin{itemize}
\item[(a)]$A( ~, )=g_{ij}^a dx_{i} dx_{j}+dt^2$ ; $B(~  , )=g_{ij}^b dy_{i} dy_{j}+dt^2$.
\item[(b)] $g_{ij}^a(p_{0},t_{0})=g_{ij}^b(p_{0},t_{0})$.
\item[(c)]There is a number $C>0$, which is independent of $(p_{0},t_{0})$, such that for all non negative integers $k$, $s$ satisfying $k + s\leq 2$, the following must hold:\\~\\
$\left |{\frac{\displaystyle{\partial}^{k+s} g_{ij}^a}{\displaystyle{\partial x_{i_{1}}}...\displaystyle{\partial x_{i_{k}}} \displaystyle{\partial t^s}}(p_{0},t_{0})}\right |<C$ and $\left |{\frac{\displaystyle{\partial}^{k+s} g_{ij}^b}{\displaystyle{\partial y_{i_{1}}}...\displaystyle{\partial y_{i_{k}}} \displaystyle{\partial t^s}}(p_{0},t_{0})}\right |<C$.
\end{itemize}
Now define the coordinates $\bar{x}_{i}$, $\bar{y}_{i}$ and $\bar{t}$ as follows :\\
$\mathbf{Property~~ 6 :}$
\begin{itemize}
\item[(a)]$\bar{t}= \alpha t$, $\bar{t}_{0}= \alpha t_{0}$, $\bar{x}_{i}=\displaystyle{\sinh}(\bar{t}_{0})x_{i}$, $\bar{y}_{i}=\displaystyle{\sinh}(\bar{t}_{0})y_{i}$.\\ Then we also have the following equalities :
\item[(b)] $d\bar{t}= \alpha dt$, $d\bar{x}_{i}=\displaystyle{\sinh}(\bar{t}_{0})dx_{i}$, $d\bar{y}_{i}=\displaystyle{\sinh}(\bar{t}_{0})dy_{i}.$
\item[(c)] $\frac{\displaystyle{\partial}}{\displaystyle{\partial\bar{t}}}=\frac{\displaystyle{1}}{\displaystyle{\alpha}}\frac{\displaystyle{\partial}}{\displaystyle{\partial t}}$, $\frac{\displaystyle{\partial}} {\displaystyle{\partial\bar{x}_{i}}}=\frac{\displaystyle{1}}{\displaystyle{\sinh}(\bar{t}_{0})}\frac{\displaystyle{\partial}}{\displaystyle{\partial x_{i}}}$, $\frac{\displaystyle{\partial}}{\displaystyle{\partial\bar{y}_{i}}}=\frac{1}{\displaystyle{\sinh}(\bar{t}_{0})}\frac{\displaystyle{\partial}}{\displaystyle{\partial y_{i}}}.$
\end{itemize}
It follows from Property 5(a), 5(b) and from Property 6(a), 6(b) that Claim A(a), A(b) are satisfied and that the $\bar{g}_{ij}^a$, $\bar{g}_{ij}^b$ of Claim A(a) can be computed in terms of the $g_{ij}^a$, $g_{ij}^b$ of Property 5(a) as follows: \\
$\mathbf{Property~~ 7 :}$
~~~~~~$\bar{g}_{ij}^a=\frac{\displaystyle{\sinh}(\bar{t})}{\displaystyle{\sinh}(\bar{t}_{0})}^{2}g_{ij}^a$,  $\bar{g}_{ij}^b=\frac{\displaystyle{\sinh}(\bar{t})}{\displaystyle{\sinh}(\bar{t}_{0})}^{2}g_{ij}^b$.\\
Finally, note that it follows from Property 5(c), Property 6(c) and Property 7 that Claim A(c) is satisfied. This completes the proof of Lemma \ref{pinched}.
\end{proof}
\paragraph{Proof of Theorem \ref{hype.metric} :}
Let $E_{+}^{n}$, $E_{-}^{n}$ denote the northern and southern hemi-spheres of the unit $n$-sphere $\mathbb{S}^n$. Each exotic sphere $\Sigma_{i}$ can be constructed by gluing $E_{+}^{n}$ to $E_{-}^{n}$ along a diffeomorphism $f_{i}: \partial E_{+}^{n} \to \partial E_{-}^{n}$. Since $\partial E_{+}^{n}=\mathbb{S}^n$, it follows from Property 2 that $\widehat{M}\#\Sigma_{i}$ can be constructed by gluing $\mathbb{S}^{n-1}\times [1,2]$ to $\widehat{M}\setminus h(\mathbb{S}^{n-1}\times (1,2))$ along the maps $h:\mathbb{S}^{n-1}\times 1\to \widehat{M}\setminus h(\mathbb{S}^{n-1}\times (1,2))$ and $h\circ \bar{f_{i}}:\mathbb{S}^{n-1}\times 2\to \widehat{M}\setminus h(\mathbb{S}^{n-1}\times (1,2))$, where $\bar{f_{i}}(x, 2)=(f_{i}(x), 2)$. Choose a metric $B(  , )$ on $\mathbb{S}^{n-1}\times [1,2]$ which satisfies Property 3 and the following properties:
Let $B_1( , )$ be constructed from $B( , )$ as in Property 4.\\
$\mathbf{Property~~ 8 :}$
\begin{itemize}
\item[(a)]$B_1( , )|_{\mathbb{S}^{n-1}\times 1}=A_1( , )|_{\mathbb{S}^{n-1}\times 1}$.
\item[(b)]$B_1( , )|_{\mathbb{S}^{n-1}\times 2}$ equals the pull back along $\bar{f}_{i}:\mathbb{S}^{n-1}\times 2\to \mathbb{S}^{n-1}\times 2$ of $A_1( , )|_{\mathbb{S}^{n-1}\times 2}$.
\item[(c)] $B_1$ is constant in $t$ near $t=1$, $2$. We define a metric $\left \langle ,\right \rangle _{i}$ on $\widehat{M}\#\Sigma_{i}$ as follows:
\end{itemize}
$\mathbf{Property~~ 9 :}$
\begin{itemize}
\item[(a)]$\left \langle ,\right \rangle_{i}|_{\widehat{M}\setminus h(\mathbb{S}^{n-1}\times (1,2))} = \left \langle ,\right \rangle_{\widehat{M}}|_{\widehat{M}\setminus h(\mathbb{S}^{n-1}\times (1,2))}$.
\item[(b)]$\left \langle ,\right \rangle_{i}|_{\mathbb{S}^{n-1}\times [1,2]=B_1}$, where $B_{1}$ is constructed from $B$ as in Property 8. It follows from Property 2, 8, and 9 that $\left \langle ,\right \rangle_{i}$ is well defined. By Property 9 and Theorem \ref{pinched},  \[\lim_{\alpha\to \infty}K_{\left \langle,\right \rangle_{i}}(P)=-1\] uniformly in $P$, where $P$ is any $2$-sphere tangent to $\widehat{M}\#\Sigma_{i}$ and where $K_{\left \langle ,\right \rangle_{i}}(P)$ is the sectional curvature of $P$ with respect to $\left \langle ,\right \rangle_{i}$. This completes the proof of Theorem \ref{hype.metric}.
\end{itemize}
\begin{remark}\rm{
Since there are no exotic spheres in dimensions $< 7$ this does not give counterexamples to Lawson-Yau Conjecture in dimensions less than 7. (Also note that, for example, there are no exotic 12-dimensional spheres.) Moreover, since the Diff category is equivalent to the PL category in dimensions less than 7, changing the differentiable structure is equivalent to changing the PL structure.  Hence for dimensions $< 7$ the Smooth Rigidity Problem \ref{prorig} is equivalent to the following PL version :}
\end{remark}
\begin{problem}\label{prorigpl}
Let $f:N\to M$ denote a homotopy equivalence between closed smooth manifolds. Is $f$ homotopic to a PL homeomorphism?
\end{problem}
\begin{remark}\rm{
For a general dimension $n$, a negative answer to Problem \ref{prorigpl} implies a negative answer for Problem \ref{prorig}, because diffeomorphic manifolds are PL-homeomorphic. The converse is not true in general, but, as mentioned before, it is true for dimensions $< 7$. For example an (smoothly) exotic sphere $\Sigma$ is not diffeomorphic to the corresponding sphere (by definition) but it is PL-homeomorphic to it, provided $\dim \Sigma \neq 4$. In fact, there are no PL-exotic spheres in any dimension $\neq 4$. It follows that $Z$ is PL-homeomorphic to $Z\#\Sigma$ for any manifold $Z$, and exotic sphere $\Sigma$, $\dim \Sigma\neq 4$. Therefore, Theorem \ref{hype.smooth} does not answer Problem \ref{prorigpl}. Note that, since diffeomorphism implies PL-homeomorphism and this in turn implies homeomorphism, we have that Problem \ref{prorigpl} for non-positively curved manifolds lies between the Topological Rigidity for negative curvature (which is true, by Theorem \ref{gentoprigd}) and the Smooth Rigidity (
which is false, by Theorem \ref{hype.smooth}). The Theorem \ref{hype.smooth} was generalized by Ontaneda in \cite{Ont94} to dimension 6, by changing the PL structure, which gives answer to Problem \ref{prorigpl} for non-positively curved manifolds to be negative. Also the constructions in \cite{Ont94} give counterexamples to Lawson-Yau Conjecture in dimension 6. Here is the result:}
\end{remark}
\begin{theorem}\label{plexo6}
There are closed real hyperbolic manifolds $M$ of dimension $6$ such that the following holds. Given $\epsilon >0$, $M$ has a finite cover $\widehat{M}$ that supports an exotic (smoothable) PL-structure that admits a Riemannian metric with all sectional curvatures in the interval $(-1 -\epsilon, -1 +\epsilon)$. 
\end{theorem}
\begin{remark}\rm{
\indent
\begin{itemize}
\item[\bf{1.}] An exotic (smoothable) PL-structure on $\widehat{M}$ means that there exists a closed smooth manifold $N$ such that $\widehat{M}$ and $N$ are not piecewise linearly homeomorphic; i.e., the underlying simplicial complex of any PL triangulation compatible with the given smooth structure of $\widehat{M}$ must be different from that of any PL triangulation compatible with the given smooth structure of $N$. In particular, $\widehat{M}$ and $N$ are not diffeomorphic. This gives counterexamples to Lawson-Yau Conjecture in dimension 6. On the other hand, the counterexamples to smooth-rigidity Problem \ref{prorig} for negatively curved manifolds given by Theorem \ref{hype.exto} are piecewise linearly homeomorphic.
\item[\bf{2.}] P. Ontaneda's construction builds on the ideas used in the construction of the counterexamples given by Theorem \ref{hype.exto}; but employs the Kirby-Siebenmann obstruction to PL-equivalence, which lies in $H^3(M,\mathbb{Z}_2)$ \cite{KS77}, instead of exotic spheres. By using such PL obstruction, P. Ontaneda proved the following theorem \cite{Ont94}:
\end{itemize} }
\end{remark}
\begin{theorem}\label{plexo6lem}
Consider the following data. For each $k=1,2,3,....$
we have closed hyperbolic manifolds $M_{0}(k)$, $M_{1}(k)$, $M_{2}(k)$, $M_{3}(k)$ such that the following hold :
\begin{itemize}
\item[(i)] $\dim M_{0}(k)=6$, $\dim M_{1}(k)=5$, $\dim M_{2}(k)=3$, $\dim M_{3}(k)=3$.
\item[(ii)] $M_{2}(k)\subset M_{1}(k)\subset M_{0}(k)$ and $M_{3}(k)\subset M_{0}(k)$. All the inclusions are totally geodesic.
\item[(iii)] $M_{2}(k)$ and  $M_{3}(k)$ intersect at one point transversally.
\item[(iv)] For each $k$ there is a finite covering map $p(k):M_{0}(k)\to M_{0}(1)$
such that $p(k)(M_{i}(k))=M_{i}(1)$, for $i=0,1,2,3$.
\item[(v)] $M_{1}(k)$ has a tubular neighborhood in $M_{0}(k)$ of width $r(k)$ and $r(k)\to \infty $ as $k\to \infty$.
\end{itemize}
Then, given $\epsilon>0$, there is a $K$ such that all $M_{0}(k)$, $k\geq K$, have
exotic (smoothable) triangulations admitting Riemannian metrics with all sectional curvatures in the interval  $(-1 -\epsilon , -1 +\epsilon)$.
\end{theorem}
\begin{remark}\rm{
\indent
\begin{itemize}
\item[\bf{1.}] To prove Theorem \ref{plexo6}, we have to show that there are manifolds satisfying the hypothesis of Theorem \ref{plexo6lem}. P. Ontaneda constructed such manifolds, for every $n\geq 4$, $M_{i}(k)$, $i=0,1,2,3$ and $k=1,2,3,...$ with $\dim M_{0}(k)=n$, $\dim M_{1}(k)=n-1$, $\dim M_{2}(k)=n-3$, $\dim M_{3}(k)=3$ satisfying $(ii)$, $(iii)$, $(iv)$ and $(v)$ of the Theorem \ref{plexo6lem}. When $n=6$ they also satisfied $(i)$ (\cite[p.15]{Ont94}). These manifolds are the ones that appear at the end of \cite{MR81} for the real hyperbolic case.
\item[\bf{2.}] The following Theorem \ref{plexa} showed that the answer to Question \ref{prorigpl} for non-positively curved manifolds to be negative for dimensions greater than 5:
\end{itemize} }
\end{remark} 
\begin{theorem}\label{plexa}
For $n\geq 6$, there are closed non-positively curved manifolds of dimension $n$ that support exotic (smoothable) PL structures admitting
Riemannian metrics with non-positive sectional curvatures.
\end{theorem}
\begin{remark}\rm{
\indent
\begin{itemize}
\item[\bf{1.}] Theorem \ref{plexa} follows by taking one of the examples of Theorem \ref{plexo6} and multiply it by $T^n$. To see this we note that if $(M,\tau_{0})$ and $(M,\tau_{1})$ are two non-positively curved triangulations on $M$, then $(M\times T^n,\tau_{0}\times \tau_{T^n})$ and $(M\times T^n,\tau_{1}\times \tau_{T^n})$ are also non-positively curved. Moreover, if $(M,\tau_{0})$ and $(M,\tau_{1})$ are non-concordant, then $(M\times T^n,\tau_{0}\times \tau_{T^n})$ and $(M\times T^n,\tau_{1}\times \tau_{T^n})$ are also so, for the Kunneth formula tells us that $\mathbb{Z}_{2}$-cohomology classes do not vanish when we take products. Finally we can prove these triangulations are not equivalent \cite{Ont94}.
\item[\bf{2.}] Theorem \ref{plexo6} was generalized by Farrell, Jones and Ontaneda in \cite{FJO98} for every dimension $> 5$. Here is the result:
\end{itemize} }
\end{remark}
\begin{theorem}\label{plexog}
There are closed real hyperbolic manifolds $M$ in every dimension $n$, $n>5$, such that the following holds. Given $\epsilon >0$, $M$ has a finite cover $\widehat{M}$ that supports an exotic (smoothable) PL structure that admits a Riemannian metric with all sectional curvatures in the interval $(-1 -\epsilon , -1 +\epsilon)$.
\end{theorem}
\begin{remark}\rm{
\indent
\begin{itemize}
\item[\bf{1.}] The counterexamples constructed in the proof of Theorem \ref{plexog} use the results of Millson and Raghunathan \cite{MR81}, based on previous work of Millson \cite{Mil76}.
\item[\bf{2.}] Theorem \ref{hype.smooth} and Theorem \ref{plexog} were the first in a sequence of results that shed some light on the relationship between the analysis, geometry and topology of negatively curved manifolds. These results showed certain limitations of well-known powerful analytic tools in geometry, such as the Harmonic Map technique, the Ricci flow technique, the Elliptic deformation technique as well as Besson-Courtois-Gallot's Natural Map technique \cite{BCG96}. A more complete exposition on this area and how it evolved in time can be found in the survey article \cite{FO04}.
\end{itemize} }
\end{remark}
\section{\large Smooth Rigidity for Finite Volume Real Hyperbolic and Complex Hyperbolic Manifolds}
In this section we give three different variants of Theorem \ref{hype.smooth} : first for non-compact finite volume complete hyperbolic manifolds, the second for negatively curved manifolds not homotopy equivalent to a closed locally symmetric space and the third for complex hyperbolic manifolds. Let us begin with the first.\\
A possible way to change smooth structure on a smooth manifold $M^n$, without changing its homeomorphism type, is to take its connected sum $M^n\#\Sigma^n$ with a homotopy sphere $\Sigma^n$. A surprising observation reported in \cite{FJ93a} is that connected sums can never change the smooth structure on a connected, non-compact smooth manifold:
\begin{theorem}\label{sumnochan}
Let $M^n$, $n\geq 5$ be a complete, connected and non-compact smooth manifold. Let the homotopy sphere $\Sigma^n$ represent an element in $\Theta_n$. Then $M^n\#\Sigma^n$ is diffeomorphic to $M^n$.
\end{theorem}
\begin{proof}
We will denote the concordance class in $\mathcal{C}(M)$ of $M^n\#\Sigma$ by $[M^n\#\Sigma]$.(Note that $[M^n\#\mathbb{S}^n]$ is the class of  $M^n$.) Let $f_{M}:M^n\to \mathbb{S}^n $  be a degree one map and note that $f_{M}$ is well-defined up to homotopy. Composition with $f_{M}$ defines a homomorphism $f_{M}^*:[\mathbb{S}^n, Top/O]\to [M^n ,Top/O]$. And in terms of the identifications 
$\Theta_n=[\mathbb{S}^n, Top/O]$ and $\mathcal{C}(M^n)=[M^n ,Top/O]$ given by Theorem \ref{kirby}, $f_{M}^*$ becomes $[\Sigma^n]\to [M^n\#\Sigma^n]$. But every map $M^n\to  \mathbb{S}^n $ is homotopic to a constant map. Therefore $M^n\#\Sigma^n$ is concordant to $M^n$ and hence diffeomorphic to $M^n$.
\end{proof}
\begin{remark}\rm{
Note that taking the connected sum of a non-compact manifold $M$ with an exotic sphere can never change the differential structure of $M$. Therefore we do not have an exact analogue of Theorem \ref{hype.smooth} for the finite volume non-compact case. F.T. Farrell and L.E. Jones used a different method in \cite{FJ93a}, which can sometimes change the smooth structure on a non-compact manifold $M^n$. The method is to remove an embedded tube $\mathbb{S}^1\times \mathbb{D}^{n-1}$ from $M^n$ and then reinsert it with a twist. More precisely, the method is as follows : }
\end{remark}
\begin{definition}(Dehn surgery method)\rm{
Pick a  smooth embedding $f:\mathbb{S}^1\times \mathbb{D}^{n-1}\to M^n$ and an orientation preserving diffeomorphism $\phi:\mathbb{S}^{n-2}\to \mathbb{S}^{n-2}$. Then a new smooth manifold $M_{f,\phi}$ is obtained as the quotient space of the disjoint union
\begin{center}
 $\mathbb{S}^1\times \mathbb{D}^{n-1} \sqcup M^n\setminus f(\mathbb{S}^1\times int(\mathbb{D}^{n-1}))$,
\end{center}
where we identify points
\begin{center}
$(x,v)$ and $f(x,\phi(v))$ if $(x,v)\in \mathbb{S}^1\times \mathbb{S}^{n-2}$.
\end{center}}
\end{definition}
The smooth manifold $M_{f,\phi}$ is canonically homeomorphic to $M^n$ but is not always diffeomorphic to $M^n$. F.T. Farrell and L.E. Jones \cite{FJ93a} proved the following result in this way:
\begin{theorem}\label{0.1}\cite{FJ93a}
Let $n$ be any integer such that $\Theta_{n-1}\neq 1$, and $\displaystyle{\epsilon}$ be any positive real number. Then there exists an $n$-dimensional complete Riemannian manifold $M^n$ with finite volume and all its sectional curvatures contained in the interval $[-1-\displaystyle{\epsilon}, -1+\displaystyle{\epsilon}]$, and satisfying the following: $M^n$  is not diffeomorphic to any complete Riemannian locally symmetric space; but it is homeomorphic to $\mathbb{R}\textbf{H}^n/\Gamma $ where $\Gamma$ is a torsion free non uniform lattices in Iso$(\mathbb{R}\textbf{H}^n)$.
\end{theorem}
\begin{remark}\rm{
\indent
\begin{itemize}
\item[\bf{1.}] Theorem \ref{0.1} showed that the answer to Problem \ref{prorig} for finite volume but non-compact Riemannian locally symmetric spaces to be negative. 
\item[\bf{2.}] Recall that Kervaire and Milnor \cite{KM63} and Browder \cite{Bro69} showed that $\Theta_{2n-1}$ is non trivial for every integer $n>2$ which does not have the form $n=2^{i}-1$. On the other hand $\Theta_{12}$ is trivial. Therefore if $M^{12}$ is a closed real hyperbolic manifold, $M^{12}\#\Sigma^{12}$ must be diffeomorphic to  $M^{12}$ where $\Sigma^{12}$ is any homotopy $12$-sphere. The Theorem \ref{hype.smooth} consequently fails to yield an exotic smooth structure on a negatively curved $12$-manifold which is homeomorphic to a real hyperbolic manifold. But Dehn surgery method does. In particular, F.T. Farrell and L.E. Jones \cite{FJ93a} proved the following result:
\end{itemize} }
\end{remark}
\begin{theorem}\label{Theorem 0.2}\cite{FJ93a}
Let $n$ be any integer such that either $\Theta_{n}$ or $\Theta_{n-1}$ is not trivial. Then there exists a closed Riemannian manifold $M^n$ whose sectional curvatures are all pinched within $\displaystyle{\epsilon}$  of $-1$ and such that
\begin{itemize}
\item[(i)] $\dim M^n=n$.
\item[(ii)] $M^n$ is homeomorphic to a real hyperbolic manifold.
\item[(iii)] $M^n$ is not diffeomorphic to any Riemannian locally symmetric space.
\end{itemize}
\end{theorem}
We shall now need the following definitions and results to prove Theorem \ref{0.1}:
\begin{definition}\rm{
A tube $f:\mathbb{S}^1\times \mathbb{D}^{n-1}\to M^n$ determines a framed simple closed curve 
$\alpha:\mathbb{S}^1\to M^n$ where $\alpha(y)=f(y,0)$ for each $y\in \mathbb{S}^1$. The framing of $\alpha$ consists of the vector fields $X_{1}$, $X_{2}$,..., $X_{n-1}$ where $X_{i}(y)$ is the vector tangent to the curve at $t\to f(y, te_{i})$ at $t=0$. Here $e_{i}$ denotes the point in $\mathbb{R}^{n-1}$ whose $i^{th}$ coordinate is 1 and all other coordinates are 0. We use $\alpha$ to denote the curve equipped with this framing. It is called the core of $f$. The concordance class of $M_{f,\phi}$ depends only on $M^n$, the core $\alpha$ of $f$, and the (pseudo)-isotopy class of $\phi$ denoted by $x$. We consequently denote the concordance class of $M_{f, \phi}$ by $\mathcal{M}(\alpha, x)$. Recall that the isotopy classes of orientation-preserving diffeomorphisms of $\mathbb{S}^{n-2}$ are in one-one correspondence with the elements in the abelian group $\Theta_{n-1}$ which is also identified with $\pi_{n-1}(Top/O)$; therefore, $x\in \Theta_{n-1}$.\\
Let  $\widehat{\alpha}:M^n\to \mathbb{S}^{n-1}$ be the result of applying the Pontryagin-Thom construction to the framed $1$-manifold $\alpha$. It is explicitly described by $\widehat{\alpha}(f(y,v))=q(v)$ where $(y, v)\in \mathbb{S}^1\times \mathbb{D}^{n-1}$, and $q: \mathbb{D}^{n-1}\to \mathbb{D}^{n-1}/\partial\mathbb{D}^{n-1}=\mathbb{S}^{n-1}$ is the canonical quotient map if $y\notin$ image $f$, then $\widehat{\alpha}(y)=q(\partial\mathbb{D}^{n-1})$.}
\end{definition}
\begin{definition} \rm{
Let $M^n$ be a complete (connected) Riemannian manifold with finite volume and whose all sectional curvatures $-1$.
We say that an element $\gamma \in \pi_{1}(M)$ is cuspidal if there are arbitrarily short closed curves in $M^n$ which are freely homotopic to a curve representing $\gamma$.}
\end{definition}
\begin{definition}\rm{
Let $M^n$ be a complete (connected) Riemannian manifold with finite volume and whose all sectional curvatures $-1$.
A closed geodesic $\gamma:\mathbb{S}^1\to M^n$ is said to be $t$-simple if $\dot{\gamma}:\mathbb{S}^1\to TM$ is simple, i.e., a one-to-one function.}
\end{definition}
\begin{theorem}\label{hype.finite.volu}\cite{FJ93a} \rm{
Let $M^n$, with $n\ge 2$, be a complete (connected) Riemannian manifold with finite volume and all sectional curvatures $-1$. Let $N^n$ be a complete Riemannian locally symmetric space. If $M$ and $N$ are homeomorphic, then they are isometrically equivalent (after rescaling the metric on $N$ by a positive constant).}
\end{theorem}
\begin{theorem}\label{simple}\cite{FJ93a}\rm{
Let $M^n$  be a complete (connected) Riemannian manifold with finite volume and all sectional curvatures $-1$. Assume $M^n$ is orientable and $\phi:\pi_{1}(M)\to \mathbb{Z}$ is an epimorphism. Then there is a $t$-simple closed geodesic $\gamma:\mathbb{S}^1\to M^n$ such that $\phi([\gamma])\neq 0$ where $[\gamma]$ denotes the free homotopy class of $\gamma$.}
\end{theorem}
\begin{proposition}\label{homo}\cite{FJ93a}
Given a semi simple element $A\in SO^{+}(n,1,\mathbb{Z})$ of infinite  order and a positive integer $m$, there exists a finite group $G$ and a homomorphism $\chi:SO^{+}(n,1,\mathbb{Z})\to G$ with the following properties:
\begin{itemize}
\item[(i)] The order of $\chi(A)$ is divisible by $m$.
\item[(ii)] Let $\beta$ be any unipotent element in $SO^{+}(n,1,\mathbb{Z})$ such that  $\chi(B)= \chi(A)^s$ where $s\in \mathbb{Z}$, then $m$ divides $s$.
\end{itemize}
\end{proposition}
\begin{theorem}\label{non uniform}\cite{FJ93a}
Assume that $M^n$, with $n\ge 6$, is closed (connected) Riemannian manifold with all sectional curvatures $-1$ and has positive first Betti number. Given $\displaystyle{\epsilon}>0$ and an infinite order element $y\in H_{1}(M^n,\mathbb{Z})$, there exist a (connected) finite sheeted covering space $P:\mathcal{M}^n\to M^n$ and a simple framed geodesic $\alpha$ in $\mathcal{M}^n$ with the following properties:
\begin{itemize}
\item[(i)] Some multiple of the homology class represented by $\alpha$ maps to a non zero multiple of $y$ via  $P_{*}: H_{1}(\mathcal{M}^n,\mathbb{Z})\to H_{1}(M^n,\mathbb{Z})$.
\item[(ii)] There is no manifold diffeomorphic to $N^n \#\Sigma^n$ in the concordance class $\mathcal{M}(\alpha,x)$ provided $N^n$ is a Riemannian locally symmetric space, $\Sigma^n$ represents an element in $\Theta_{n}$ and $x$ is a non zero element in $\Theta_{n-1}$.
\item[(iii)] Each concordance class $\mathcal{M}(\alpha,x)$ contains a complete and finite volume Riemannian manifold whose sectional curvatures are all in the interval $[-1-\displaystyle{\epsilon}, -1+\displaystyle{\epsilon}]$.
\end{itemize}
\end{theorem}
\begin{theorem}\label{Adden1}\cite{FJ93a}
Let $M^n$, with $n\ge 6$, be complete (connected) Riemannian manifold with all sectional curvatures $-1$.
Assume that $M^n$ is a $\pi$- manifold (non necessarily compact), $\gamma$ is a $t$-simple framed closed geodesic in $M^n$ and $\lambda:\pi_{1}(M)\to \mathbb{Z}$ is a homomorphism such that
\begin{itemize}
\item[(i)] $\lambda([\gamma])=1$ where $[\gamma]$ denotes the free homotopy class of $\gamma$ and
\item[(ii)]$\lambda(\beta)$ is divisible by the order of  $\Theta_{n-1}$ for each cuspidal element $\beta$ in $\pi_{1}(M^n)$.
Given a positive real number $\displaystyle{\epsilon}$, there exist a (connected) finite sheeted covering space  $P:\mathcal{M}^n\to M^n$ and a simple framed geodesic $\alpha$ in $\mathcal{M}^n$ such that the composite $p\circ \alpha=\gamma$.
\item[(iii)]There is no manifold diffeomorphic to $N^n \#\Sigma^n$ in the concordance class $\mathcal{M}(\alpha,x)$ provided $N^n$ is a Riemannian locally symmetric space, $\Sigma^n$ represents an element in $\Theta_{n}$ and $x$ is a non zero element in $\Theta_{n-1}$.
\item[(iv)] Each concordance class $\mathcal{M}(\alpha,x)$ contains a complete and finite volume Riemannian manifold whose sectional curvatures are all in the interval $[-1-\displaystyle{\epsilon}, -1+\displaystyle{\epsilon}]$.
\end{itemize}
\end{theorem}
\begin{theorem}\label{millson}\cite{Mil76}
For each integer $n>1$, there exist two complete (connected) finite volume Riemannian manifolds $K^n$ and $N^n$ of dimension $n$ which satisfy the following properties:
\begin{itemize}
\item[(i)] All the sectional curvatures of both $K^n$ and $N^n$ are $-1$.
\item[(ii)]Both $K^n$ and $N^n$ have positive first Betti number.
\item[(iii)] $K^n$ is compact.
\item[(iv)] $N^n$  is not compact.
\item[(v)]$\pi_{1}(N^n)$ is isomorphic to a finite index subgroup of  $SO^{+}(n,1,\mathbb{Z})$.
\end{itemize}
\end{theorem}
\paragraph{Proof of Theorem \ref{0.1} :}
Define two sequences of positive integers $a_m$ and  $b_m$  as follows: Let  $a_m$ be the order of the finite group $\Theta_{m}$ and let $b_m$ be the least common multiple of the orders of the holonomy groups of lattices in the lie group of all rigid motions of Euclidean $m$-dimensional space. Bieberbach \cite{Bie10} showed that $b_m$ exists and divides the order of the finite group $GL_m(\mathbb{Z}_3)$ because of Minkowski's theorem \cite{Min87}. Let $N^n$ be the Millson manifold in Theorem \ref{millson}. Because of Theorem \ref{sullivan}, there is a finite sheeted (connected) covering space $P:\mathcal{N}^n\to N^n$ such that every covering space of $\mathcal{N}^n$ is a $\pi$-manifold. There is an epimorphism $\phi:\pi_{1}(\mathcal{N}^n)\to \mathbb{Z}$ since $N^n$ has positive first Betti number. Because of Theorem \ref{simple}, there is a $t$-simple framed closed geodesic $w$ in $\mathcal{N}^n$ such that  $\phi([w])\neq 0$ where $[w]$ denotes the fundamental group element corresponding to $w$. (Note that $[w]$ is 
well 
defined upto conjugacy). Let $A\in SO^{+}(n,1,\mathbb{Z})$ denote the semi simple matrix corresponding to $[w]$ under an identification of $\pi_{1}(\mathcal{N}^n)$ with a subgroup of $SO^{+}(n,1,\mathbb{Z})$. Let $\chi: SO^{+}(n,1,\mathbb{Z})\to G$ be a homomorphism satisfying the conclusions of Theorem \ref{homo} relative to $A$ and $n=a_{n-1}b_{n-1}$. Note that $g^b_{n-1}$ is unipotent for every cuspidal element $g\in \pi_{1}(\mathcal{N}^n)$. Hence conclusion (2) of  Theorem \ref{homo} yields the following fact:
\begin{itemize}
\item[(i)]if $\chi(B)=\chi(A)^s$ where $B$ is a cuspidal element of $\pi_{1}(\mathcal{N}^n)$, then $a_{n-1}$ divides $s$.
\end{itemize}
Consider the homomorphism  $\phi\times \chi:\pi_{1}(\mathcal{N}^n)\to \mathbb{Z}\times G$ and let $\mathfrak{B}$  denote the infinite cyclic subgroup generated by  $\phi\times \chi(A)=(\phi(A),\chi(A))$. Let $q:M^n\to \mathcal{N}^n$  be the covering space corresponding to $(\phi\times \chi)^{-1}_{\#}(\mathfrak{B})$ and $\lambda:\pi_{1}(M^n)\to \mathbb{Z}$ be the composite of  $q_{\#}$,
$\phi\times \chi$ and the identification of $\mathfrak{B}$ with $\mathbb{Z}$ determined by making $(\phi(A),\chi(A))$ correspond to $1\in \mathbb{Z}$. Let $\Gamma$ be a lift of $w$ to $M^n$. Since the conditions of Theorem \ref{Adden1} are clearly satisfied, the examples posited in this theorem  can now be drawn from the conclusions of Theorem \ref{Adden1}.

\paragraph{Proof of Theorem \ref{Theorem 0.2} :}
When $\Theta_{n}$ is non trivial, this result follows from Theorem \ref{hype.smooth} and Theorem \ref{hype.finite.volu}. When $\Theta_{n-1}$ is non trivial, it follows from Theorem \ref{non uniform} by setting  $M^n$ (in Theorem \ref{non uniform}) equal to the  manifold $K^n$ of Theorem \ref{millson}.
\begin{remark}\rm{
\indent
\begin{itemize}
\item[\bf{1.}] A given topological manifold $M$ can often support many distinct smooth structures. F.T. Farrell and L.E. Jones constructed in \cite{FJ94a} examples of topological manifolds $M$ supporting at least two distinct smooth structures $\mathcal{M}_1$ and $\mathcal{M}_2$, 
where $\mathcal{M}_1$ is a complete, finite volume, real hyperbolic manifold, while $\mathcal{M}_2$ cannot support a complete, finite volume, pinched negatively curved Riemannian metric. This paper \cite{FJ94a} supplements the results of papers \cite{FJ89a} and \cite{FJ93a} where the opposite phenomenon was studied. Namely, in these earlier papers examples of distinct smoothings $\mathcal{M}_1$ and $\mathcal{M}_2$ were constructed, where $\mathcal{M}_1$ is as above but $\mathcal{M}_2$ also supports a complete, finite volume, pinched negatively
curved Riemannian metric. Compact examples were constructed in \cite{FJ89a} (see Theorem \ref{hype.smooth}) and noncompact examples in \cite{FJ93a} (see Theorem \ref{0.1}).
\item[\bf{2.}] Since a finite volume pinched negatively curved metric on a manifold induce an infranil structure on the cross section of the cusps, we have that the Smooth Rigidity Theorem \ref{infranil} of infranilmanifolds can be used to prove the following result of Farrell and Jones \cite{FJ94a} :
\end{itemize}}
\end{remark}
\begin{theorem}\label{nonpinch}
Let $n>5$ such that $\Theta_{n-1}$ is non trivial. Then there exists a connected smooth manifold $N^n$ such that
\begin{itemize}
\item[(i)] $N$ is homeomorphic to a complete, non-compact, finite volume real hyperbolic manifold.
\item[(ii)] $N$ does not admit a finite volume complete pinched negatively curved Riemannian metric.
\end{itemize}
\end{theorem}
We shall now need the following definitions and results to prove Theorem \ref{nonpinch}
\begin{definition}\rm{
Let $M^n$ be a connected smooth manifold (with $n > 5$) and $f:\mathbb{R}\times \mathbb{D}^{n-1}\to M^n$ a smooth embedding which is also a proper map. We call $f$ a proper tube. Let $\phi:\mathbb{S}^{n-2}\to \mathbb{S}^{n-2}$ be an orientation-preserving diffeomorphism of the sphere $\mathbb{S}^{n-2}$ and identify $\mathbb{S}^{n-2}$ with $\mathbb{D}^{n-1}$. Then a new smooth manifold $M^{n}_{f,\phi}$ is obtained as a quotient space of the disjoint union
\begin{center}
 $\mathbb{R}\times \mathbb{D}^{n-1} \sqcup M^n\setminus f(\mathbb{R}\times int(\mathbb{D}^{n-1}))$,
\end{center}
where we identify points
\begin{center}
$(t, x)$ and $f(t, \phi(x))$ if $(t, x)\in \mathbb{R}\times \mathbb{S}^{n-2}$.
\end{center}
Recall that the isotopy class $[\phi]$ is identified with an element in $\Theta_{n-1}$.}
\end{definition}
\begin{definition}\rm{
Let $f:\mathbb{R}\times \mathbb{D}^{n-1}\to M^n$ be a proper tube. We say that it connects different ends of $M^n$ if there exist a compact subset $K$ in $M^n$ and a positive real number $r$ such that $f([r, +\infty)\times \mathbb{D}^{n-1})$ and $f((-\infty, -r]\times \mathbb{D}^{n-1})$ lie in different components of $M^n\setminus K$.}
\end{definition}
\begin{theorem}\label{corr.nonpinch}
Let $M^n$ $(n > 5)$ be a parallelizable complete real hyperbolic manifold with finite volume and $f:\mathbb{R}\times \mathbb{D}^{n-1}\to M^n$ a proper tube which connects different ends of $M^n$. If the diffeomorphism $\phi:\mathbb{S}^{n-2}\to \mathbb{S}^{n-2}$ represents a nontrivial element of $\Theta_{n-1}$, then the smooth manifold $M^{n}_{f,\phi}$ is
\begin{itemize}
\item[(i)] homeomorphic to $M^n$ and
\item[(ii)] does not support a complete pinched negatively curved Riemannian metric with finite volume.
\end{itemize}
\end{theorem}
\paragraph{Proof of Theorem \ref{nonpinch} assuming Theorem \ref{corr.nonpinch}:}
Let $\mathbb{Z}$ denote the additive group of integers. The argument in the proof of Theorem \ref{0.1} is easily modified to yield a connected, complete, non-compact, parallelizable, real hyperbolic manifold $M^n$ and an
epimorphism $\lambda: \pi_{1}M^n\to \mathbb{Z}$ such that $\lambda(\beta)$ is divisible by $2$ for every cuspidal element $\beta$ in $\pi_{1}M^n$. 
Let $\mathcal{N}^n$ be the finite sheeted covering space of $M^n$ corresponding to the subgroup $\lambda^{-1}(2\mathbb{Z})$ of $\pi_{1}M^n$.
Then $\mathcal{N}^n$ has twice as many cusps as $M^n$. Consequently, $\mathcal{N}^n$ is connected and has at least two distinct cusps. Therefore, we can construct a
proper tube $f:\mathbb{R}\times \mathbb{D}^{n-1}$ connecting different ends of $\mathcal{N}^n$. Let $N^n$ be $\mathcal{N}^{n}_{f,\phi}$ where $\phi$ represents a nontrivial element of $\Theta_{n-1}$.
Now applying Theorem \ref{corr.nonpinch}, in whose statement $M^n$ is replaced by $\mathcal{N}^n$, we see that $N^n$ satisfies the conclusions of Theorem \ref{nonpinch}.\\

The proof of Theorem \ref{corr.nonpinch} requires some preliminary results :
\begin{lemma}\label{lem.nonpinch}\cite{FJ94a}
Let $N^n$ $(n > 4)$ be a closed infranilmanifold which is also a $\pi$-manifold.
Let $\Sigma^n$ be a homotopy sphere which is not diffeomorphic to $\mathbb{S}^{n}$. Then
the connected sum $N^n\#\Sigma^n$ is not diffeomorphic to any infranilmanifold.
\end{lemma}
\begin{proof}
By Theorem \ref{infranil}, it is sufficient to show that $N^n\#\Sigma^n$ is not diffeomorphic to $N^n$. Recall from Theorem \ref{kirby} that the concordance classes of smooth structures on (the topological manifold) $N^n$ can be put in bijective correspondence with the homotopy classes of maps from $N^n$ to $Top/O$ denoted $[N^n, Top/O]$ with the infranilmanifold structure corresponding to the constant map. Let $\gamma :N^n\to \mathbb{S}^{n}$ be a degree one map. It induces a map
$$ \gamma^{*} : \Theta_n=[\mathbb{S}^{n}, Top/O]\to [N^n, Top/O],$$ and it can be shown that $\gamma^{*}(\Sigma^n)$ is the concordance class of $N^n\#\Sigma^n$. Since $N^n$
is a $\pi$-manifold, the argument given in Theorem \ref{hype.exto} due to Browder \cite{Bro65} and Brumfiel \cite{Bru71} applies to show that  $\gamma^{*}$ is monic.
Hence $N^n\#\Sigma^n$ is not concordant to $N^n$. Then a slight modification of the argument given to prove Addendum \ref{Addendum 2.3} shows that $N^n\#\Sigma^n$ is not diffeomorphic to $N^n$. The modification consists of using Theorem \ref{infranil} in place of Mostow's Rigidity Theorem \ref{mostow} and  Theorem \ref{nilrig} in place of Farrell and Jones Topological Rigidity Theorem \ref{gentoprigd}. This completes the proof of Lemma \ref{lem.nonpinch}.
\end{proof}
By using Theorem \ref{lem.nonpinch}, F.T. Farrell and A. Gogolev recently proved the following result \cite{FG12}:
\begin{theorem}
Let $M$ be an $n$-dimensional ($n\not =4$) orientable infranilmanifold with a $q$-sheeted cover $N$ which is a nilmanifold. Let $\Sigma$ be an exotic homotopy sphere of order $d$ from $\Theta_n$. Then $M\#\Sigma$ is not diffeomorphic to any infranilmanifold if $d$ does not divide $q$. In particular, $M\#\Sigma$ is not diffeomorphic to $M$ if $d$ does not divide $q$.
\end{theorem}
The following final preliminary result for proving Theorem \ref{corr.nonpinch} strings together facts contained in \cite{BK81} and \cite{Ruh82}:
\begin{lemma}\label{lem1.nonpinch}\cite{FJ94a}
Let $M^n$ be a connected complete pinched negatively curved Riemannian manifold with finite volume. Then there exists a compact smooth manifold
$\overline{M}^n$ such that
\begin{itemize}
\item[(i)] the interior of $\overline{M}^n$ is diffeomorphic to $M^n$ and
\item[(ii)]the boundary of $\overline{M}^n$ is the disjoint union of a finite number of infranilmanifolds.
\end{itemize}
\end{lemma}
\begin{remark}\label{lem1.rem}\rm{
Almost flat Riemannian manifolds are defined and investigated by Gromov in \cite{Gro78a}. Buser and Karcher \cite{BK81} showed that there exists a
compact smooth manifold $M$ satisfying
\begin{itemize}
\item[(i)] $int(\overline{M}^n)= M^n$ and
\item[(ii)] each component of $\partial \overline{M}^n$ is an almost flat Riemannian manifold.
\end{itemize}
Their construction is a consequence of concatenating results from \cite{Gro78}, \cite{Ebe80}, and \cite{HI77}. But Ruh \cite{Ruh82} (extending results of Gromov \cite{Gro78a}) showed that every almost flat manifold is an infranilmanifold. }
\end{remark}
Lemma \ref{lem1.nonpinch} now follows from the above Remark \ref{lem1.rem}. 
\begin{definition}\rm{
We consider the class $\mathcal{C}$ of poly-(finite or cyclic) groups : $\Gamma \in \mathcal{C}$ if it has normal series 
$$ \Gamma=\Gamma_1 \supseteq \Gamma_2 \supseteq ..........\Gamma_n=1 $$
such that each factor group $\Gamma_i \setminus \Gamma_{i+1}$ is either a finite group or an infinite cyclic group.
If all the factor groups are infinite cyclic, then $\Gamma$ is a poly-$\mathbb{Z}$ group; a group is virtually poly-$\mathbb{Z}$ (poly-$\mathbb{Z}$ by finite) if it contains a poly-$\mathbb{Z}$ subgroup of finite index.}
\end{definition}
\begin{remark}\rm{
It is well known (\cite{Wal71}) that $\mathcal{C}$ is the same as the class of virtually poly-$\mathbb{Z}$ group. F. T. Farrell and W. C. Hsiang \cite{FH81} proved the following result :}
\end{remark}
\begin{theorem}\label{whgro}
 Let $\Gamma$ be a torsion-free, poly-(finite or cyclic) group; then $Wh(\Gamma) = 0$ and $K_0(\mathbb{Z}\Gamma)=0$.
\end{theorem}
\begin{remark}\rm{
Theorem \ref{whgro} extends F. T. Farrell and W. C. Hsiang earlier result \cite[Theorem 3.1]{FH78}, where they showed that
$Wh (\Gamma)=0$ if $\Gamma$ is a Bieberbach group, that is, if $\Gamma$ is torsion-free and contains a finitely generated abelian subgroup of finite index.}
\end{remark}
\paragraph{Proof of Theorem \ref{corr.nonpinch}:}
Let $\overline{M}^n$ be the compactification of $M^n$ posited in Lemma \ref{lem1.nonpinch}. Note that $\overline{M}^n$ is a $\pi$-manifold since $M^n$ is parallelizable. Let $N^{n-1}$ denote the boundary component of $\overline{M}^m$ corresponding to that cusp of $M^n$ which contains the sets $f(t\times \mathbb{D}^{n-1})$ for all sufficiently large real numbers $t$.
Hence, there clearly exists a compact smooth manifold $W^n$ with the following properties:
\begin{itemize}
\item[(i)] The interior of $W^n$ is diffeomorphic to $M^{n}_{f,\phi}$.
\item[(ii)] One of the components of $\partial W^n$ is diffeomorphic to $N^{n-1}\#\Sigma^{n-1}$, where $\Sigma^{n-1}$ is an exotic homotopy sphere representing the element of $\Theta_{n-1}$ determined by $\phi:\mathbb{S}^{n-1}\to \mathbb{S}^{n-1}$.\\
We assume that $M^{n}_{f,\phi}$, supports a complete pinched negatively curved Riemannian metric with finite volume. Then, Lemma \ref{lem1.nonpinch} yields a second compactification of $M^{n}_{f,\phi}$, i.e., a compact smooth manifold $\overline{M}^{n}_{f,\phi}$ satisfying properties (i) and (ii) of the conclusion of Lemma \ref{lem1.nonpinch} with $M^n$ and $\overline{M}^n$ replaced respectively by $M^{n}_{f,\phi}$ and $\overline{M}^{n}_{f,\phi}$. But one easily sees that the boundaries of two smooth compactifications of the same manifold are smoothly $h$-cobordant. In particular, $N^{n-1}\#\Sigma^{n-1}$ is smoothly $h$-cobordant to an infranilmanifold. Hence $N^{n-1}\#\Sigma^{n-1}$ is diffeomorphic to an infranilmanifold because of Theorem \ref{whgro}, where it is
shown that $Wh (\pi_{1}N^{n-1})=0$. But this contradicts Lemma \ref{lem.nonpinch} since $N^{n-1}$ is a $\pi$-manifold. This completes the proof of Theorem \ref{corr.nonpinch}.
\end{itemize}
\begin{remark}\rm{
\indent
\begin{itemize}
\item[\bf{1.}] Note that this process can not be used for closed hyperbolic (or negatively curved) manifolds because they do not have cusps. There are a couple of canonical constructions of closed non-positively
curved manifolds from non-compact finite volume hyperbolic manifolds : (1) the double of a hyperbolic manifold and (2) the ones obtained by cusp closing
due to Schroeder \cite{Sch89}.
\item[\bf{2.}] C.S. Aravinda and F.T. Farrell \cite{AF94} constructed examples of compact topological manifolds $M$ with two different smooth
structures $\mathcal{M}_1$ and $\mathcal{M}_2$ such that $\mathcal{M}_1$ carries a Riemannian metric with non-positive sectional
curvature and geometric rank one while $\mathcal{M}_2$ cannot support a nonpositively curved metric. The examples $(M,\mathcal{M}_1)$ are of two types. The first class is obtained from Heintze’s examples of compact
non-positively curved manifolds, and the second class is obtained using a cusp closing construction. 
\item[\bf{3.}] A non-compact finite volume real hyperbolic $n$-manifold is homeomorphic to the interior of a compact manifold with boundary. Each boundary component is a flat manifold of dimension $n\text{-}1$. The compact manifold with boundary can be doubled along its boundary to form $DM$, the double of $M$. It is well known that the hyperbolic metric can be modified on each cusp to produce a metric of non-positive curvature. In \cite{Ont00}, P. Ontaneda constructed examples of non-compact finite volume real hyperbolic manifolds of dimension greater than five, such that their doubles admit at least three non-equivalent smoothable PL structures, two of which admit a Riemannian metric of non-positive curvature while the other does not. Here is the result:
\end{itemize}}
\end{remark}
\begin{theorem}\label{doub1}
 There are examples of non-compact finite volume real hyperbolic $n$-manifolds $M$, $n > 5$, such that
 \begin{itemize}
\item[(i)] $DM$ has, at least, three non-equivalent (smoothable) PL structures $\Sigma_{1}$, $\Sigma_{2}$, $\Sigma_{3}$.
\item[(ii)] $\Sigma_{1}$, $\Sigma_{2}$ admit a Riemannian metric with non-positive curvature, and negative curvature outside a hypersurface.
\item[(iii)] $\Sigma_{3}$ does not admit a Riemannian metric with non-positive curvature. In fact, $\Sigma_{3}$ does not even admit a piecewise flat metric of non-positive curvature.
\end{itemize}
\end{theorem}
\begin{remark}\rm{
\indent
\begin{itemize}
\item[\bf{1.}] In \cite{AF94}, examples are given of doubles of non-compact finite volume hyperbolic manifolds with exotic differentiable structures not admitting a non-positively curved metric. In these examples, the PL structure does not change.
\item[\bf{2.}] Now, the negatively curved manifolds mentioned up to this point were homeomorphic (hence homotopy equivalent) to hyperbolic manifolds. We call these manifolds of hyperbolic homotopy type. In \cite{Ard00}, Ardanza gave a version of Theorem \ref{hype.smooth} for manifolds that are not homotopy equivalent to a closed locally symmetric space; in particular, they do not have a hyperbolic homotopy type. We call these manifolds of non-hyperbolic homotopy type. His constructions use branched covers of hyperbolic manifolds. Recall that Gromov and Thurston \cite{GT87} proved that large branched covers of hyperbolic manifolds do not have the homotopy type of a closed locally symmetric space. Here is the statement of Ardanza's result :
\end{itemize}}
\end{remark}
\begin{theorem}\label{ardanza}
For all $n= 4k-1$, $k\geq 2$, there exist closed Riemannian $n$-dimensional manifolds $M$ and $N$ with negative sectional curvature such
that they do not have the homotopy type of a locally symmetric space and
\begin{itemize}
\item[(i)] $M$ is homeomorphic to $N$.
\item[(ii)] $M$ is not diffeomorphic to $N$.
\end{itemize}
\end{theorem}
\begin{remark}\rm{
Up to now the hyperbolic manifolds considered were real hyperbolic manifolds. We now consider Rigidity Questions for complex, quaternionic and Cayley hyperbolic manifolds. Recall that these are Riemannian $m$-manifolds whose universal covers, with the pulled back metric, are isometric to complex hyperbolic space $\mathbb{C}\textbf{H}^m$, quaternionic hyperbolic space $\mathbb{H}\textbf{H}^m$ $(m=4n)$, or Cayley hyperbolic plane $\mathbb{O}\textbf{H}^2$ $(m = 16)$, respectively. These manifolds have sectional curvatures in the interval $[-4, -1]$ and they also are isometric rigid due to Mostow Rigidity Theorem \ref{mostow}. In fact they satisfy the following superrigidity property in the quaternionic and Cayley cases.}
\end{remark}
\begin{theorem}\label{hern.yau}(Hern$\acute{a}$́ndez \cite{Her91} and Yau and Zheng \cite{YZ91})
Assume that $M^m$ and $N^m$ are homeomorphic closed Riemannian manifolds. If $M^m$ is complex, quaternionic $(m=4n, n\geq2)$ or Cayley hyperbolic $(m = 16)$ and $N^m$ has sectional curvatures in $[-4, -1]$, then $M^m$ and $N^m$  are isometric. 
\end{theorem}
\begin{theorem}\label{Corlette}(Corlette \cite{Cor92})
Assume that $M^m$ and $N^m$ are homeomorphic closed Riemannian manifolds. If $M^m$ is quaternionic $(m = 4n, n\geq 2)$ or Cayley hyperbolic $(m=16)$
and $N^m$ has non-positive curvature operator, then $M^m$ and $N^m$  are isometric (up to scaling). 
\end{theorem}
\begin{theorem}\label{MSY}(Mok, Siu and Yeung \cite{MSY93})
Assume that $M^m$ and $N^m$ are homeomorphic closed Riemannian manifolds. If $M^m$ is quaternionic $(m = 4n, n\geq 2)$ or Cayley hyperbolic $(m=16)$
and the complexified sectional curvatures of $N^m$ are nonpositive, then $M^m$ and $N^m$ are isometric (up to scaling).
\end{theorem}
\begin{remark}\rm{
\indent
\begin{itemize}
\item[\bf{1.}] The conditions in Theorem \ref{hern.yau} or Theorem \ref{Corlette} for $N^m$ imply the condition in Theorem \ref{MSY} for $N^m$.
\item[\bf{2.}] For the complex case in Theorem \ref{hern.yau}, Farrell and Jones \cite{FJ94} proved that this Rigidity can not be strengthened to requiring that the curvatures lie in the interval $[-4-\epsilon, -1 +\epsilon]$, for some $\epsilon>0$ :
\end{itemize}}
\end{remark}
\begin{theorem}\label{complex.hyp}
Let $m$ be either $4$ or any integer of the form $4n + 1$ where $n\geq 1$ and
$n$ is an integer. Given a positive real number $\displaystyle{\epsilon}$, there exists a closed smooth manifold $N^{2m}$ and a homotopy sphere $\Sigma^{2m}\in \Theta_{2m}$ such that the following is true:
\begin{itemize}
\item[(i)] $N^{2m}$ is  a complex hyperbolic manifold of complex dimension $m$.
\item[(ii)] The smooth manifolds $N^{2m}$ and $N^{2m}\#\Sigma^{2m}$ are homeomorphic but not diffeomorphic.
\item[(iii)]The connected sum  $N^{2m}\#\Sigma^{2m}$ supports a negatively curved Riemannian
metric whose sectional curvatures all lie in the closed interval $[-4 -\displaystyle{\epsilon}, -1 +\displaystyle{\epsilon}]$.
\end{itemize}
\end{theorem}
\begin{remark}\rm{
\indent
\begin{itemize}
\item[\bf{1.}] Note that $\epsilon$ cannot be 0 in the Theorem \ref{complex.hyp} due to  Theorem \ref{hern.yau}. 
\item[\bf{2.}] Theorem \ref{complex.hyp} showed that the answer to Problem \ref{prorig} for closed complex hyperbolic manifolds is negative.
\item[\bf{2.}] The idea of the proof of Theorem \ref{complex.hyp} can be paraphrased as follows: \\
Recall that $N$ and $N \# \Sigma$ are always homeomorphic since $\dim \Sigma >4$. Hence we need only choose $N$ and $\Sigma$ in Theorem \ref{complex.hyp} so that $N$ and $N \# \Sigma$ are not diffeomorphic in order to satisfy (ii). Letting $[M]$ denote the concordance class of $M$, Theorem \ref{diff-con} shows that it is sufficient to choose $N$ and $\Sigma$ so that $[N \# \Sigma]\neq [N]$  in $\mathcal{C}(N)$; i.e., so that $f^{*}([\Sigma])\neq 0$. It would be convenient at this point to be able to use Theorem \ref{brum}; but unfortunately this can't be done since a closed complex $n$-dimensional hyperbolic manifold $N$ is never stably parallelizable when $n > 1$; in fact, its first Pontryagin class is never zero. This last fact is a result of the close relationship between the tangent bundle $TM$ of $M$ and that of its positively curved dual symmetric space $\mathbb{C}\textbf{P}^n$. In fact, the following result was proven in \cite{FJ94}.
\end{itemize}}
\end{remark}
\begin{theorem}\label{staequ}\cite{FJ94}
Let $M$ be any closed complex $n$-dimensional hyperbolic manifold. Then there exists a finite sheeted cover $\mathcal{M}$ of $M$ and a map $f:\mathcal{M}\to \mathbb{C}\textbf{P}^n$ such that the pullback bundle $f^*(T\mathbb{C}\textbf{P}^n)$ and $T\mathcal{M}$ are stably equivalent complex vector bundles.
\end{theorem}
Theorem \ref{staequ} and the Whitney Embedding Theorem yield the following useful result:
\begin{theorem}\label{emb1}\cite{FJ94}
Let $M^{2n}$ be a closed complex $n$-dimensional hyperbolic manifold. Then there exists a finite sheeted cover $\mathcal{M}^{2n} of M^{2n}$ such that the following is true for any finite sheeted cover $N^{2n}$ of $\mathcal{M}^{2n}$. The manifold $N^{2n}\times \mathbb{D}^{2n+1}$ is orientation
preserving diffeomorphic to a codimension 0-submanifold contained in the interior of $\mathbb{C}\textbf{P}^n\times \mathbb{D}^{2n+1}$.
\end{theorem}
Let $f_N: N^{2n}\to \mathbb{S}^{2n}$ and $f_{\mathbb{C}\textbf{P}^n}: \mathbb{C}\textbf{P}^n\to \mathbb{S}^{2n}$ denote degree 1 maps. By using Theorem \ref{emb1} and Theorem \ref{kirby}, we have the following useful analogue of Theorem \ref{brum} :
\begin{corollary}\label{corr.emb1}\cite{FJ94}
Let $M^{2n}$ be a closed complex $n$-dimensional hyperbolic manifold and $\mathcal{M}^{2n}$ be the finite sheeted cover of $M^{2n}$ posited in Theorem \ref{emb1}. Then the following is true for every finite sheeted cover $N^{2n}$ of $\mathcal{M}$. The group homomorphism $f^{*}_{\mathbb{C}\textbf{P}^n}:\Theta_{2n}\to \mathcal{C}(\mathbb{C}\textbf{P}^n)$ factors through $f^{*}_{N}:\Theta_{2n}\to \mathcal{C}(N)$; i.e., there exists a homomorphism $\eta:\mathcal{C}(N)\to \mathcal{C}(\mathbb{C}\textbf{P}^n)$  such that $$\eta\circ f^{*}_{N}=f^{*}_{\mathbb{C}\textbf{P}^n}.$$
\end{corollary}
\begin{theorem}\label{com.con1}\cite{FJ94}
Let $M^{2n}$ be a closed complex $n$-dimensional hyperbolic manifold with $n > 2$ and $\mathcal{M}^{2n}$ be the finite sheeted cover of $M^{2n}$ posited in Theorem \ref{emb1}. Let $N^{2n}$ be a finite sheeted cover of $\mathcal{M}^{2n}$ and $\Sigma_1, \Sigma_2\in \Theta_{2n}$; i.e., $\Sigma_1$ and $\Sigma_2$ are a pair of homotopy $2n$-spheres. If the connected sums $N^{2n}\#\Sigma_1$ and $N^{2n}\#\Sigma_2$ are diffeomorphic, then $\mathbb{C}\textbf{P}^n\#\Sigma_1$ is concordant to either $\mathbb{C}\textbf{P}^n\#\Sigma_2$ or $\mathbb{C}\textbf{P}^n\#(-\Sigma_2)$.
\end{theorem}
\begin{proof}
The argument given to prove Addendum \ref{Addendum 2.3} shows that $N^{2n}\#\Sigma_1$ is concordant to either $N^{2n}\#\Sigma_2$ or $N^{2n}\#(-\Sigma_2)$. But in this argument one must note that both Mostow's Rigidity Theorem \ref{mostow} and Topological Rigidity Theorem \ref{gentoprigd} remain valid for compact complex hyperbolic manifolds of (real) dimension greater than 4. Next by Theorem \ref{kirby} that concordance classes of smooth structures on a smooth manifold $X$ are in bijective correspondence with homotopy classes of maps from $X$ to $Top/O$ provided $\dim X > 4$. In particular, $\Theta_{2n}=\pi_{2m}(Top/O)$ and $N^{2n}\#\Sigma_1$ is concordant to either $N^{2n}\#\Sigma_2$ or $N^{2n}\#(-\Sigma_2)$ can be interpreted as showing that
\begin{equation}\label{eq.con}
f^*_N(\Sigma_1)=f^*_N(\pm \Sigma_1).
\end{equation}
By using Equation (\ref{eq.con}) together with Corollary \ref{corr.emb1} and arguing analogous to the proof of Theorem \ref{brum}, we see that 
\begin{equation}\label{com.con}
f^{*}_{\mathbb{C}\textbf{P}^n}(\Sigma_1)=f^{*}_{\mathbb{C}\textbf{P}^n}(\pm \Sigma_1);
\end{equation}
i.e., $\mathbb{C}\textbf{P}^n\#\Sigma_1$ is concordant to either $\mathbb{C}\textbf{P}^n\#\Sigma_2$ or $\mathbb{C}\textbf{P}^n\#(-\Sigma_2)$.
\end{proof}
\begin{theorem}\label{adam}\cite{Ada66}
Suppose that $r \equiv$ $1$ or $2$ $\rm{mod}~8$ and $r > 0$. Then $\pi_{r}^{s}$ contains an element $\mu_r$, of order $2$, such that any map $h: \mathbb{S}^{q+r}\to \mathbb{S}^q$ representing $\mu_r$, induces a non-zero homomorphism of $\widetilde{KO}^{q}$.
\end{theorem}
\begin{lemma}\label{lem.con}\cite{FJ94}
For each positive integer $n$, there exists a homotopy sphere $\Sigma^{8n+2}\in \Theta_{8n+2}$ such that $\mathbb{C}\textbf{P}^{4n+1}\#\Sigma^{8n+2}$ is not concordant to $\mathbb{C}\textbf{P}^{4n+1}$. Furthermore there exists a homotopy sphere $\Sigma^{8}\in \Theta_{8}$  such that $\mathbb{C}\textbf{P}^{4}\#\Sigma^{8}$ is not concordant to $\mathbb{C}\textbf{P}^{4}$. Also for any two elements $\Sigma_{1}$, $\Sigma_{2} \in \Theta_{10}$, $\mathbb{C}\textbf{P}^{5}\#\Sigma_1$ is concordant to $\mathbb{C}\textbf{P}^{5}\#\Sigma_2$ if and only if $\Sigma_{1}=\Sigma_{2}$. (Recall that $\Theta_{10}$ is a cyclic group of order 6.)
\end{lemma}
\begin{proof}
We start by recollecting some facts from smoothing theory \cite{Bru71}. There are $H$-spaces $SF$, $F/O$ and $Top/O$ and $H$-space maps $\phi:SF\to F/O$,  $\psi: Top/O\to F/O$ such that
\begin{equation}\label{eq1.bru}
\psi_{*}:\Theta_{m}=\pi_{m}(Top/O)\to \pi_{m}(F/O)
\end{equation}
is an isomorphism if $m=8n+2$ where $n\geq 1$, and is a monomorphism when $m = 8$. The homotopy groups of $SF$ are the stable homotopy groups of spheres $\pi^{s}_m$ ; i.e., $\pi_{m}(SF)=\pi^{s}_m$ for $m\geq 1$. And
\begin{equation}\label{eq2.bru}
 \phi_{*}:\pi^{s}_{8n+2}\to \pi_{8n+2}(F/O)
\end{equation}
is an isomorphism for $n\geq 1$. Also the image of the homeomorphism 
\begin{equation}\label{eq3.bru}
 \phi_{*}:\pi^{s}_{8}\to \pi_{8}(F/O)
\end{equation}
is $\psi_{*}(\Theta_{8})$ and its kernel is a cyclic group of order $2$. (Recall that $\pi^{s}_{8}$ is the Klein $4$-group.) Consider the following commutative of diagram :
\begin{equation}\label{digram1}
\begin{CD}
[\mathbb{S}^{2m},Top/O]=\Theta_{2m}@>f^*_{\mathbb{C}\mathbb{P}^{m}}>> [\mathbb{C}\mathbb{P}^{m},Top/O]=\mathcal{C}(\mathbb{C}\mathbb{P}^{m})\\
@VV\psi_{*}V                      @VV\psi_{*}V\\
[\mathbb{S}^{2m},F/O]  @>f^*_{\mathbb{C}\mathbb{P}^{m}}>> [\mathbb{C}\mathbb{P}^{m},F/O]\\
@AA\phi_{*}A                      @AA\phi_{*}A\\
[\mathbb{S}^{2m},SF] @>f^*_{\mathbb{C}\mathbb{P}^{m}}>>  [\mathbb{C}\mathbb{P}^{m},SF]
\end{CD}
\end{equation}
In this diagram, the homomorphism $\phi_{*}:[\mathbb{C}\mathbb{P}^m,SF]\to [\mathbb{C}\mathbb{P}^m,F/O]$ is monic for all $m\geq 1$ by a result of Brumfiel \cite[p.77]{Bru71a}. Recall that the concordance class $[\mathbb{C}\mathbb{P}^m\#\Sigma]\in [\mathbb{C}\mathbb{P}^m,Top/O]$ of $\mathbb{C}\mathbb{P}^m\#\Sigma$ is $f^{*}_{\mathbb{C}\mathbb{P}^m}([\Sigma])$ when $m > 2$, and that $[\mathbb{C}\mathbb{P}^m]=[\mathbb{C}\mathbb{P}^m\#\mathbb{S}^{2m}]$ is the zero element of this group.\\
Let $\mu_{8n+2}$ be the element of order 2 in $\pi^{s}_{8n+2}$~$(n\geq1)$ given by Theorem \ref{adam}. Let $\Sigma^{8n+2}\in \Theta_{8n+2}$ such that
\begin{equation}\label{eq4.bru}
 \Sigma^{8n+2}=\psi_{*}^{-1}(\phi_{*}(\mu_{8n+2})).
\end{equation}
We also set $\Sigma^{8}= \psi_{*}^{-1}(\phi_{*}(x))$ where $x$ is any element which is not in the kernel of the homomorphism (\ref{eq3.bru}). Recall that $[X,~SF]$ can be identified with the $0^{th}$ stable cohomotopy group $\pi^0(X)$. Let $h:\mathbb{S}^{q+8n+2}\to \mathbb{S}^{q}$ represent $\mu_{8n+2}\in \pi^{s}_{8n+2}$. By Theorem \ref{adam}, $h$ induces a non-zero homomorphism on $\widetilde{KO}^q()$. Adams and Walker \cite{AW65} showed that $\Sigma^{q}f_{\mathbb{C}\mathbb{P}^{4n+1}}:\Sigma^{q}\mathbb{C}\mathbb{P}^{4n+1}\to \mathbb{S}^{q+8n+2}$ induces a monomorphism on $\widetilde{KO}^q()$. Then the composite map
\begin{equation}\label{eq8.bru}
h\circ \Sigma^{q}f_{\mathbb{C}\mathbb{P}^{4n+1}}: \Sigma^{q}\mathbb{C}\mathbb{P}^{4n+1}\to \mathbb{S}^{q}
\end{equation}
induces a non-zero homomorphism on $\widetilde{KO}^q()$. This shows that $$f^{*}_{\mathbb{C}\mathbb{P}^{4n+1}}(\mu_{8n+2})=[h\circ \Sigma^{q}f_{\mathbb{C}\mathbb{P}^{4n+1}}]\neq 0.$$ Since the homomorphism $\phi_{*}:[\mathbb{C}\mathbb{P}^m,SF]\to [\mathbb{C}\mathbb{P}^m,F/O]$ is monic, by the commutative diagram (\ref{digram1}) where $m=4n+1$, we have $$\psi_{*}( f^*_{\mathbb{C}\mathbb{P}^{4n+1}}(\Sigma^{8n+2}))=\phi_{*}(f^{*}_{\mathbb{C}\mathbb{P}^{4n+1}}(\mu_{8n+2}))\neq 0.$$ This implies that $f^*_{\mathbb{C}\mathbb{P}^{4n+1}}(\Sigma^{8n+2})\neq0$ and hence $\mathbb{C}\mathbb{P}^{4n+1}\#\Sigma^{8n+2}$ is not concordant to $\mathbb{C}\mathbb{P}^{4n+1}$. This completes the proof of the first sentence of Lemma \ref{lem.con}.\\
A similar argument but using \cite[Lemma I.9]{Bru71} in place of Theorem \ref{adam} and \cite{AW65} shows that the second sentence of Lemma \ref{lem.con} is also true. Consider the homomorphism induced by $f^{*}_{\mathbb{C}\textbf{P}^{5}}$
\begin{equation}\label{eq9.bru}
 \Theta_{10}=[\mathbb{S}^{10},Top/O]\to [\mathbb{C}\textbf{P}^5,Top/O]
\end{equation}
It becomes a monomorphism when localized at the prime 3 since $\pi_{i}(Top/O)$ localized at 3 is the zero group for all $i < 10$. This monomorphism together with the first sentence of Lemma \ref{lem.con} and the fact that $\Theta_{10}$ is cyclic of order 6 imply the truth of the last sentence of Lemma \ref{lem.con}.
\end{proof}
\begin{lemma}\label{constru.met}\cite{FJ94}
Given a positive integer $n$, there exists a family $b_{r}( , )$ of complete Riemannian metrics on $\mathbb{R}^{2n}$ which is parameterized by the real number $r\geq e$ and has the following three properties:
\begin{itemize}
\item[(i)] The sectional curvatures of $b_{r}( , )$ are all contained in $[-4-\epsilon(r), -4+\epsilon(r) ]$ where $\epsilon(r)$ is a $\mathbb{R}^+$ valued function such that \[\lim_{r \to\infty} \epsilon(r)=0.\]
\item[(ii)] The ball of radius $r$ about $0$ in $(\mathbb{R}^{2n},b_{r})$ is isometric to a ball of radius $r$ in $\mathbb{H}^{2n}$.
\item[(iii)] There is a diffeomorphism $f$ from $(\mathbb{R}^{2n},b_{r})$ to $\mathbb{C}\mathbb{H}^{n}$ which maps the complement of the ball of radius $r^2$ centered at $0$ isometrically to the complement of the ball of radius $r^2$ centered at $f(0)$. 
\end{itemize} 
\end{lemma}
Since the proof of Theorem \ref{hype.metric} was local in nature, one sees by examining it that the following stronger result was actually proven:
\begin{lemma}\label{pinmetric}\cite{FJ94}
Given a positive integer $n > 4$ and a positive real number $\epsilon$, there exists a positive real number $\alpha$ such that the following is true. Let $M^n$  be any $n$-dimensional Riemannian manifold whose sectional curvatures are contained in the interval $[a, b]$.
Suppose that $M^n$  contains a codimension 0-submanifold which is isometric to an open ball of radius $\alpha$ in $\mathbb{H}^n$. Then given any homotopy sphere $\Sigma \in \Theta_{n}$, there exists a Riemannian metric on $M^n\#\Sigma$ whose sectional curvatures are all contained in the interval $[a', b']$ where $a' = min\{ a, -1-\epsilon \}$ and $b' = max\{b, -1+\epsilon\}$.
\end{lemma}
Lemma \ref{pinmetric} and Theorem \ref{constru.met} have the following important consequence:
\begin{theorem}\label{compinch}
Let $M^{2n}$  be a closed complex hyperbolic manifold of complex dimension $n$ and let $\mathcal{M}^{2n} $  be the finite sheeted cover of $M^{2n}$  posited in Theorem \ref{emb1}. Given $\epsilon>0$, there exists a finite sheeted cover $\mathcal{N}^{2n}$ of $\mathcal{M}^{2n}$ such that, for any finite
sheeted cover $N^{2n}$ of $\mathcal{N}^{2n}$ and any homotopy sphere $\Sigma \in \Theta_{2n}$, the connected sum $N^{2n}\#\Sigma$ supports a Riemannian metric all of whose sectional curvatures lie in the interval $[-4-\epsilon, -4+\epsilon ]$.
\end{theorem}
\begin{proof}
Let $\alpha$ be the number posited in Lemma \ref{pinmetric} relative to $2n$ and $\epsilon$. Fix a number $r\geq \alpha$ whose magnitude will be presently determined. By Theorem \ref{malcev}, $\pi_1\mathcal{M}^{2n}$ is residually finite. Arguing as in Proof of Theorem \ref{hype.smooth}, we construct a finite sheeted cover $\mathcal{N}^{2n}$ of $\mathcal{M}^{2n} $ such that $\mathcal{N}^{2n}$ contains a codimension-0 submanifold which is
isometric to an open ball of radius $2r^2$ in $\mathbb{C}\textbf{H}^n$. Let $N^{2n}$ be any finite sheeted cover of $\mathcal{N}^{2n}$. Then $N^{2n}$ contains a codimension-0 submanifold $U^{2n}$ which is also isometric to an open ball of radius $2r^2$ in $\mathbb{C}\textbf{H}^n$. Using Lemma \ref{constru.met}, we can put a new Riemannian metric $b(,)$ on $N^{2n}$ (changing it only on  $U^{2n}$) such that all the sectional curvatures of $(N^{2n}, b)$ lie in the interval $[-4-\epsilon(r), -4+\epsilon(r) ]$ and  $(N^{2n}, b)$ contains a codimension-0 submanifold isometric to an open ball of radius $r \in \mathbb{H}^{2n}$. Here $\epsilon(r)$ is a positive real number which depends only on $r$ (and n) and
$\epsilon(r)$ as $r\to \infty$. Now pick $r$ large enough so that $\epsilon(r)\leq \epsilon$. Then Lemma \ref{pinmetric} is applicable to $(N^{2n}, b)$ completing the proof of Theorem \ref{compinch}.
\end{proof}

The following two results are immediate consequences of stringing together Theorem \ref{com.con1}, Lemma \ref{lem.con} and Theorem \ref{compinch}.
\begin{theorem}\label{com.smooth}
Let $M^{2m}$  be any closed complex hyperbolic manifold of complex dimension $m$ where  $m$ be either $4$ or any integer of the form $4n + 1$ where $n\geq 1$ and $n$ is an integer. Let $\Sigma^{2m} \in \Theta_{2m}$ denote the specific homotopy sphere posited in Lemma \ref{lem.con}. Given a positive real number $\displaystyle{\epsilon}$, there exists a finite sheeted cover $\mathcal{N}^{2m}$ of $M^{2m}$ such that the following is true for any finite
sheeted cover $N^{2m}$ of $\mathcal{N}^{2m}$:
\begin{itemize}
\item[(i)] The connected sum  $N^{2m}\#\Sigma^{2m}$ supports a negatively curved Riemannian
metric whose sectional curvatures all lie in the closed interval $[-4 -\displaystyle{\epsilon}, -1 +\displaystyle{\epsilon}]$.
\item[(ii)] The smooth manifolds $N^{2m}$ and $N^{2m}\#\Sigma^{2m}$ are not diffeomorphic.
\end{itemize}
\end{theorem}
\begin{addendum}\label{Addendum.com}
Let $M^{10}$  be any closed complex hyperbolic manifold of complex dimension $5$. Let $\Sigma^{10} \in \Theta_{10}$ denote the specific homotopy sphere posited in Lemma \ref{lem.con}. Given a positive real number $\displaystyle{\epsilon}$, there exists a finite sheeted cover $\mathcal{N}^{10}$ of $M^{10}$ such that the following is true for any finite sheeted cover $N^{10}$ of $\mathcal{N}^{10}$ :\\
There exist two other homotopy spheres $\Sigma^{10}_{1}$ and $\Sigma^{10}_{2}$ (besides $\Sigma^{10}$ such that the following is true.
\begin{itemize}
\item[(i)] The manifolds $N^{10}$, $N^{10}\#\Sigma^{10}$, $N^{10}\#\Sigma^{10}_1$ and $N^{10}\#\Sigma^{10}_2$ are pairwise non diffeomorphic.
\item[(ii)] Each of the manifolds $N^{10}\#\Sigma^{10}$, $N^{10}\#\Sigma^{10}_1$ and $N^{10}\#\Sigma^{10}_2$ supports a negatively curved Riemannian metric whose sectional curvatures lie in the interval\\
 $[-4 -\displaystyle{\epsilon}, -1 +\displaystyle{\epsilon}]$.
\end{itemize}
\end{addendum}
By using Theorem \ref{com.con1} and Theorem \ref{compinch}, the author proved the following result, which gives counterexamples to smooth rigidity Problem \ref{prorig} for negatively curved manifolds \cite{Ram14}: 
\begin{theorem}\label{ramesh}
Let $n$ be either 7 or 8. Given any positive number $\epsilon\in \mathbb{R}$, there exists a pair of closed negatively curved Riemannian manifolds $M$ and $N$ having the following properties:
\begin{itemize}
\item[\rm{(i)}] $M$ is a complex $n$-dimensional hyperbolic manifold.
 \item[\rm{(ii)}] The sectional curvatures of $N$ are all in the interval $[-4 -\epsilon, -1 +\epsilon]$.
 \item[\rm{(iii)}] The manifolds $M$ and $N$ are homeomorphic but not diffeomorphic.
 \end{itemize}
\end{theorem}
\begin{remark}\label{compopen}\rm{
\indent
\begin{itemize}
\item[\bf{1.}] Theorem \ref{complex.hyp} now  follows from Theorem \ref{com.smooth} and Theorem \ref{recomhy}.
\item[\bf{2.}] Since $M^{2m}\#\Sigma^{2m}$ $(m > 2)$ is always homeomorphic to $M^{2m}$, we are left with the problem of detecting when $M^{2m}\#\Sigma^{2m}$ and $M^{2m}$ are not diffeomorphic by Farrell and Jones \cite{FJ94}. Using Mostow's Rigidity Theorem \ref{mostow} and Topological Rigidity Theorem \ref{gentoprigd} together with the fundamental paper of Kervaire and Milnor \cite{KM63}, this question is essentially reduced to a (non-trivial) question about the stable homotopy group $\pi^0(M^{2m})$. To address this question F.T. Farrell and L.E. Jones showed in Theorem \ref{com.smooth}, using a result of Deligne and Sullivan \cite{DS75}, that all "sufficiently large" finite sheeted covers of $M^{2m}$ embed in $\mathbb{C}\textbf{P}^m\times \mathbb{R}^{2m+1}$ with trivial normal bundle. This allows us to look at the above question via Theorem \ref{com.con1} on the specific manifold $\mathbb{C}\textbf{P}^m$ instead of the arbitrary compact complex hyperbolic manifold $M^{2m}$ (see Corollary \ref{corr.emb1} and the proof of 
Lemma \ref{lem.con}). 
\end{itemize}}
\end{remark}
\section{\large Tangential Maps and Exotic Smoothings of Locally Symmetric Spaces}
In this section we discuss the existence of tangential map between dual symmetric spaces which was constructed in \cite{Oku01}. We also discuss that how tangential map can be used to obtain exotic smooth structures on a compact locally symmetric space of non-compact type.\\~\\
In \cite{Mat62}, Y. Matsushima constructed a map  $j^* : H^*(X_u,\mathbb{R})\to H^*(X,\mathbb{R})$, where $X=\Gamma\setminus G/K$ is a compact locally symmetric space of non-compact type and $X_u$ is its global dual twin of compact type. Moreover, Y. Matsushima showed that this map is monomorphic and, upto a certain dimension depending only on the Lie algebra of $G$, epimorphic. A refinement of Matsushimas argument, due to H. Garland \cite{Gar71} and Borel \cite{Bor74}, allowed the later to extend these results to the case where $X$ is non-compact but has a finite volume. However, since the construction of the map $j^*$ is purely algebraic (in terms of invariant exterior differential forms), the following natural question was asked by B. Okun \cite{Oku01} :\\
\begin{question}
Is there a topological map  $X\to X_u$ inducing $j^*$ in cohomology?.
\end{question}
\begin{remark}\rm{
For the case of $X$ being a complex hyperbolic manifold this kind of map was constructed by F.T. Farrell and L.E. Jones in Theorem \ref{staequ}, where it was used to produce non-trivial smooth structures on complex hyperbolic manifolds. In general, the answer to this question is negative, since Matsushima's map does not necessarily take rational cohomology classes into rational ones, so it cannot be induced by a topological map. However, Boris Okun \cite{Oku01} showed that there is a finite sheeted cover $\widehat{X}$ of $X=\Gamma\setminus G/K $ and a tangential map (i.e., a map covered by a map of tangent bundles) $\widehat{X} \to X_u$, and also showed that in the case where $G$ and $K$ are of equal rank, $k^*$ coincides with Matsushima's map $j^*$ and therefore has nonzero degree.}
\end{remark}
We shall recall the following definition and result to prove the existence of a tangential map between dual symmetric spaces \cite{Oku01}:
\begin{definition}\rm{
Let $M^n$ and $N^n$ be smooth $n$-dimensional manifolds.
A smooth map $k:M^n\to N^n$ is called tangential map if there is a smooth map $h: TM\to TN$ such that
\begin{itemize}
\item[(a)] $k\circ \pi_{1}=h\circ \pi_{2}$, where $\pi_{1}:TM\to M$ and $\pi_{2}:TN\to N$ be tangent bundle projections.
\item[(b)] For every $x\in M$, the map $\pi_{1}^{-1}({x})\to \pi_{2}^{-1}({k(x)})$ induced by $h$ is an isomorphism between vector spaces.
\end{itemize}}
\end{definition}
\begin{lemma}\label{tangential}\rm{
Any map $g:X\to X_u$ between the dual symmetric spaces which preserves canonical $K$-bundle structure is tangential.}
\end{lemma}
The following theorem provides the existence of a tangential map between dual symmetric space due to B. Okun \cite{Oku01}:
\begin{theorem}\label{finite sheeted tang}
Let $X=\Gamma\setminus G/K$ and $X_u= G_u/K$ be dual symmetric spaces. Then there exist a finite sheeted cover $\widehat{X}$ of $X$ (i.e., a subgroup $\widehat{\Gamma}$ of finite index in $\Gamma$, $\widehat{X}=\widehat{\Gamma}\setminus G/K)$ and a tangential map $k:\widehat{X}\to X_u$.
\end{theorem}
\begin{proof}
Consider the canonical principal fiber bundle with structure group $K$ over $X$:  $p:\Gamma\setminus G\to \Gamma\setminus G/K$. If we extend the structure group to the group $G$ we get a flat principal bundle: $\Gamma\setminus G\times_{K} G=G\setminus K\times_{\Gamma} G\to \Gamma\setminus G/K$ \cite{KT75}. Extend the structure group further to the group $G_c.$ The resulting bundle is a flat bundle with an algebraic linear complex Lie structure group, so by Theorem of Deligne and Sullivan \cite{DS75} there is a finite sheeted cover $\widehat{X}$ of $X$ such that the pullback of this bundle to $\widehat{X}$ is trivial. This means that for $\widehat{X}$ the bundle obtained by extending the structure group from $K$ to $G_u$ is trivial too, since $G_u$ is the maximal compact subgroup of $G_c$ . Consider now the following
diagram:
$$ \xymatrix{ & &  X_{u}\ar[d]^{C_{_{p_u}}}& \\
&\widehat{X}\ar[rd]_{\simeq 0}\ar[ur]^{k}\ar[r]_{C_{_{\widehat{p}}}}  & BK \ar[d]^{i}&\\
& &    BG_{u} &\\}$$

Here the map $c_{\widehat{p}}$ is a classifying map for the canonical bundle $\widehat{p}$ over $\widehat{X}$. The map $i$, induced by standard inclusion $K\subset G$ is a fibration with a fiber
$X_u=G_{u}/K$. Note that the inclusion of $X_u$ in $BK$ as a fiber also classifies canonical principal bundle $p_u :G_u\to G_{u}/K$. By the argument above, the composition $i\circ c_{\widehat{p}} :\widehat{X}\to BG_u$ is homotopically trivial.
Choose a homotopy contracting this composition to a point. As the map $i$ is a fibration we can lift this homotopy to $BK$. The image of the end map of the lifted homotopy is contained in the fiber $X_u$, since its projection to $BG_u$ is a point. Thus, we obtain a map $k:\widehat{X}\to X_u$ which makes the upper triangle of the diagram homotopy commutative. It follows that the map $k$ preserves canonical bundles on the spaces $\widehat{X}$ and $X_u:k^*(p_u)=\widehat{p}$. By Lemma \ref{tangential}, the map $k$ is tangential.
\end{proof}
\begin{remark}\rm{
Both Theorem \ref{staequ} and Theorem \ref{finite sheeted tang} depend on a deep result about flat complex vector bundles due to Deligne and Sullivan \cite{DS75}. Their result was also used by Sullivan in \cite{Sul79} to prove Theorem \ref{sullivan}.}
\end{remark}
The following theorem showed that the above tangential map coincides with Matsushima's map and hence has a non zero degree \cite{Oku01}:
\begin{corollary}\label{equalrank}
Let $X=\Gamma\setminus G/K$ and $X_u= G_u/K$ be dual symmetric spaces. Let $\widehat{X}$ be the finite sheeted cover of $X$ and $k:\widehat{X}\to X_u$ be the tangential map, constructed in Theorem \ref{finite sheeted tang}. If the groups $G_u$ and $K$ are of equal rank and the group $\Gamma$ is cocompact in $G$ then the map induced by $k$ in cohomology coincides with Matsushima's map $j^*$ and the map $k$ has a non zero degree.
\end{corollary}
 We will now discuss how non triviality of certain types of smooth structures is preserved under tangential maps. First, we recall the following very useful theorem :
\begin{theorem}\label{trivial bund} \cite{KM63}
Let $\xi$ be a $k$-dimensional vector bundle over an $n$-dimensional space, $k > n$. If the Whitney sum of $\xi$ with a trivial bundle is trivial, then $\xi$ itself is trivial.
\end{theorem}
The following theorem is a generalization of Theorem \ref{emb1} and is due to Boris Okun \cite{Oku02}:
\begin{theorem}\label{tubular}
Let $k:M^n\to N^n$ be a tangential map between two closed smooth $n$-dimensional manifolds. Then $M^n\times \mathbb{D}^{n+1}$ is diffeomorphic to a co-dimension $0$-submanifold of the interior of $N^n\times \mathbb{D}^{n+1}$.
\end{theorem}
\begin{proof}
Let $i:N^n\to N^n\times \mathbb{D}^{n+1}$ be the standard inclusion $i(x) = (x, 0)$ and
$p:N^n\times \mathbb{D}^{n+1}\to N^n$  be the projection on the first factor $p(x, y) = x$.
Consider the composition $i\circ k:M^n\to N^n\times \mathbb{D}^{n+1}$. By Whitney Embedding Theorem, this composition can be approximated by an embedding $w:M^n\to N^n\times \mathbb{D}^{n+1}$, such that $w$ is homotopic to $i\circ k$. Let $\nu$ denote the normal bundle of the manifold $M^n$ considered as a sub manifold of $N^n\times \mathbb{D}^{n+1}$ via $w$. By definition of the normal bundle, we have:
$$w^{*}(T(N^n\times \mathbb{D}^{n+1})) = \nu \oplus T(M^n)$$
On the other hand, $T(N^n\times \mathbb{D}^{n+1}) = p^{*}(T(N^n))\oplus \varepsilon^{n+1}$, where $\varepsilon^{n+1}$ denotes the $(n + 1)$-dimensional trivial bundle. Combining these two equations together, we obtain:
\begin{center}
$\nu \oplus T(M^n) = w^{*}(p^{*}(T(N^n)))\oplus \varepsilon^{n+1}$
\end{center}
Since by construction $w\simeq i\circ k$, we have $w^{*}=(i\circ k)^* = k^*\circ i^*$.
Note that $i^*\circ p^* = (p\circ i)^* = id$ and $k^*T(N^n) =T(M^n)$ since the map $k$ is tangential. It follows that $\nu \oplus T(M^n) = \nu \oplus \varepsilon^{n+1}$ i.e., the bundle $\nu$ is stably trivial. By Theorem \ref{trivial bund}, $\nu$  is trivial itself; therefore, the tubular neighborhood of $M^n$  in $N^n\times \mathbb{D}^{n+1}$ is diffeomorphic to $M^n\times \mathbb{D}^{n+1}$.
\end{proof}
The following theorem is a generalization of Corollary \ref{corr.emb1} and is due to Boris Okun \cite{Oku02}:
\begin{theorem}\label{suspension.homot}
Let $k:M^n\to N^n$ be a tangential map between $n$-dimensional manifolds. Then there exist a map $g:\Sigma^{n+1}N^n\to \Sigma^{n+1}M^n$ such that the suspension $\Sigma^{n+1}f_N$ and the composite $\Sigma^{n+1}f_M\circ g$ are homotopic as maps from
$\Sigma^{n+1}N^n\to \Sigma^{n+1}\mathbb{S}^n=\mathbb{S}^{2n+1}$.
\end{theorem}
\begin{proof}
Fix base points $x_0\in M^n$ and $y_0\in N^n$. Let $B$ denote a small neighborhood of $x_0$ in $M^n$. By general position, we may assume that the embedding $F :M^n\times \mathbb{D}^{n+1}\to N^n\times \mathbb{D}^{n+1}$ of Theorem \ref{tubular} has the additional property that:
$$\rm{Im}(F)\cap(y_0\times \mathbb{D}^{n+1})\subseteq F(B\times \mathbb{D}^{n+1})$$
The suspensions $\Sigma^{n+1}M^n$ and $\Sigma^{n+1}N^n$ can be identified as the following quotient
spaces:\\~~\\ $\Sigma^{n+1}M^n= \frac{\displaystyle{M^n}\times \displaystyle{\mathbb{D}^{n+1}}}{ \displaystyle{M^n}\times \displaystyle{\partial\mathbb{D}^{n+1}} \cup \displaystyle{B}\times \displaystyle{\mathbb{D}^{n+1}}}$ ,   $\Sigma^{n+1}N^n= \frac{\displaystyle{N^n}\times \displaystyle{\mathbb{D}}^{n+1}}{\displaystyle{N^n}\times \displaystyle{\partial\mathbb{D}}^{n+1} \cup \displaystyle{y_{0}}\times \displaystyle{\mathbb{D}^{n+1}}}$.\\~\\
Let $*$ denote the point in $\Sigma^{n+1}M^n$ corresponding to the subset $M^n\times \partial \mathbb{D}^{n+1}\cup B\times \mathbb{D}^{n+1}$ in the first formula. Define $g: \Sigma^{n+1}N^n\to \Sigma^{n+1}M^n$ by
\begin{eqnarray*}
g(y) =  \left\{
\begin{array}{l}
 F^{-1}(y)~~ if\ \  y\in F((M^n-B)\times Int(\mathbb{D}^{n+1})). \\ \\

 * \ \ , ~~otherwise.
\end{array}
\right .
\end{eqnarray*}
It is easy to see that $\Sigma^{n+1}f_N$ and the composite $\Sigma^{n+1}f_M\circ g$ are homotopic.
\end{proof}
The following lemma shows that, atleast for this type of variation of smooth structures, tangential maps preserve non triviality of smooth structures:
\begin{theorem}\label{locally.concord}\cite{Oku02}
Let $k:M^n\to N^n$ be a tangential map between $n$-dimensional manifolds and assume $n\geq7$. Let $\Sigma_{1}$ and  $\Sigma_{2}$ be homotopy $n$-spheres. Suppose that $M^n\#\Sigma_{1}^n$ is concordant to $M^n\#\Sigma_{2}^n$, then $N^n\#\Sigma_{1}^n$ is concordant to $N^n\#\Sigma_{2}^n$.
\end{theorem}
\begin{proof}
By Theorem \ref{kirby}, we know that concordance classes of smooth structures on a smooth manifold $Z^n$ $(n > 4)$ are in one-to-one correspondence with homotopy classes of maps from $Z^{n}$ to $Top/O$ denoted by $[Z^n,Top/O]$. The space $Top/O$ is an infinite loop space. In particular, $Top/O=\Omega^{n+1}(Y)$ for some space $Y$. We have suspension isomorphism $[Z^n,Top/O]\cong [\Sigma^{n+1}Z^n, Y]$, so the smooth structures on $Z^n$ can be thought of as (homotopy classes of) maps from $\Sigma^{n+1}Z^n$ to $Y$. In this way a connected sum $Z^n\#\Sigma^n$ gives rise to a map from  $\Sigma^{n+1}Z^n$ to $Y$, which we will denote by $\displaystyle{z}_{\Sigma}$. Using the standard sphere $\mathbb{S}^n$ in place of $Z^n$ we see that an exotic sphere $\Sigma^n$ itself corresponds to a map $\displaystyle{s}_{\Sigma}:\mathbb{S}^{2n+1}\to Y$. The naturality of this construction and Theorem \ref{suspension.homot}  imply that the diagram
$$\xymatrix{ \Sigma^{n+1}\ar[dd]_{g} N^{n}\ar[dr]_{\displaystyle{n}_{\Sigma_{i}}}\ar[drrr]^{\Sigma^{n+1}f_{N}} & &   \\
&Y &  &\ar[ll]_{\displaystyle{s}_{\Sigma_{i}}}\mathbb{S}^{2n+1} \\
\Sigma^{n+1}M^{n}\ar[ur]^{\displaystyle{m}_{\Sigma_{i}}}\ar[urrr]_{\Sigma^{n+1}f_{M}}& &\\}$$
is homotopy commutative for $i = 1$, $2$. In particular, we see that $\displaystyle{n}_{\Sigma_{i}}\simeq \displaystyle{m}_{\Sigma_{i}}\circ g$. Since the connected sums $M^n\#\Sigma_{1}^n$ and $M^n\#\Sigma_{2}^n$ are concordant. It follows that the maps $m_{\Sigma_{1}}$ and $m_{\Sigma_{2}}$ are homotopic, and therefore the connected sums $N^n\#\Sigma_{1}^n$ and $N^n\#\Sigma_{2}^n$ are concordant.
\end{proof}
We will now discuss that how to produce nonstandard smooth structures on the non-positively curved symmetric spaces. For that we will use the following rigidity result.
\begin{theorem}\label{Eberlein.Gromov}(Eberlein and Gromov Strong Rigidity Theorem \cite{Ebe83, BGS85})
Let $X$ be a compact locally symmetric space of non-compact type such that all metric factors of $X$ have rank greater than $1$. Let $M$ be a closed connected non-positively curved Riemannian manifold. Then any isomorphism from  $\pi_{1}(X)$ to $\pi_{1}(Y)$ is induced by a unique isometry (after adjusting the normalizing constants for $X$).
\end{theorem}
The following lemma provides a connection between concordance and diffeomorphism classes of smooth structures for locally symmetric spaces by using the argument given in Theorem \ref{Lemma 2.1}:
\begin{theorem}\label{concord}\cite{Oku02}
Let $X =\Gamma \setminus G/K$ be a compact orientable symmetric space of non-compact type such that the universal cover $G/K$ of $X$ has no 2-dimensional metric factor projecting to a closed subset of $X$ and $\dim X\geq 7$. Let $\Sigma_{1}$ and $\Sigma_{2}$ be homotopy spheres of the same dimension as $X$. Suppose $X\#\Sigma_{1}$ is diffeomorphic to $X\#\Sigma_{2}$. Then $X\#\Sigma_{1}$ is concordant either to $X\#\Sigma_{2}$ or to $X\#(-\Sigma_{2})$.
\end{theorem}
\begin{proof}
We can assume that the connected sums of  $X$ with $\Sigma_{1}$ and $\Sigma_{2}$ are taking place
on the boundary of a small metric ball $B$ in $X$, so we think of $X\#\Sigma_{1}$ and $X\#\Sigma_{2}$ as being topologically identified with $X$, and the changes in the smooth structure happen inside $B$. Let $f:X \#\Sigma_{1}\rightarrow X \#\Sigma_{2}$ be a diffeomorphism.\\
First we consider a special case, where $f:X\rightarrow X$ is homotopic to the identity. Thus we have a homotopy $h:X\times [0,1]\rightarrow X$ with $h|X\times 1 = f$ and $h|X\times 0 = id$. Define $H:X\times [0,1]\rightarrow X\times [0,1]$ by $H(x, t)=(h(x, t), t)$. Note that $H$ is a homotopy equivalence which restricts to a homeomorphism on the boundary. Therefore, by Farrell and Jones Topological Rigidity Theorem \ref{gentoprigd}, $H$ is homotopic rel boundary to a homeomorphism $\widehat{H}:X \times [0,1]\rightarrow X\times [0,1]$. We put the smooth structure $X\#\Sigma_{2}\times [0,1]$ on the range of $\widehat{H}$ and, by pulling it back along $\widehat{H}$, obtain the smooth structure $N$ on the domain of $\widehat{H}$. Since, by construction, $\widehat{H}:N\rightarrow X\#\Sigma_{2}\times [0,1]$ is a diffeomorphism, $\widehat{H}|X\times 1 = f$ and $\widehat{H}|X\times 0 = id$, $N$ is a concordance between $X\#\Sigma_{1}$ and $X\#\Sigma_{2}$ .

The general case reduces to the above special case as follows: If $f:X\rightarrow X$ is an orientation preserving homeomorphism, then, by the Strong Mostow Rigidity Theorem \ref{mostow}, $f^{-1}$ is homotopic to an orientation preserving isometry $g$. Since one can move around small metric balls in  $X$ by smooth isotopies, $X$ is homotopic to a diffeomorphism $\widehat{g}:X\rightarrow X$ such that  $\widehat{g}|B = id$. Since $X\#\Sigma_{2}$ is obtained by taking the connected sum along the boundary of $B$ and $\widehat{g}|B = id$ and $\widehat{g}|X\setminus B$ is a diffeomorphism, it follows that $\widehat{g}:X\#\Sigma_{2}\rightarrow X\#\Sigma_{2}$ is also a diffeomorphism. Therefore the composition \\$\widehat{g}\circ f:X\#\Sigma_{1}\rightarrow X\#\Sigma_{2}$ is a diffeomorphism homotopic to the identity and it follows from the previous special case that $X\#\Sigma_{1}$ and $X\#\Sigma_{2}$ are concordant.
If $f$ is an orientation reversing homeomorphism, then similar argument produces a diffeomorphism $\widehat{g}:X\#\Sigma_{2}\rightarrow X\#\Sigma_{2}$ and it follows that $X\#\Sigma_{1}$ is concordant to $X\#\Sigma_{2}$.
\end{proof}
The following theorem provides examples of exotic smooth structures on locally symmetric space of noncompact type, which give counterexamples to smooth rigidity Problem \ref{prorig}:
\begin{theorem}\label{locally.exto}\cite{Oku02}
Let  $X =\Gamma\setminus G/K$ and $X_u = G_u/K$  be compact dual symmetric spaces such that the universal cover $G/K$ of $X$ has no $2$-dimensional metric factor projecting to a closed subset of $X$ and assume $\dim X\geq7$. Let $\widehat{X}$ be the oriented finite sheeted cover of $X$, the existence of which was established by Theorem \ref{finite sheeted tang}. Let $\Sigma_{1}$ and $\Sigma_{2}$ be homotopy spheres of the same dimension as $X$. If the connected sum $X_{u}\#\Sigma_{1}$ is not concordant to both $X_{u}\#\Sigma_{2}$ and $X_{u}\#-\Sigma_{2}$ then $\widehat{X}\#\Sigma_{1}$ and $\widehat{X}\#\Sigma_{2}$  are not diffeomorphic.
\end{theorem}
\begin{proof}
Suppose the connected sums $\widehat{X}\#\Sigma_{1}$ and $\widehat{X}\#\Sigma_{2}$ are diffeomorphic. Then, by Theorem \ref{concord}, $\widehat{X}\#\Sigma_{1}$ is concordant either to $\widehat{X}\#\Sigma_{2}$ or to $\widehat{X}\#\Sigma_{2}$. It follows from Theorem \ref{locally.concord} that either $X_{u}\#\Sigma_{1}$ is concordant to $X_{u}\#\Sigma_{2}$ or to $X_{u}\#(-\Sigma_{2})$, which contradicts the hypothesis. This contradiction proves the theorem.
\end{proof}
Applying Theorem \ref{locally.exto} to the Example \ref{exm.loc4}, we obtain the following corollary \cite{Oku02}:
\begin{corollary}\label{locally.complex}
Let $G =G_c$ be a complex semi simple Lie group, and let $X =\Gamma\setminus G_c/G_u$ and $X_u = G_u$ be compact dual symmetric spaces such that $\dim X\geq 7$. Let $\widehat{X}$ be the oriented finite sheeted cover of $X$, the existence of which was established by Theorem \ref{finite sheeted tang}. Let $\Sigma_{1}$ and $\Sigma_{2}$ be two non diffeomorphic homotopy spheres of the same dimension as $X$. Then the connected sums $\widehat{X}\#\Sigma_{1}$ and $\widehat{X}\#\Sigma_{2}$ are not diffeomorphic.
\end{corollary}
\begin{proof}
Since the group $G_c$ is a complex Lie group, by Theorem \ref{locally.exto}, it suffices to show that the connected sum $G_{u}\#\Sigma_{1}$ is not concordant to both $G_{u}\#\Sigma_{2}$ and $G_{u}\#(-\Sigma_{2})$. Since the tangent bundle of a Lie group is trivial. Therefore the constant map $G_{u}\to \mathbb{S}^{\dim G_u}$ is tangential, and the required statement follows from Theorem \ref{locally.concord}.
\end{proof}
\begin{remark}\rm{
\indent
\begin{itemize}
\item[\bf{1.}] The above Corollary \ref{locally.complex} provides examples of exotic smooth structures on locally symmetric space of higher rank. Therefore by Eberlein and Gromov Rigidity Theorem \ref{Eberlein.Gromov} these examples do not admit Riemannian metric of non-positive curvature. This answers in the negative the question, due to Eberlein, of whether a smooth closed manifold, homotopy equivalent to a non-positively curved one, admits a Riemannian metric of non-positive curvature. Rank 1 examples with the same property were independently constructed in \cite{AF94} by C.S. Aravinda and F.T. Farrell. 
\item[\bf{2.}] F.T. Farrell and L.E. Jones conjectured in \cite{FJ94} as follows :
\end{itemize}}
\end{remark}
\begin{conjecture}\label{gene.hyp}
Given a positive real number $\displaystyle{\epsilon}$, there exist pairs of closed smooth manifolds $M_{1}^{8}$, $M_{2}^{8}$ and $M_{1}^{16}$, $M_{2}^{16}$ such that the following is true:
\begin{itemize}
\item[(i)] $M_{2}^8$ is a quaternionic hyperbolic manifold of real dimension 8 and $M_{2}^{16}$ is a Cayley hyperbolic manifold (of real dimension 16).
\item[(ii)]The smooth manifolds $M_{1}^{2m}$ and $M_{2}^{2m}$ where $m=4$ and $8$ are homeomorphic but not diffeomorphic.
\item[(iii)] $M_{1}^{2m}$ ($m=4$ and $8$) supports a negatively curved Riemannian metric whose sectional curvatures all lie in the closed interval $[-4 -\displaystyle{\epsilon}, -1 +\displaystyle{\epsilon}]$.
\end{itemize}
\end{conjecture}
\begin{remark}\rm{
C.S. Aravinda and F.T Farrell  proved Conjecture \ref{gene.hyp} for the Cayley case in \cite{AF03} where they used Theorem \ref{locally.exto} and Theorem \ref{finite sheeted tang} to construct exotic smooth structures on $M_{1}^{2m}$ supporting almost $\frac{1}{4}$-pinched negatively curved metric. Here is the result :}
\end{remark}
\begin{theorem}\label{cay.hyp}
Let $M^{16}$ be any closed Cayley hyperbolic manifold. Given a positive real number $\displaystyle{\epsilon}$, then there exists a finite sheeted cover $\mathcal{N}^{16}$ of $M^{16}$ such that the following is true for any finite sheeted cover $N^{16}$ of $\mathcal{N}^{16}$.
\begin{itemize}
\item[(i)] $N^{16}\#\Sigma^{16}$ is not diffeomorphic to $N^{16}$;
\item[(ii)]$N^{16}\#\Sigma^{16}$ supports a negatively curved Riemannian metric whose sectional curvatures all lie in the closed interval $[-4 -\displaystyle{\epsilon}, -1]$.
\end{itemize}
\end{theorem}
\begin{remark}\rm{
\indent
\begin{itemize}
\item[\bf{1.}] Here $\Sigma^{16}$ is the unique closed, oriented smooth $16$-dimensional manifold which is homeomorphic but not diffeomorphic to the sphere $\mathbb{S}^{16}$. The existence and uniqueness of $\Sigma^{16}$ is a consequence of the following Proposition \ref{cay.pro}.
\item[\bf{2.}] Theorem \ref{cay.hyp} provides counterexamples to smooth rigidity Problem \ref{prorig}.
\end{itemize}}
\end{remark}
\begin{proposition}\label{cay.pro}\cite{AF03}
The group of smooth homotopy spheres $\Theta_{16}$ is cyclic of order $2$.
\end{proposition}
\begin{proof}
We have the following surgery exact sequence from \cite{Wal71}:
\begin{center}
$0\rightarrow \Theta_{16} \rightarrow \pi_{16}(F/O) \rightarrow L_{16}(O)=\mathbb{Z}$.
\end{center}
This sequence together with the fact that $\Theta_{16}$ is a finite group show that
$\Theta_{16}$ can be identified with the subgroup $S$ of $\pi_{16}(F/O)$ consisting of all elements having finite order. Next consider the exact sequence
\begin{center}
$\pi_{16}(O)\stackrel{J}{\rightarrow}\pi_{16}(F)\rightarrow \pi_{16}(F/O)\rightarrow \pi_{15}(O)=\mathbb{Z}$.
\end{center}
This sequence and the fact that $\pi_{16}(F)=\pi^{s}_{16}$ is a finite group show that $S$ can be identified with cokernel of $J$. Recall now that Adams \cite{Ada66} proved that $J$ is monic. This result together with the facts that $\pi_{16}(O)=\mathbb{Z}_2$ and $\pi^{s}_{16}=\mathbb{Z}_2\oplus \mathbb{Z}_2$ (\cite{Tod62}) show that $\Theta_{16}=\mathbb{Z}_2$. 
\end{proof}
To prove Theorem \ref{cay.hyp}, we have to solve the following two problem :
\begin{itemize}
\item[\bf{1.}] How to put an “almost $\frac{1}{4}$-pinched negatively curved metric on $N^{4m}\#\Sigma^{4m}$.
\item[\bf{2.}] How to show that $N^{16}\#\Sigma^{16}$ is not diffeomorphic to $N^{16}$. ($N^{16}\#\Sigma^{16}$ is clearly homeomorphic to $N^{16}$.) 
\end{itemize}

To solve the first problem, in view of Lemma \ref{pinmetric}, it suffices to construct a $1$-parameter family $b_{r}( , )$ of Riemannian metrics on $\mathbb{R}^{16}$ indexed by $r\in [e,\infty)$ which satisfy the following properties:
\begin{itemize}
\item[(i)] The sectional curvatures of $b_{r}( , )$ lie in the closed interval $[-4 -\displaystyle{\epsilon}(r), -1]$ where $\displaystyle{\epsilon}(r)>0$ and $\displaystyle{\epsilon}(r)\to 0$ as $r \to \infty$.
\item[(ii)] The ball of radius $r$ about $0$ in $(\mathbb{R}^{16},b_{r})$ is isometric to a ball of radius $r$ in the real hyperbolic space $\mathbb{R}\textbf{H}^{16}$.
\item[(iii)] The complement of the ball of radius $r^2$ about $0$ in $(\mathbb{R}^{16},b_{r})$ is isometric to the complement of a ball of radius $r^2$ in $\mathbb{O}\textbf{H}^{2}$.
\end{itemize} 
\begin{remark}\rm{
\indent
\begin{itemize}
\item[\bf{1.}] C.S. Aravinda and F.T. Farrell \cite{AF03} constructed these metrics using the explicit description of the Riemannian curvature tensor for $\mathbb{O}\textbf{H}^{2}$ given in \cite{BG72}. Now  by using Lemma \ref{pinmetric}, we can put “almost $\frac{1}{4}$-pinched negatively curved Riemannian metric on $N^{16}\#\Sigma^{16}$ provided $N^{16}$ has sufficiently large injectivity radius. Here $N^{16}$ is a closed, orientable Cayley hyperbolic manifold. This injectivity radius condition is satisfied when we pass to sufficiently large finite sheeted covers of $N^{16}$ since $\pi_1(N^{16})$ is a residually finite group.
\item[\bf{2.}] Theorem \ref{finite sheeted tang} gives a finite sheeted cover $\mathcal{N}^{16}$ of $M^{16}$ and a nonzero degree tangential map $f :\mathcal{N}^{16}\to \mathbb{O}\textbf{P}^{2}$. And we can arrange that $\mathcal{N}^{16}$ has arbitrarily large preassigned injectivity radius $r$ by taking larger covers since $\pi_1(\mathcal{N}^{16})$ is residually finite group. Once $r$ is determined, this is the manifold $\mathcal{N}^{16}$ in Theorem \ref{cay.hyp}.
\end{itemize}}
\end{remark}
The argument in Theorem \ref{com.con1} is now easily adapted to yield the following lemma since all the relevant Theorems \ref{kirby}, \ref{mostow} and \ref{gentoprigd} remain valid.
\begin{lemma}\label{cay.lemma}\cite{AF03}
Let $N^{16}$ be any finite sheeted cover of $\mathcal{N}^{16}$. If $N^{16}\#\Sigma^{16}$ is diffeomorphic to $N^{16}$, then $\mathbb{O}\textbf{P}^{2}\#\Sigma^{16}$ is concordant to $\mathbb{O}\textbf{P}^{2}$.
\end{lemma}
\begin{remark}\rm{
The octave projective plane $\mathbb{O}\textbf{P}^{2}$ is the mapping cone of the Hopf map $p : \mathbb{S}^{15}\to \mathbb{S}^8$. Let $\phi : \mathbb{O}\textbf{P}^{2}\to \mathbb{S}^{16}$ be the collapsing map obtained by identifying $\mathbb{S}^{16}$ with $\mathbb{O}\textbf{P}^{2}/ \mathbb{S}^8 $ in an orientation preserving way.  By making delicate use of some calculations of Toda \cite{Tod62} on the stable homotopy groups of spheres, C.S. Aravinda and F.T Farrell proved the following :}
\end{remark}
\begin{lemma}\label{cay.lemma1}\cite{AF03}
The homomorphism $\phi^* : [\mathbb{S}^{16}, Top/O]\to [\mathbb{O}\textbf{P}^{2}, Top/O]$ is monic.
\end{lemma}
\begin{remark}\rm{
\indent
\begin{itemize}
\item[\bf{1.}]
Lemma \ref{cay.lemma1} implies that $\mathbb{O}\textbf{P}^{2}\#\Sigma^{16}$  is not concordant to $\mathbb{O}\textbf{P}^{2}$ since the concordance classes of smooth structures on a smooth manifold $X$ are in bijective correspondence with $[X, Top/O]$ provided $\dim X > 4$. Thus assertion (i) of the Theorem \ref{cay.hyp} is a direct consequence of Lemma \ref{cay.lemma} and Lemma \ref{cay.lemma1}. Since Borel \cite{Bor63} has constructed closed Riemannian manifolds  $M^{16}$ whose universal cover is $\mathbb{O}\textbf{H}^2$, Theorem \ref{cay.hyp} produces the examples for Conjecture \ref{gene.hyp}.
\item[\bf{2.}] C.S. Aravinda and F.T. Farrell also produce the examples for Conjecture \ref{gene.hyp} for the quaternionic case in \cite{AF04}. Here is the result :
\end{itemize}}
\end{remark}
\begin{theorem}\label{quot.hyp}
Let $M^{4m}$ be any closed Riemannian manifold whose universal cover is $\mathbb{H}\textbf{H}^m$ where $m = 2$, $4$ or $5$. Then there exists a finite sheeted cover $\mathcal{N}^{4m}$ of $M^{4m}$ such that the following is true for any finite sheeted cover $N^{4m}$ of $\mathcal{N}^{4m}$. Let $\Sigma^{4m}$ be any exotic homotopy sphere satisfying the extra constraint when $m = 5$ that $6[\Sigma^{20}]\neq 0$. Then
\begin{itemize}
\item[(i)] $N^{4m}\#\Sigma^{4m}$ is not diffeomorphic to $N^{4m}$;
\item[(ii)]$N^{4m}\#\Sigma^{4m}$ supports a Riemannian metric whose sectional curvatures are all negative.
\end{itemize}
\end{theorem}
\begin{addendum}\label{quot.add}
The three manifolds $N^{20}$, $N^{20}\#\Sigma^{20}_1$ and $N^{20}\#\Sigma^{20}_2$ are pairwise non-diffeomorphic
when order $[\Sigma^{20}_1] =8$ and order $[\Sigma^{20}_2] =4$ in $\Theta_{20}$.
\end{addendum}
\begin{remark}\rm{
Recall that Kervaire and Milnor \cite{KM63} showed that each $\Theta_{m}$ is a finite group; in particular, the abelian groups $\Theta_{8}$, $\Theta_{16}$ and $\Theta_{20}$ have orders 2, 2, 24, respectively.}
\end{remark}
To prove Theorem \ref{quot.hyp}, we have to solve the following two problem :
\begin{itemize}
\item[\bf{1.}] How to put a negatively curved metric on $N^{4m}\#\Sigma^{4m}$.
\item[\bf{2.}] How to show that $N^{4m}\#\Sigma^{4m}$ is not diffeomorphic to $N^{4m}$. ($N^{4m}\#\Sigma^{4m}$ is clearly homeomorphic to $N^{4m}$.) 
\end{itemize}
To solve the first problem by proving the following theorem :
\begin{theorem}\label{quo.neg}\cite{AF04}
Let $M^{4m}$ be any closed Riemannian manifold whose universal cover is  $\mathbb{H}\textbf{H}^m$ $(m\geq 2)$. Then there exists a finite sheeted cover $\mathcal{M}^{4m}$ such that the following is true for any finite sheeted cover $N^{4m}$ of $\mathcal{M}^{4m}$. Let $\Sigma^{4m}$ be any homotopy sphere. Then $N^{4m}\#\Sigma^{4m}$ supports a Riemannian metric whose sectional curvatures are all negative.
\end{theorem}
To prove Theorem \ref{quo.neg}, in view of Lemma \ref{pinmetric}, it suffices to construct a $1$-parameter family $b_{r}( , )$ of Riemannian metrics on $\mathbb{R}^{4m}$ indexed by $r\in [e,\infty)$ which satisfy the following properties:
\begin{itemize}
\item[(i)] The sectional curvatures of $b_{r}( , )$ are all negative provided $r$ is sufficiently large.
\item[(ii)] The ball of radius $r$ about $0$ in $(\mathbb{R}^{4m},b_{r})$ is isometric to a ball of radius $r$ in the real hyperbolic space $\mathbb{R}\textbf{H}^{4m}$.
\item[(iii)] The complement of the ball of radius $r^2$ about $0$ in $(\mathbb{R}^{4m},b_{r})$ is isometric to the complement of a ball of radius $r^2$ in $\mathbb{H}\textbf{H}^{m}$.
\end{itemize} 
\begin{remark}\rm{
C.S. Aravinda and F.T Farrell in \cite{AF04} constructed these metrics which is similar to that made in \cite{FJ94, AF03} but the verification of property (i) requires a different technique. They used O'Neill’s semi-Riemannian submersion formula \cite[pp.213]{Ont83} to estimate the curvatures rather than getting an explicit formula as is done in \cite{FJ94, AF03}. By using this result together with Lemma \ref{pinmetric}, they put a negatively curved Riemannian metric on $N^{4m}\#\Sigma^{4m}$ provided $N^{4m}$ has sufficiently large injectivity radius. Here $N^{4m}$ is a closed, orientable quarternionic hyperbolic
manifold. This injectivity radius condition is satisfied when we pass to sufficiently large finite sheeted covers of $N^{4m}$ since $\pi_1(N^{4m})$ is a residually finite group. Motivated by this construction, C.S. Aravinda and F.T Farrell \cite{AF04} asked, more generally, the following naive question :}
\end{remark}
\begin{question}\label{queneg}
Let $b( , )$ be a complete Riemannian metric on $\mathbb{R}^m$ whose sectional curvatures are bounded above by $-1$ and let $r$ be a positive real number. Does there exist a second complete, negatively curved Riemannian metric $\bar{b}( , )$ on $\mathbb{R}^m$ satisfying :
\begin{itemize}
\item[(i)] A metric ball of radius $r$ in $\mathbb{R}\textbf{H}^{m}$ can be isometrically embedded in $(\mathbb{R}^{m},\bar{b})$.
\item[(ii)] $b$ and $\bar{b}$ agree at infinity, i.e., off some compact subset of $\mathbb{R}^m$.
\end{itemize}
\end{question}
\begin{remark}\rm{
\indent
\begin{itemize}
\item[\bf{1.}] An affirmative answer to Question \ref{queneg} would allow to extend Theorem \ref{quo.neg} to arbitrary closed negatively curved manifolds $M$ with residually finite fundamental groups by setting $b$ equal to the Riemannian metric on the universal cover of $M$.
\item[\bf{2.}] Theorem \ref{finite sheeted tang} gives a finite sheeted cover $\mathcal{N}^{4m}$ of $M^{4m}$ and a nonzero degree tangential map $f :\mathcal{N}^{4m}\to \mathbb{H}\textbf{P}^{m}$. Since $\pi_1(\mathcal{N}^{4m})$ is residually finite group, we can arrange that $\mathcal{N}^{4m}$ has arbitrarily large preassigned injectivity radius $r$ by taking larger covers. Once $r$ is determined, this is the manifold $\mathcal{N}^{4m}$ in Theorem \ref{quot.hyp}.
\end{itemize}}
\end{remark}
The argument in Theorem \ref{com.con1} is now easily adapted to yield the following lemma since all the relevant Theorems \ref{kirby}, \ref{mostow} and \ref{gentoprigd} remain valid.
\begin{lemma}\label{qout.lem}\cite{AF04}
Let $N^{4m}$ be any finite sheeted cover of $\mathcal{N}^{4m}$. If $N^{4m}\#\Sigma^{4m}$ is diffeomorphic to $N^{4m}$, then $\mathbb{H}\textbf{P}^{m}\#\Sigma^{4m}$ is concordant to $\mathbb{H}\textbf{P}^{m}$.
\end{lemma}
C.S. Aravinda and F.T. Farrell \cite{AF04} proved the following result using the solution of Adam’s Conjecture together with calculations of the stable homotopy groups of spheres through dimension $20$ (\cite{Rav86, Tod62, MM79}):
\begin{corollary}\label{qout.corr1}
Let $\phi: \mathbb{H}\textbf{P}^{m}\to \mathbb{S}^{4m}$ denote a degree 1 map. Then $\phi^* :  \Theta_{4m}=[\mathbb{S}^{4m}, Top/O]\to [\mathbb{H}\textbf{P}^{m}, Top/O]$ is non-zero if $m = 2$, $4$ or $5$. In fact, $\phi^*|_{\Theta_{4m}^2 }$ is monic when $m = 2$, $4$  and its kernel has order less than or equal to $2$ when $m = 5$.
\end{corollary}
\begin{remark}\rm{
\indent
\begin{itemize}
\item[\bf{1.}] Corollary \ref{qout.corr1} implies that $\mathbb{H}\textbf{P}^{m}\#\Sigma^{4m}$ is not concordant to $\mathbb{H}\textbf{P}^{m}$ when $m = 2$, $4$ or $5$ since the concordance classes of smooth structures on a smooth manifold $X$ are in bijective correspondence with $[X,Top/O]$ provided $\dim X > 4$. Thus assertion (i) of Theorem \ref{quot.hyp} is a direct consequence of Lemma \ref{qout.lem} and Corollary \ref{qout.corr1}. And using Theorem \ref{locally.exto} in place of Lemma \ref{qout.lem}, Addendum \ref{quot.add} also follows from Corollary \ref{qout.corr1}. Since Borel \cite{Bor63} has constructed closed Riemannian manifolds  $M^{4m}$ (for each $m\geq 2$) whose universal cover is $\mathbb{H}\textbf{H}^m$, Theorem \ref{quot.hyp} produces the examples for Conjecture \ref{gene.hyp}.
\item[\bf{2.}] Theorem \ref{quot.hyp} also provides counterexamples to smooth rigidity Problem \ref{prorig}.
\item[\bf{2.}] Finally Theorem \ref{quot.hyp} and Theorem \ref{cay.hyp} together with Corlette's Superrigidity Theorem \ref{Corlette} were used by C.S. Aravinda and F.T. Farrell \cite{AF05} to answer positively the following question posed by Petersen in his text book \cite[pp. 239-240]{Pet98}:
\end{itemize}}
\end{remark}   
\begin{question}\label{pet.quest}
Are there any compact rank $1$ manifolds of non-positive sectional curvature that do not admit a metric with non-positive curvature operator?
\end{question}
Recall that the curvature operator at a point $p\in M$ of a manifold $(M,\langle,\rangle)$ is a linear map on the space  $\Lambda^{2}(T_{p}M)$ to itself where $T_{p}M$ is the tangent space to $M$ at $p$. If $X$, $Y$, $Z$, $W \in T_{p}M$, then there is a scalar product $\langle, \rangle$ on $\Lambda^{2}(T_{p}M)$ that is given by the formula $$\langle X\wedge Y, Z\wedge W \rangle= \langle X, Z \rangle \langle Y, W \rangle - \langle X, W \rangle \langle Y, Z \rangle$$ and is extended by linearity to all of $\Lambda^2(T_{p}M)$. Then the curvature operator $\mathcal{R}$ of $M$ is defined by $$\langle \mathcal{R}(X\wedge Y), Z\wedge W \rangle := \langle R(X,Y)W, Z \rangle$$ where $R(X,Y)W$ is the Riemann curvature tensor of $M$.\\
\indent The curvature operator is said to be non-positive if all its eigenvalues are nonpositive. An elementary linear algebra argument shows that if the curvature operator is non-positive, then all the sectional curvatures of $M$ are also non-positive.\\
\indent Recall that any closed Riemannian manifold all of whose sectional curvatures are negative is automatically a rank $1$ manifold. Hence the following result shows that the answer to Question \ref{pet.quest} is yes :
\begin{theorem}\label{curopra}\cite{AF05}
There exist closed Riemannian manifolds whose sectional curvatures are all negative, but which do not admit any Riemannian metric whose curvature
operator is non-positive. In fact, such manifolds exist (at least) in dimensions 8, 16 and 20. Furthermore, given any $\epsilon > 0$, such examples can be constructed (at least in dimension 16) whose sectional curvatures all lie in the interval $[-4, -1+\displaystyle{\epsilon}]$.
\end{theorem}
\begin{remark}\rm{
Given a closed quaternionic hyperbolic manifold of quaternionic dimension 2, 4 or 5 (i.e., real dimension 8, 16 or 20) or a closed Cayley hyperbolic manifold $M$ (whose real dimension is 16), C.S. Aravinda and F.T. Farrell constructed, in Theorem \ref{quot.hyp} and Theorem \ref{cay.hyp}, exotic differential structures carrying negatively curved Riemannian metrics on certain finite covers $\widehat{M}$ of $M$; that is, there exist closed negatively curved manifolds that are homeomorphic but not diffeomorphic to the natural locally symmetric structure on $\widehat{M}$. Applying Theorem \ref{Corlette}, we conclude that the exotic differential structures on $\widehat{M}$ can not support metrics with non-positive curvature operator. Also, in the case where $M$ is Cayley hyperbolic the sectional curvatures of the Riemannian metric constructed in Theorem \ref{cay.hyp} for the exotic differential structure on $\widehat{M}$ all lie in the interval $[-4, -1+\displaystyle{\epsilon}]$. This proves Theorem \ref{curopra}.
}
\end{remark}
\begin{remark}\rm{
\indent
\begin{itemize}
\item[\bf{1.}] We can recover exotic smoothings of Theorem \ref{hype.smooth}, Theorem \ref{complex.hyp}, Theorem \ref{quot.hyp} and Theorem \ref{cay.hyp} by applying Theorem \ref{locally.exto} to real, complex, quaternionic and cayley hyperbolic manifolds. This is trivial in the real
case, since the dual space is a sphere. Since the dual symmetric space of complex hyperbolic manifold is complex projective space , exotic smoothings of Theorem \ref{complex.hyp} follows from Theorem \ref{locally.exto} and Lemma \ref{lem.con}. Similarly, exotic smoothings of Theorem \ref{quot.hyp} and Theorem \ref{cay.hyp} follow from Theorem \ref{locally.exto}, Corollary \ref{qout.corr1} and Lemma \ref{cay.lemma1} since the dual symmetric spaces of quarternionic and Cayley hyperbolic manifolds are the quaternionic projective space and Cayley projective plane respectively.
\item[\bf{2.}] As we have observed in the above remark (1) and Remark \ref{compopen}, by using Mostow's Rigidity Theorem \ref{mostow}, Farrell-Jones Topological Rigidity Theorem \ref{gentoprigd}, Theorem \ref{kirby}, Theorem \ref{tubular} and Theorem \ref{finite sheeted tang} together with the fundamental paper of Kervaire and Milnor \cite{KM63}, the problem of detecting when $M^n\#\Sigma^n$ and $M^n$ are not diffeomorphic, where $M$ is a closed locally symmetric space of noncompact type  such that the universal cover of $M$ has no 2-dimensional metric factor projecting to a closed subset of $M$ is essentially reduced to look at the problem of detecting exotic structure on the dual symmetric space $M_u$ of $M$.
\end{itemize}}
\end{remark}
\section{\large Topology on the Space of all Riemannian Metrics}
In this section we want to study the basic topological properties of the space of all negatively curved Riemannian metrics and the Teichmuller space of negatively curved metrics on a manifold. Let us introduce some notation.\\
Let $M^n$ be a closed smooth manifold and let $\mathcal{MET}(M^n)$ be the space of all Riemannian metrics on $M^n$ with the smooth topology. Since any two metrics $g_0$ and $g_1$ are connected by a line segment $tg_0+(1-t)g_1$, the space $\mathcal{MET}(M^n)$ is contractible. A subspace of metrics whose sectional curvatures lie in some interval (closed, open, semi-open) will be denoted by placing a superscript on $\mathcal{MET}(M^n)$. For example, $\mathcal{MET}^{sec<\epsilon}(M^n)$ denotes the subspace of $\mathcal{MET}(M^n)$ of all Riemannian metrics on $M^{n}$ that have all sectional curvatures less that $\epsilon$. Thus saying, for instance, that $M^n$ admits a negatively curved metric is equivalent to saying that $\mathcal{MET}^{sec<0}(M^n)\neq \emptyset$. Or, saying that all sectional curvatures of a Riemannian metric $g$ lie in the interval $[a, b]$ is equivalent to saying that $g\in \mathcal{MET}^{a\leq sec\leq b}(M^n)$. Note that $\mathcal{MET}^{sec=-1}(
M^n)$ is the space of hyperbolic metrics $\mathcal{H}yp(M^n)$ on $M^n$.

A natural question about a closed negatively curved manifold M posed by K. Burns and A. Katok (\cite[Question 7.1]{BK85}) is the following :
\begin{question}
 Is the space $\mathcal{MET}^{sec<0}(M^n)$ path connected?
\end{question}
\begin{remark}\rm{
\indent
\begin{itemize}
\item [\bf{1.}] In dimension two, Richard Hamilton's theorem on Ricci Flow \cite{Ham88} shows that $\mathcal{H}yp(M^2)$ is a deformation retract of $\mathcal{MET}^{sec<0}(M^n)$. But $\mathcal{H}yp(M^2)$ fibers over the Teichm¨uller space $\mathcal{T}(M^2)\cong \mathbb{R}^{6\mu-6}$ ($\mu$ is the genus of $M^2$), with contractible fiber $\mathcal{D}_0= \mathbb{R}^{+}\times DIFF_{0}(M^2)$ \cite{EE69} (Here $DIFF_0(M^2)$ consists of all self diffeomorphisms of $M^2$ which are homotopic to $id_{M^2}$). Therefore, $\mathcal{H}yp(M^2)$ and $\mathcal{MET}^{sec<0}(M^2)$ are contractible.
\item[\bf{2.}] It was shown by Farrell and Ontaneda \cite{FO10} that, for $n\geq 10$, $\mathcal{MET}^{sec<0}(M^n)$ has infinitely many path components. Moreover, they showed that all the groups $\pi_{2p-4}(\mathcal{MET}^{sec<0}(M^n))$ are non-trivial for every prime number $p > 2$, and such that $p < \frac{n+5}{6}$. (In fact, these groups contain the infinite sum $(\mathbb{Z}_p)^{\infty}$ of $\mathbb{Z}_p = \mathbb{Z}/p\mathbb{Z}$ 's, and hence they are not finitely generated.) They also showed that $\pi_{1}(\mathcal{MET}^{sec<0}(M^n))$ is not finitely generated when $n\geq 14$. These results about $\pi_k$ are true for each path component of $\mathcal{MET}^{sec<0}(M^n)$ i.e., relative to any base point. 
\end{itemize}}
\end{remark}
Before we state Farrell and Ontaneda \cite{FO10} Main Theorem, we need some definitions.\\

Denote by $DIFF(M)$ the group of all smooth self-diffeomorphisms of $M$. We have that $DIFF(M)$ acts on $\mathcal{MET}(M)$ pulling-back metrics: $\phi g = (\phi^{-1})^{*}g =\phi_{*}g$, for $g\in \mathcal{MET}(M)$ and $\phi \in DIFF(M)$, that is, $\phi g$ is the metric such that $\phi : (M, g)\to (M, \phi g)$ is an isometry. Note that $DIFF(M)$ leaves invariant all spaces $\mathcal{MET}^{sec\in I}(M)$, for any $I\subset \mathbb{R}$. For any metric $g$ on $M$ we denote by $DIFF(M)g$ the orbit of $g$ by the action of $DIFF(M)$. We have a map $\Lambda_g : DIFF(M)\to \mathcal{MET}(M)$, given by $\Lambda_g(\phi) = \phi_{*}g$. Then the image of $\Lambda_g$ is the orbit $DIFF(M)g$ of $g$. And $\Lambda_g$ of course naturally factors through $\mathcal{MET}^{sec\in I}(M)$, if $g\in \mathcal{MET}^{sec\in I}(M)$. Note that if $\dim M\geq 3$ and $g\in \mathcal{MET}^{sec=-1}(M)$, then the statement of Mostow's Rigidity Theorem is equivalent to saying that the map $\Lambda_g : DIFF(M)\to \mathcal{MET}^{sec=-1}(M)=\mathcal{H}
yp(M)$ is a surjection. Here is the statement of Farrell and Ontaneda \cite{FO10} main result :
\begin{theorem}\label{moduli}
Let $M$ be a closed smooth $n$-manifold and let $g$ be a negatively curved Riemannian metric on $M$. Then we have that:
\begin{itemize}
\item[(i)] the map $\pi_{0}(\Lambda_g) : \pi_{0}(DIFF(M) )\to \pi_{0}(\mathcal{MET}^{sec<0}(M))$ is not constant, provided $n\geq 10$.
\item[(ii)] the homomorphism $\pi_{1}(\Lambda_g) : \pi_{1}(DIFF(M) )\to \pi_{1}(\mathcal{MET}^{sec<0}(M))$  is non-zero, provided $n\geq 14$.
\item[(iii)] For $k = 2p-4$, $p$ prime integer and $1 < k\leq \frac{n-8}{3}$, the homomorphism $\pi_{k}(\Lambda_g) : \pi_{k}(DIFF(M) )\to \pi_{k}(\mathcal{MET}^{sec<0}(M))$ is non-zero.
\end{itemize}
\end{theorem}
\begin{addendum}\label{moduliadd}\cite{FO10}
We have that the image of $\pi_{0}(\Lambda_g)$ is infinite and in cases (ii), (iii) mentioned in the Theorem \ref{moduli}, the image of $\pi_{k}(\Lambda_g)$ is not finitely generated. In fact we have:
\begin{itemize}
\item[(i)] For $n\geq 10$, $\pi_{0}(DIFF(M))$ contains $(\mathbb{Z}_2)^{\infty}$, and $\pi_{0}(\Lambda_g)_{|(\mathbb{Z}_2)^{\infty}}$ is one-to-one.
\item[(ii)] For $n\geq 14$, the image of $\pi_{1}(\Lambda_g)$ contains $(\mathbb{Z}_2)^{\infty}$.
\item[(iii)] For $k = 2p-4$, $p$ prime integer and $1 < k\leq \frac{n-8}{3}$, the image of $\pi_{k}(\Lambda_g)$ contains $(\mathbb{Z}_p)^{\infty}$.
\end{itemize}
\end{addendum}
For $a < b < 0$ the map $\Lambda_g$ factors through the inclusion map $\mathcal{MET}^{a\leq sec\leq b}(M) \hookrightarrow \mathcal{MET}^{sec<0}(M)$ provided $g\in Met^{a\leq sec\leq b}(M)$. Therefore we have
\begin{corollary}\label{modulicorr1}\cite{FO10}
Let $M$ be a closed smooth $n$-manifold, $n\geq 10$. Let $a < b < 0$ and assume that $\mathcal{MET}^{a\leq sec\leq b}(M)$ is not empty. Then the inclusion map $\mathcal{MET}^{a\leq sec\leq b}(M) \hookrightarrow \mathcal{MET}^{sec<0}(M)$ is not null-homotopic. Indeed, the induced maps, at the $k$-homotopy level, are not constant for $k = 0$, and non-zero for the cases (ii), (iii) mentioned in the Theorem \ref{moduli}. Furthermore, the image of these maps satisfy a statement analogous to the one in the Addendum \ref{moduliadd}  to the Theorem \ref{moduli}.
\end{corollary}
If $a = b = -1$ we have
\begin{corollary}\label{modulicorr2}\cite{FO10}
Let $M$ be a closed hyperbolic $n$-manifold, $n\geq 10$. Then the inclusion map $\mathcal{H}yp(M) \hookrightarrow \mathcal{MET}^{sec<0}(M)$ is not null-homotopic. Indeed, the induced maps, at the $k$-homotopy level, are not constant for $k = 0$, and non-zero for the cases (ii), (iii) mentioned in the Theorem \ref{moduli}. Furthermore, the image of these maps satisfy a statement analogous to the one in the Addendum \ref{moduliadd} to the Theorem \ref{moduli}.
\end{corollary}
\begin{remark}\rm{
\indent
 \begin{itemize}
\item [\bf{1.}] Hence, taking $k = 0$ (i.e., $p = 2$) in Corollary \ref{modulicorr2}, we get that for any closed hyperbolic manifold $(M^n, g)$, $n\geq 10$, there is a hyperbolic metric $g'$ on $M$ such that $g$ and $g'$ cannot be joined by a path of negatively curved metrics.
\item [\bf{2.}] Also, taking $a = -1-\epsilon$, $b = -1 (0\leq \epsilon)$ in Corollary \ref{modulicorr1} we have that the space $\mathcal{MET}^{-1-\epsilon \leq sec\leq -1}(M)$ of $\epsilon$-pinched negatively curved Riemannian metrics on $M$ has infinitely many path components, provided it is not empty and $n\geq 10$. And the homotopy groups $\pi_k(\mathcal{MET}^{-1-\epsilon \leq sec\leq -1}(M))$, are non-zero for the cases (ii), (iii) mentioned in the Theorem \ref{moduli}. Moreover, these groups are not finitely generated.
\item [\bf{3.}] The restriction on $n = dim ~M$ given in the Theorem \ref{moduli}, its Addendum \ref{moduliadd} and its Corollaries \ref{modulicorr1} and \ref{modulicorr2} are certainly not optimal. In particular, in Theorem \ref{moduli} (iii) it can be improved to $1 \le k \le \frac{n-10}{2}$ by using Igusa’s “Surjective Stability Theorem” (\cite[p. 7]{Igu88}).
\item[\bf{4.}] Another interesting application of the Theorem \ref{moduli} shows that the answer to the following natural question is negative:
\end{itemize}}
\end{remark}
\begin{question}\label{}\cite{FO10}
Let $E\to B$ be a fibre bundle whose fibres are diffeomorphic to a closed negatively curved manifold $M^n$. Is it always possible to equip its fibres with negatively curved Riemannian metrics (varying continuously from fibre to fibre)?
\end{question}
\begin{remark}\rm{
\indent
\begin{itemize}
\item[\bf{1.}] The negative answer is gotten by setting $B = \mathbb{S}^{k+1}$, where $k$ is as in the Theorem \ref{moduli} case (iii) (or $k = 0$, $1$, case (i), (ii)), and the bundle $E\to \mathbb{S}^{k+1}$ is obtained by the standard clutching construction using an element $\alpha \in \pi_k(DIFF(M))$ such that $\pi_{1}(\Lambda_g)(\alpha)\neq 0$, for every negatively curved Riemannian metric $g$ on $M$. F. T. Farrell and P. Ontaneda \cite{FO10} method for proving the Theorem \ref{moduli} (in particular Theorem \ref{moduthe1} below) one sees that such elements $\alpha$, which are independent of $g$, exist in all cases (i), (ii), (iii).
\item [\bf{2.}] Theorem \ref{moduli} follows from Theorems \ref{moduthe1} and \ref{moduthe2} below. Before we state these results we need some definitions and constructions. 
\end{itemize}}
\end{remark}
For a manifold $N$, let $P(N)$ be the space of topological pseuso-isotopies of $N$, that is, the space of all homeomorphisms $N\times I\to N\times I$, $I=[0, 1]$, that are the identity on $(N\times {0})\cup (\partial N\times I)$. We consider $P(N)$ with the compact-open topology. Also, $P^{diff}(N)$ is the space of all smooth pseudo-isotopies on $N$, with the smooth topology. Note that $P^{diff}(N)$ is a subset of $P(N)$. The map of spaces $P^{diff}(N)\to P(N)$ is continuous and will be denoted by $\iota_N$, or simply by $\iota$. Let $DIFF(N,\partial)$ denotes the subspace of $DIFF(N)$ of all self-diffeomorphism of $N$ which are the identity on $\partial N$. Note that $DIFF(N\times I, \partial)$ is the subspace of $P^{diff}(N)$ of all smooth pseudo-isotopies whose restriction to $N\times {1}$ is the identity. The restriction of $\iota_N$ to $DIFF(N\times I, \partial)$ will also be denoted by $\iota_N$. The map $\iota_N : DIFF(N\times I, \partial)\to P(N)$ is one of the ingredients in the statement Theorem \
ref{moduthe1}.\\ We will also need the following construction. Let $M$ be a negatively curved n-manifold. Let $\alpha: \mathbb{S}^1\to M$ be an embedding. Sometimes we will denote the image $\alpha(\mathbb{S}^1)$ just by $\alpha$. We assume that the normal bundle of $\alpha$ is orientable, hence trivial. Let $V : \mathbb{S}^1\to TM\times...\times TM$, be an orthonormal trivialization of this bundle: $V (z) = (v_1(z), ..., v_{n-1}(z))$ is an orthonormal base of the orthogonal complement of $\alpha(z)'$ in $T_zM$. Also, let $r > 0$, such that $2r$ is less that the width of the normal geodesic tubular neighborhood of $\alpha$. Using $V$, and the exponential map of geodesics orthogonal to $\alpha$, we identify the normal geodesic tubular neighborhood of width $2r$ minus $\alpha$, with $\mathbb{S}^1\times \mathbb{S}^{n-2}\times (0, 2r]$. Define $\Phi=\Phi^{M}(\alpha,V,r) : DIFF(\mathbb{S}^1\times \mathbb{S}^{n-2}\times I, \partial)\to DIFF(M)$ in the following way. For $\phi \in DIFF(\mathbb{S}^1 \times \mathbb{S}
^{n-2} \times I, 
\partial)$ let $\Phi(\phi) : M \to M$ be the identity outside $\mathbb{S}^1\times \mathbb{S}^{n-2} \times[r, 2r]\subset M$, and $\Phi(\phi)=\lambda^{-1}\phi\lambda$, where $\lambda (z, u, t) = (z, u, \frac{t−r}{r})$, for $(z,u,t)\in \mathbb{S}^1\times \mathbb{S}^{n-2}\times[r, 2r]$. Note that the dependence of $\Phi(\alpha,V,r)$ on $\alpha$ and $V$ is essential, while its dependence on $r$ is almost irrelevant.\\
We denote by $g$ the negatively curved metric on $M$. Hence we have the following diagram :
\begin{equation*}
\begin{CD}
 DIFF(\mathbb{S}^{1}\times \mathbb{S}^{n-2}\times I, \partial)       @>\Phi>> DIFF(M)  @>\Lambda_{g}>> \mathcal{MET}^{sec<0}(M)\\
 @V \iota VV                                                           @.                   @.\\ 
 P(\mathbb{S}^{1}\times \mathbb{S}^{n-2})                              @.                   @. 
\end{CD}
\end{equation*}
where $\iota=\iota_{\mathbb{S}^{1}\times \mathbb{S}^{n-2}}$ and $\Phi=\Phi^{M}(\alpha,V,r)$.
\begin{theorem}\label{moduthe1}\cite{FO10}
Let $M$ be a closed $n$-manifold with a negatively curved metric $g$. Let $\alpha$, $V$, $r$ and $\Phi=\Phi(\alpha, V, r)$ be as above, and assume that $\alpha$ in not null-homotopic. Then $Ker (\pi_k(\lambda_g \Phi))\subset Ker (\pi_k(\iota))$, for $k < n-5$. Here $\pi_k(\Lambda_g \Phi)$ and $\pi_k(\iota)$ are the homomorphisms at the $k$-homotopy group level induced by $\Lambda_g \Phi$ and 
$\iota = \iota_{\mathbb{S}^1\times \mathbb{S}^{n-2}}$, respectively.
\end{theorem}
\begin{remark}\rm{
\indent
\begin{itemize}
 \item[\bf{1.}] In the statement of Theorem \ref{moduthe1} above, by $Ker (\pi_0(\Lambda_g \Phi))$ (for $k = 0$) we mean
the set $(\pi_0(\Lambda_g \Phi))^{-1}([g])$, where $[g]\in \pi_0(\mathcal{MET}^{sec<0}(M))$ is the connected component of the metric $g$.
\item[\bf{2.}] Hence to deduce the Theorem \ref{moduli} from Theorem \ref{moduthe1} we need to know that $\pi_k(\iota_{\mathbb{S}^1\times \mathbb{S}^{n-2}})$ is a non-zero homomorphism. Furthermore, to prove the Addendum \ref{moduliadd} to the Theorem \ref{moduli} we have to show that $\pi_k(DIFF(\mathbb{S}^1\times \mathbb{S}^{n-2} \times I, \partial))$ contains an infinite sum of $\mathbb{Z}_p$'s (resp. $\mathbb{Z}_2$'s) where $k=2p-4$, $p$ prime (resp. $k = 1$) and $\pi_k(\iota_{\mathbb{S}^1\times \mathbb{S}^{n-2}})$ restricted to this sum is one-to-one. This follows from the following result :
\end{itemize}}
\end{remark}
\begin{theorem}\label{moduthe2}\cite{FO10}
Let $p$ be a prime integer such that max $\{9, 6p-5\}< n$. Then for $k=2p-4$ we have that $\pi_k(DIFF(\mathbb{S}^1\times \mathbb{S}^{n-2}\times I, \partial))$ contains $(\mathbb{Z}_p)^{\infty}$ and $\pi_k(\iota_{\mathbb{S}^1\times \mathbb{S}^{n-2}})$ restricted to $(\mathbb{Z}_p)^{\infty}$ is one-to-one.
\end{theorem}
\begin{definition}\rm{
Let $\mathcal{D}(M)$ be the group $\mathbb{R}^{+}\times DIFF(M)$. The group $\mathcal{D}(M)$ acts on $\mathcal{MET}(M)$ by scaling and pulling-back metrics: $(\lambda, \phi)g=\lambda(\phi^{-1})^{*}g = \lambda \phi_{*}g$, for $g\in \mathcal{MET}(M)$ and $(\lambda, \phi)\in \mathcal{D}(M)$. The quotient space $\mathcal{M}(M) = \mathcal{MET}(M)/\mathcal{D}(M)$ is called the moduli space of metrics on $M$.}
\end{definition}
Let us go back to dimension two for a moment and let $\Sigma_g$ be an orientable two-dimensional manifold of genus $g > 1$. Recall that uniformization techniques (see \cite{EE69}, or, more recently, Hamilton's Ricci flow \cite{Ham88}) show that every Riemannian metric on $\Sigma_g$, $g > 1$, can be canonically deformed to a hyperbolic metric. Moreover, Hamilton's
Ricci flow \cite{Ham88} shows that every negatively curved metric on $\Sigma_g$, $g > 1$ can be canonically deformed (through negatively curved metrics) to a hyperbolic metric. Hence the space of all hyperbolic metrics on $\Sigma_g$ is canonically a deformation retract of the space of all negatively
curved Riemannian metrics on $\Sigma_g$. This deformation commutes with the action of $DIFF(\Sigma_g)$ (this is true at least for the Ricci flow), therefore, the Teichm¨uller space of $\Sigma_g$ is canonically a deformation retract of the space which is the quotient of all negatively curved
Riemannian metrics on $\Sigma_g$ by the action of the group of all smooth self-diffeomorphisms which are homotopic to the identity. Also, instead of considering the space of all negatively curved metrics we can consider the space of all pinched negatively curved metrics, or for
that matter, the space of all Riemannian metrics. These are the concepts that F. T. Farrell and P. Ontaneda  generalized in \cite{FO09}. Next, we give detailed definitions and introduce some notation.
\begin{definition}\rm{
Let $M$ be a closed smooth manifold. We denote by $DIFF_0(M)$ the subgroup of $DIFF(M)$ of all smooth diffeomorphisms of $M$ which are homotopic to the identity $I_{M}$ and by $\mathcal{D}_0(M)$ the group $\mathbb{R}^{+}\times DIFF_0(M)$. In \cite{FO09}, the Teichm¨uller space of metrics on $M$ is defined as the quotient space $\mathcal{T}(M)=\mathcal{MET}(M)/\mathcal{D}_0(M)$. Given $0\leq \epsilon \leq \infty $, let $\mathcal{MET}^{\epsilon}(M)$ denote the space of all $\epsilon$-pinched negatively curved Riemannian metrics on $M$, that is, $g\in \mathcal{MET}^{\epsilon}(M)$ if and only if there is a positive real number $\lambda$ such that $\lambda g$ has all its sectional curvatures in the interval $[−(1+\epsilon), -1]$. Note that a 0-pinched metric is a metric of constant negative sectional curvature and an $\infty$-pinched metric is just a negatively curved Riemannian metric. The quotient space $\mathcal{M}^{\epsilon}(M) = \mathcal{MET}^{\epsilon}(M)/\mathcal{D}(M)$ is called the moduli space of $\epsilon$-
pinched negatively curved metrics on $M$. Denote by $\kappa$ the quotient map $\mathcal{MET}^{\epsilon}(M) \to \mathcal{M}^{\epsilon}(M)$. Also, $\mathcal{T}^{\epsilon}(M)=\mathcal{MET}^{\epsilon}(M)/\mathcal{D}_0(M)$ is called the Teichm¨uller space of $\epsilon$-pinched negatively curved metrics on $M$. In particular, $\mathcal{T}^{\infty}(M)$ is the Teichm¨uller space of all negatively curved metrics on $M$. }
\end{definition}
Note that the inclusions $\mathcal{MET}^{\epsilon}(M)\hookrightarrow \mathcal{MET}(M)$ induce the inclusions $\mathcal{T}^{\epsilon}(M) \hookrightarrow \mathcal{T}(M)$. Also note that, for $\delta \geq \epsilon$, these inclusions factor as follows: $\mathcal{MET}^{\epsilon}(M) \hookrightarrow \mathcal{MET}^{\delta}(M)\hookrightarrow \mathcal{MET}^{\epsilon}(M)$  and $\mathcal{T}^{\epsilon}(M) \hookrightarrow \mathcal{T}^{\delta}(M) \hookrightarrow \mathcal{T}(M)$.
\begin{remark}\label{remtech}\rm{
\indent 
 \begin{itemize}
\item [\bf{1.}] If $M_g$ is an orientable two-dimensional manifold of genus $g>1$, then the original Teichm¨uller space of $M_g$ is denoted (in our notation) by $\mathcal{T}^{0}(M_g)$ and $\mathcal{T}^{0}(M_g)$ is homeomorphic to $\mathbb{R}^{6g-6}$ (see \cite{EL88}). Hence $\mathcal{T}^{0}(M_g)$ is contractible. By the uniformization techniques mentioned above (\cite{EE69, Ham88}), it follows that $\mathcal{T}^{\epsilon}(M_g)$, $\mathcal{T}^{\infty}(M_g)$, $\mathcal{T}(M_g)$ are all contractible. (This is also true for non-orientable surfaces of Euler characteristic $< 0$.)
\item [\bf{2.}] Let $M$ be a closed hyperbolic manifold. If $\dim M\geq 3$, Mostow’s Rigidity Theorem implies that $\mathcal{T}^{0}(M)=\star $ ; ie.  $\mathcal{T}^{0}(M)$ contains exactly one point. Therefore, $\mathcal{MET}^0(M)=\mathcal{D}_0(M)$. It also follows (see Remark (1) above) that $\mathcal{T}^{0}(M)$ is contractible when $\dim M\geq 2$.
\item [\bf{3.}] In dimensions two and three it is known that $\mathcal{D}_0(M)$ (and hence $\mathcal{MET}^0(M)$) is contractible. (This is due to Earle and Eells \cite{EE69} in dimension two and to Gabai \cite{Gab01} in dimension three.) This is certainly false in dimensions $\geq 11$, because $\pi_0(\mathcal{D}_0(M))$ is not finitely generated (see \cite[Cor.10.16 and 10.28]{FJ89}), and F. T. Farrell and P. Ontaneda \cite{FO09} conjectured that $\mathcal{D}_0(M)$ is also not contractible for dimension $n$, $5\leq n \leq 10$.
 \end{itemize}}
 \end{remark}
\begin{lemma}\label{freediff}\cite{FO09}
If $M$ is aspherical and the center of $\pi_1M$ is trivial, then the action of $DIFF_0(M)$ and $\mathcal{D}_0(M)$ on $\mathcal{MET}(M)$ is free.
\end{lemma}
\begin{proof}
Let $g\in \mathcal{MET}(M)$. Note that the isotropy group $H =\{\phi \in DIFF_0(M), \phi g = g\}$ of the action of $DIFF_0(M)$ at $g$ is $Isom_0(M, g)$, the group of all isometries of the Riemannian manifold $(M, g)$ that are homotopic to the identity. Hence this isotropy group $H$ is compact.
Let $\gamma : DIFF(M)\to \rm{Out}(\pi_1M)$ be the homomorphism induced by $\phi \to \phi_{*}$. Borel-Conner Raymond showed (see \cite[p. 43]{CR77}) that under the assumptions above, $\gamma$ restricted to compact subgroups is monic. But $\gamma(H)$ is trivial, since every element in $DIFF_0(M)$ is, by definition, homotopic to the identity. It follows that $H$ is trivial. Hence the action of $DIFF_0(M)$ is free. Therefore, the action of $\mathcal{D}_0(M)$ is also free. This proves the Lemma.
\end{proof}
\begin{remark}\rm{
\indent
\begin{itemize}
\item[\bf{1.}] Let $M$ be a hyperbolic manifold. Then the action of $\mathcal{D}_0(M)$ on $\mathcal{MET}(M)$ is free. Since $\mathcal{MET}(M)$ is contractible by Ebin’s Slice Theorem \cite{Ebi68} we have that $\mathcal{D}_0(M)\to \mathcal{MET}(M)\to \mathcal{T}(M)$ is a principal $\mathcal{D}_0(M)$-bundle and $\mathcal{T}(M)$ is the classifying space $B\mathcal{D}_0(M)$ of $\mathcal{D}_0(M)$. Since $\mathcal{T}(M)$ is homotopy equivalent to $\mathcal{MET}(M)/ DIFF0(M)$ we can also write $B(DIFF_0(M))=\mathcal{T}(M)$.
\item [\bf{2.}] Therefore, if $M$ is a closed hyperbolic manifold then $\mathcal{MET}^{\epsilon}(M)$ interpolates between $\mathcal{MET}^0(M)$ (which is equal to $\mathcal{D}_0(M)$) and $\mathcal{MET}(M)$ (which is contractible). Likewise $\mathcal{T}^{\epsilon}(M)$ interpolates between $\mathcal{T}(M)$ (which is equal to $B\mathcal{D}_0(M)$) and $\mathcal{T}^{0}(M)$ (which is contractible). Schematically, we have the following diagram:
\begin{equation}\label{diagtech}
\begin{array}[c]{ccccccc}
\mathcal{MET}^{0}(M)&\hookrightarrow&\mathcal{MET}^{\epsilon}(M)&\hookrightarrow&\mathcal{MET}^{\infty}(M)&\hookrightarrow&\mathcal{MET}(M)\\
\downarrow&&\downarrow&&\downarrow&&\downarrow\\
\mathcal{T}^{0}(M)&\hookrightarrow&\mathcal{T}^{\epsilon}(M)&\hookrightarrow&\mathcal{T}^{\infty}(M)&\hookrightarrow&\mathcal{T}(M)
\end{array}
\end{equation}
All vertical arrows represent quotient maps by the action of the group $\mathcal{D}_0(M)$.
The main result of Farrell and Ontaneda \cite{FO09} states that for a hyperbolic manifold the last two horizontal arrows of the lower row of the diagram above are not in general homotopic to a constant map. In particular, $\mathcal{T}^{\epsilon}$ $0\leq \epsilon \leq \infty$ is in general not contractible. More specifically, Farrell and Ontaneda \cite{FO09} proved that under certain conditions on the dimension n of the hyperbolic manifold $M$, the manifold $M$ has a finite cover $N$ (which depends on ) such that $\pi_k(\mathcal{T}^{\epsilon}(N))\to \pi_k(\mathcal{T}(N))$ is non-zero. In particular, $\mathcal{T}^{\epsilon}(N)$ is not contractible. The requirements on the dimension $n$ are implied by one of the following conditions: $n$ is larger than some constant $n_0(4)$ or $n$ is larger than 5 but in this last case we need that $\Theta_{n+1}\neq 0$. Here is a more detailed statement of Farrell and Ontaneda \cite{FO09} main result :
\end{itemize}}
\end{remark}
\begin{theorem}\label{techmu}
For every integer $k_0\geq 1$ there is an integer $n_0 = n_0(k_0)$ such that the following holds. Given $\epsilonǫ> 0$ and a closed real hyperbolic n-manifold $M$ with $n\geq n_0$, there is a finite sheeted cover $N$ of $M$ such that, for every $1\leq k \leq k_0$ with $n + k\equiv 2~\rm{mod}~ 4$, the map $\pi_k(\mathcal{T}^{\epsilon}(N))\to \pi_k(\mathcal{T}(N))$, induced by the inclusion $\mathcal{T}^{\epsilon}(N)\hookrightarrow \mathcal{T}(N) $, is non-zero. Consequently $\pi_k(\mathcal{T}^{\epsilon}(N))\neq 0$. In particular, $\mathcal{T}^{\delta}(N)$ is not contractible, for every $\delta$ such that ǫ$\epsilon \leq \delta \leq \infty$ (provided $k_0\geq 4$).
\end{theorem}
Here (and in the Corollary below) we consider the given hyperbolic metric as the basepoint for $\mathcal{T}^{\epsilon}(N)$, $\mathcal{T}(N)$. As a Corollary of (proof of the) Theorem \ref{techmu} we get :
\begin{corollary}\label{corrtechmu}\cite{FO09}
Let $M$ be a closed real hyperbolic manifold of dimension $n$, $n\geq 6$. Assume that $\Theta_{n+1}\neq 0$. Then for every $\epsilonǫ> 0$ there is a finite sheeted cover $N$ of $M$ such that $\pi_1(\mathcal{T}^{\epsilon}(N))\neq 0$. Therefore, $\mathcal{T}^{\epsilon}(N)$ is not contractible.
\end{corollary}
\begin{remark}\rm{
\indent
 \begin{itemize}
\item[\bf{1.}] Since $\mathcal{MET}(M)$ is contractible, Theorem \ref{techmu} implies that, for a general hyperbolic manifold $M$, the map $\pi_k(\mathcal{MET}^{\epsilon}(M))\to \pi_k(\mathcal{T}^{\epsilon}(M))$, induced by the second vertical arrow of the diagram, is not onto for some $k$.
\item[\bf{2.}] By Remark \ref{remtech} (1), the lower row of the diagram above is homotopically trivial in dimension 2. In dimension 3 one could ask the same: is the lower row of the diagram above homotopically trivial in dimension 3?. In view of a result of Gabai (see \cite{Gab01}), this is equivalent to asking: is $\mathcal{T}^{\infty}(M)$ contractible?.
\item [\bf{3.}] Let $M$ be a hyperbolic manifold. Consider the upper row of the diagram (\ref{diagtech}). It follows from a result of Ye on the Ricci flow (see \cite{Ye93}) that, provided the dimension of $M$ is even, there is an ǫ$\epsilon_0 = \epsilon_0(M) > 0$ such that for all ǫ$\epsilon\leq \epsilon_0$ the inclusion map $\mathcal{MET}^{\epsilon}\to \mathcal{MET}^{\infty}$ is $\mathcal{D}_0(M)$-equivariantly homotopic to a retraction $\mathcal{MET}^{\epsilon}\to \mathcal{MET}^{0}(M)\subset  \mathcal{MET}^{\infty}$. This has the following consequences. First the retraction above descends to a retraction $\mathcal{T}^{\epsilon}(M)\to \mathcal{T}^{0}(M)$, hence the inclusion map $\mathcal{T}^{\epsilon}(M)\to \mathcal{T}^{\infty}(M)$ is homotopic to a constant map (provided $\epsilon\leq \epsilon(M)$), and hence induces the zero homomorphism $\pi_k(\mathcal{T}^{\epsilon}(M))\to \pi_k(\mathcal{T}(M))$ for all $k$. Second, the inclusion map $\mathcal{D}_0(M)=\mathcal{MET}^0(M)\to \mathcal{MET}^{\epsilon}(M)$ 
induces monomorphisms $\pi_k(\mathcal{D}_0(M))=\pi_k(\mathcal{MET}^0(M))\to \pi_k(\mathcal{MET}^{\epsilon}(M))$, provided $\epsilon \leq \epsilon(M)$. Theorem \ref{techmu} then shows that in many cases ǫ$\epsilon_0(M) < \infty$.
\item [\bf{4.}] Let $M$ be a hyperbolic manifold. Since $DIFF(M)/DIFF_0(M)\cong \rm{Out}(\pi_1(M))$ we have
that $\mathcal{M}(M)\cong \mathcal{T}(M)/\rm{Out}(\pi_1(M))$ or, in general, $\mathcal{M}^{\epsilon}(M)\cong \mathcal{T}^{\epsilon}(M)/\rm{Out}(\pi_1(M)$. Note that $\rm{Out}(\pi_1(M))$ is a finite group, provided $\dim M \geq 3$.
\end{itemize}}
\end{remark}
Recall that Smooth bundles over a space $X$, with fiber $M$, modulo smooth equivalence, are classified by $[X, B(DIFF(M))]$, the set of homotopy classes of (continuous) maps from $X$ to the classifying space $B(DIFF(M))$. If we assume that $X$ is simply connected, then we obtain a reduction in the structural group of these bundles: smooth bundles over a simply connected space $X$, with fiber $M$, modulo smooth equivalence, are classified by $[X,B(DIFF_0(M))]=[X,\mathcal{T}(M)]$. Also, bundles with negatively curved fibers over a (simply connected) space $X$, modulo negatively curved equivalence, are classified by $[X, \mathcal{T}^{sec<0}(M)]$. And the inclusion map $F: \mathcal{T}^{sec<0}(M)\to \mathcal{T}(M)$ gives a relationship between the two bundle theories: $$F_X: [X, \mathcal{T}^{sec<0}(M)]\to[X, \mathcal{T}(M)]$$ and the map $F_X$ is the “forget the negatively curved structure” map. The “kernel”
$\mathcal{K}_X$ of this map between the two bundle theories is given by bundles over $X$, with negatively curved fibers, that are smoothly trivial. Every bundle in $\mathcal{K}_X$ can be represented by the choice of a negatively curved metric on each fiber of the trivial bundle $X\times M$, that is, by a map $X\mapsto \mathcal{MET}^{sec<0}(M)$. Note that this representation is not unique, because smoothly equivalent representations give rise to the same
bundle with negatively curved fibers. In any case, we have that $\mathcal{K}_X$ is the image of $[X,\mathcal{MET}^{sec<0}(M)]$ by the map 
$[X, \mathcal{MET}^{sec<0}(M)]\to [X, \mathcal{T}^{sec<0}(M)]$, induced by the quotient map $\mathcal{MET}^{sec<0}(M)\mapsto \mathcal{T}^{sec<0}(M)$. Note that we can think of $[X, \mathcal{MET}^{sec<0}(M)]$ as a bundle theory in the same way as $[X, \mathcal{T}^{sec<0}(M)]$ is a bundle theory. Summarizing, we get the following exact sequence of bundle theories:
\begin{equation}\label{forget}
[X, \mathcal{MET}^{sec<0}(M)]\stackrel{R_X}{\longrightarrow}[X, \mathcal{T}^{sec<0}(M)]\stackrel{F_X}{\longrightarrow}[X, \mathcal{T}(M)]
\end{equation}
where the map $R_X$ is the “representation map”: for $E\in \mathcal{K}_X$, $R^{-1}_X(E)$ is the set of representations of $E$ of the form $\phi : X\to \mathcal{MET}^{sec<0}(M)$, i.e. “bundles” of the form $(X\times M, id)$ where the Riemannian metric on ${x}\times M$ is $\phi(x)$.
The following question is asked by Farrell and Ontaneda \cite{FO10a}:
\begin{question}
 Are $F_X$ and $R_X$ non-constant, one to one or onto ?
\end{question}
If, in Equation \ref{forget}, we specify $X=\mathbb{S}^k$, $k > 1$ (recall we are using basepoint preserving maps), we obtain $\pi_k(\mathcal{MET}^{sec<0}(M))\mapsto \pi_k(\mathcal{T}^{sec<0}(M))\mapsto \pi_k(\mathcal{T}(M))$. Some information about these maps between homotopy groups was given in Theorem \ref{moduli}, Addendum \ref{moduliadd}, Theorem \ref{techmu} and Corollary \ref{corrtechmu}:
\begin{remark}\rm{
\indent
\begin{itemize}
\item [\bf{1.}] It was proved in Theorem \ref{moduli} that $\pi_2(\mathcal{MET}^{sec<0}(M))$ is never trivial, provided \\
$\mathcal{MET}^{sec<0}(M)\not=\emptyset$ and $\dim M > 13$. But the nonzero elements in \\
$\pi_2(\mathcal{MET}^{sec<0}(M))$, constructed in Theorem \ref{moduli}, are mapped to zero by the map $\pi_2(\mathcal{MET}^{sec<0}(M))\mapsto \pi_2(\mathcal{T}^{sec<0}(M))$. Therefore, the representation map $R_{\mathbb{S}^2}$ in Equation \ref{forget} is never one-to-one, provided $\mathcal{MET}^{sec<0}(M)\not=\emptyset$ and $\dim M >13$.
\item[\bf{2.}] It was also proved in Addendum \ref{moduliadd} (assuming  $\mathcal{MET}^{sec<0}(M)\not=\emptyset$)\\
that $\pi_2(\mathcal{MET}^{sec<0}(M))$ contains the infinite sum $(\mathbb{Z}_3)^{\infty}$ as a subgroup, thus \\$\pi_2(\mathcal{MET}^{sec<0}(M))$ is not finitely generated. Moreover, it was proved that the same is true for $\pi_k(\mathcal{MET}^{sec<0}(M))$, for $k=2p-4$, $p > 2$ prime (with $(\mathbb{Z}_3)^{\infty}$ instead of $(\mathbb{Z}_3)^{\infty}$), provided $\dim M$ is large (how large depending on $k$). Furthermore, $\pi_1(\mathcal{MET}^{sec<0}(M))$ contains $(\mathbb{Z}_3)^{\infty}$, provided $\dim M > 11$. And all these elements constructed in Addendum \ref{moduliadd} map to zero in the corresponding homotopy group of $\mathcal{T}^{sec<0}(M)$.
\item[\bf{3}] The result mentioned in the remark (2) about $\pi_1(\mathcal{MET}^{sec<0}(M))$ also proves that the forget structure map $F_{\mathbb{S}^2}$ is not onto. To see this just glue two copies of $\mathbb{D}^2\times M$ along $\mathbb{S}^1$ using an element $\alpha\in \pi_1(DIFF_0(M))$ which maps to one of the non-trivial elements $\beta$ in $\pi_1(\mathcal{MET}^{sec<0}(M))$ constructed in Theorem  \ref{moduli}; $\alpha$ exists because of the homotopy exact sequence for the bundle $DIFF_0(M)\mapsto \mathcal{MET}^{sec<0}(M)\mapsto \mathcal{T}^{sec<0}(M)$ and the fact that $\beta$ maps to zero
in $\pi_1(\mathcal{T}^{sec<0}(M))$ (see remark (2) above). Thus, there are (nontrivial) smooth bundles $E$ over $\mathbb{S}^2$ which do not admit a collection of negatively curved Riemannian metrics on the fibers of $E$. Using the remark, the same is true for $\mathbb{S}^k$, $k = 2p-3$, $p > 2$.
\item[\bf{4.}] It was proved in Theorem \ref{techmu} that there are examples of closed hyperbolic manifolds for which $\pi_k(\mathcal{T}^{sec<0}(M))$ is nonzero. Here $M$ depends on $k$ and always $k > 0$. In Theorem \ref{techmu} no conclusion was reached on the case $k = 0$ (i.e. about the connectedness of $\mathcal{T}^{sec<0}(M))$. Also, the images of these elements by the inclusion map $\mathcal{T}^{sec<0}(M))\mapsto \mathcal{T}(M))$ are not zero. Hence the forget structure map $F_{\mathbb{S}^k}$ is, in general, not trivial. This means also that there are bundles with negatively curved fibers that are not smoothly trivial, i.e. the representation map $R_{\mathbb{S}^k}$ is not onto in these cases.
\item[\bf{5.}] In all the discussion above we can replace “negatively curved metrics” by “$\epsilon$-pinched negatively curved metrics” \cite{FO10a}. And also Farrell and Ontaneda proved the following result \cite{FO10a}:
\end{itemize}}
\end{remark}
\begin{theorem}\label{hotech}
The forget structure map $F_{\mathbb{S}^k}: \pi_k(\mathcal{T}^{sec<0}(M))\to \pi_k(\mathcal{T}(M))$ is, in general, not one-to-one, for $k=2p-4$, $p$ prime. Consequently $\mathcal{T}^{sec<0}(M)$ is, in general, not connected.
\end{theorem}
\begin{remark}\rm{
 F.T. Farrell and P. Ontaneda proved a version of Theorem \ref{hotech} with\\
 $\mathcal{M}^{sec<0}(M)$ instead of  $\mathcal{T}^{sec<0}(M)$ \cite{FO10b}:}
\end{remark}
\begin{theorem}\label{homodu}
Let $M$ be a closed non-arithmetic hyperbolic manifold and $k$ a nonnegative integer, with $(k, \dim M=n)$ satisfying the following condition $(\star)$.
\[
 \star =
  \begin{cases}
   1. &  k=0 ~~\text{and}~~ n \geq 10 \\
   2. &  k=1 ~~\text{and}~~ n \geq 12 \\
   3. &  k=2p-4, ~p>2 ~~\text{prime},~~ \text{and} ~~n \geq 3k+8 
  \end{cases}
\]
Then $M$ has a finite sheeted cover $N$ such that the maps
\begin{align*}
\pi_k(\mathcal{MET}^{sec<0}(N))&\stackrel{\pi_k(\kappa)}{\longrightarrow}\pi_k(\mathcal{M}^{sec<0}(N))\\
\widetilde{H}_k(\mathcal{MET}^{sec<0}(N))&\stackrel{\widetilde{H}_k(\kappa)}{\longrightarrow}\widetilde{H}_k(\mathcal{M}^{sec<0}(N))
\end{align*}
are non-zero. In particular $\pi_k(\mathcal{M}^{sec<0}(N))$ and $\widetilde{H}_k(\mathcal{M}^{sec<0}(N))$ are nontrivial.
\end{theorem}
\begin{remark}\rm{
\indent
\begin{itemize}
\item[\bf{1.}] The statements of the Theorem \ref{homodu} holds also for $\epsilon$-pinched negatively curved metrics \cite{FO10b}.
\item[\bf{2.}] By Theorem \ref{hotech}, for a closed hyperbolic manifold $M$ there are non-zero elements in $\pi_k(\mathcal{MET}^{sec<0}(M))$ that survive in $\pi_k(\mathcal{T}^{sec<0}(M))$, provided $k$ and $n$ satisfy $(\star)$ and $M$ has a closed geodesic with large enough tubular neighborhood. The main result of \cite{FO10b} shows that, assuming $M$ has a $k$-good geodesic, these non-zero elements can be chosen so that they survive all the way to $\pi_k(\mathcal{MET}^{sec<0}(M))$.
\end{itemize}}
\end{remark}
\section{\large Final Remarks and Open Problems}
In this section, we review many interesting open problems along the above direction.\\

The negatively curved Riemannian symmetric spaces are of 4 types: $\mathbb{R}\textbf{H}^m$, $\mathbb{C}\textbf{H}^m$, $\mathbb{H}\textbf{H}^m$ and $\mathbb{O}\textbf{H}^2$. The following question is asked by C.S. Aravinda and F.T. Farrell \cite{AF04}:
\begin{question}\label{finerque}
 For each division algebra $\mathbb{K}$ over the real numbers and each integer $n\geq 2$ ($n = 2$ when $\mathbb{K}=\mathbb{O})$, does there exist a closed negatively curved Riemannian manifold $M^{dn}$(where $d=\dim_{\mathbb{R}} \mathbb{K})$ which is homeomorphic but not CAT(Diff or PL)-isomorphic to a $\mathbb{K}$-hyperbolic manifold.
 \end{question}
For $\mathbb{K}=\mathbb{R}$ and $n= 2, 3,$ this is impossible since homeomorphism implies diffeomorphism in these dimensions \cite{Moi52}; but this is the only known constraint on this question. The answer to Question \ref{finerque} is yes for $\mathbb{K}=\mathbb{O}$ by Theorem \ref{cay.hyp} since only one dimension needs to be considered in this case. When $\mathbb{K}=\mathbb{R}$, the answer is yes provided $n\geq 6$ by Theorem \ref{hype.smooth}. When $\mathbb{K}=\mathbb{C}$, the answer is yes for $n=4m + 1$ for any integer $m\geq 1$ and for $n=4$ by Theorem \ref{complex.hyp}, for $n=7$ and 8 by Theorem \ref{ramesh}. When $\mathbb{K}=\mathbb{H}$, the answer is yes for  $n = 2$, $4$ and $5$ by Theorem \ref{quot.hyp}.\\

While his visit to IIT Bombay and TIFR CAM, India in 2012, F.T. Farrell mentioned the following open problem :
\begin{problem}
Suppose that $M^n$ and $N^n$ are homotopy equivalent closed smooth manifolds and $M^n$ admits a negatively curved Riemannian metric. Does $N^n$ admit a negatively curved Riemannian metric?
\end{problem}
\begin{problem}(Farrell-Jones, 1998)
Let $M^n$ and $N^n$ be complete compact flat affine manifolds with $\pi_1(M)\cong \pi_1(N)$. Are they always diffeomorphic?.
\end{problem}
\begin{problem}(Gabai)
Let $f:N^n\to M^n$ be a harmonic homeomorphism between closed negatively curved manifolds. Must $f$ be a diffeomorphic?.
\end{problem}
\begin{problem}
Let $M^n$ and $N^n$ be closed negatively curved Riemannian manifolds with isomorphic marked length spectra. Must they be diffeomorphic?.
\end{problem}
\begin{problem}(Farrell-Jones, 1994)
\begin{itemize}
\item[(i)]Find $\Sigma^{2n}\in \Theta_{2n}$ if $\mathbb{C}\textbf{P}^n\#\Sigma^{2n}$ is diffeomorphic to $\mathbb{C}\textbf{P}^n$. (\cite{FJ94})
\item[(ii)]Find $\Sigma^{4n}\in \Theta_{4n}$ if $\mathbb{H}\textbf{P}^n\#\Sigma^{4n}$ is diffeomorphic to $\mathbb{H}\textbf{P}^n$. (\cite{FJ94})
\end{itemize}
\end{problem}
\begin{problem}
Let $M^n$ be a negatively curved Riemannian manifold. Is $\pi_1(M^n)$ residually finite ?
\end{problem}
\begin{problem}
Let $X$ be a finite aspherical simplicial complex. Does there exist a complete negatively curved manifold $M$ such that $\pi_1(M)\cong \pi_1(X)$ ?
\end{problem}
Boris Okun \cite{Oku01} has provided sufficient conditions for establishing non-zero degree of the tangential map (see Theorem \ref{equalrank}). Jean-Francois Lafont and Ranja Roy \cite{LR07} asked the following question :
\begin{question}
Are there examples where Okun's tangential map has zero degree? In particular, if one has a locally symmetric space modelled on $SL(n,\mathbb{R})/SO(n)$, does the tangential map to the dual $SU(n)/SO(n)$ have non-zero degree? 
\end{question}
Of course, the interest in the special case of $SL(n,\mathbb{R})/SO(n)$ is due to the universality of this example: every other locally symmetric space of non-positive curvature isometrically embeds in a space modelled on $SL(n,\mathbb{R})/SO(n)$. Now note that while the relationship between the cohomologies of closed locally symmetric space $M^n$ and its dual $M_U$ (with real coefficients) is well understood (and has been much studied) since the work of Matsushima \cite{Mat62}, virtually nothing is known about the relationship between the cohomologies with other coefficients. Jean-Francois Lafont and Ranja Roy \cite{LR07} asked the following :
\begin{question}
If $t : M^n\to M_U$ is the tangential map, what can one say about the induced map $t^* : H^*(M_U, \mathbb{Z}_p)\to  H^*(M, \mathbb{Z}_p)$?
\end{question}

\end{document}